\documentclass[12pt]{article}
\usepackage{bm,mathtools}
\usepackage{amsfonts}
\usepackage[a4paper, margin=2cm]{geometry}
  \usepackage{enumitem}
 \usepackage{multirow}
\usepackage{amsmath}

\usepackage{booktabs}
\usepackage{graphicx}
\usepackage{subcaption}
\usepackage{siunitx}

 \usepackage{mathrsfs}
  \newtheorem{theorem}{Theorem}[section]
 
 \newtheorem{proposition}[theorem]{Proposition}
 \newtheorem{corollary}[theorem]{Corollary}

 \newenvironment{proof}{{\it Proof}. }{\hfill\END\\[0.5ex]}

\DeclareMathOperator{\vcurl}{\overrightarrow{\mathrm{curl}}}
\DeclareMathOperator{\curl}{curl}
\DeclareMathOperator{\divv}{div}

\DeclareMathOperator{\SL}{SL}
\DeclareMathOperator{\DL}{DL} 

\newcommand{\D}{\mathrm{D}}
\DeclareMathOperator{\V}{V}
\DeclareMathOperator{\W}{W}
\DeclareMathOperator{\K}{K}
\DeclareMathOperator{\HH}{H}
\DeclareMathOperator{\J}{J}
\DeclareMathOperator{\I}{I}
\DeclareMathOperator{\Hes}{Hes}

\newcommand{\nnn}{\bm{n}}
\newcommand{\ttt}{\bm{t}}
\usepackage{cancel}

\usepackage{tikz}

\usetikzlibrary{arrows,decorations.markings,arrows.meta}
\usetikzlibrary{calc}

\usepackage{pgfplots}
\usepackage{makecell}

\usepackage{amssymb}

\usepackage[normalem]{ulem}

\newcommand{\be}{\begin{equation}}
\newcommand{\ee}{\end{equation}}

\newcommand{\x}{{\bm{x}}}
\newcommand{\y}{{\bm{y}}}

 \makeatletter\@addtoreset{equation}{section}\makeatother

\newcommand{\END}{\hfill$\Box$}

\usepackage{fancyhdr}
\title{Boundary integral equation based solvers for elastic transmission problems in complex media}
\date{}
\author{V\'{\i}ctor Dom\'{\i}nguez\thanks{Dep. Ingenier\'{i}a Matem\'atica e Inform\'atica, Universidad P\'{u}blica de Navarra. Campus de Tudela 31500 - Tudela, Spain, / Institute for Advanced Materials and Mathematics (INAMAT$^2$), Campus de Arrosadia
31006,  Pamplona, Spain. e-mail: victor.dominguez@unavarra.es }\and Catalin Turc\thanks{  Department of
Mathematical Sciences, New Jersey  Institute of Technology,
Univ. Heights. 323 Dr. M. L. King Jr. Blvd, Newark, NJ 07102, USA, e-mail: catalin.c.turc@njit.edu.}}
\allowdisplaybreaks

\begin{document}
\maketitle
\begin{abstract}
We present a boundary integral equation (BIE) framework for two-dimensional time-harmonic elastic transmission problems. Using the Helmholtz decomposition, the elastodynamic displacement fields are expressed in terms of scalar pressure and shear potentials, so that only the four classical Helmholtz boundary integral operators are needed. The key step is a rewriting of the Navier traction operator in the local normal--tangential frame of the interface, in which second-order normal derivatives are eliminated in favour of tangential derivatives and curvature terms. This leads to a one-parameter family of indirect combined field integral equations, depending on a coupling constant $\alpha\ne 0$, which we prove to be uniquely solvable. The principal part of the resulting operator is, however, defective, and the discretized systems are severely ill-conditioned. Using the pseudodifferential calculus of the Helmholtz BIOs, we construct explicit left and right analytical preconditioners and prove that the preconditioned operator is a compact perturbation of the identity. Numerical experiments based on high-order Nystr\"om discretizations confirm that the preconditioner clusters the spectrum: the number of GMRES iterations becomes independent of the discretization and is reduced by one order of magnitude or more, also in the high-frequency regime and for non-convex scatterers.
\end{abstract}


\section{Introduction}

Elastic transmission problems in complex piecewise homogeneous media play a fundamental role in various engineering and scientific applications, ranging from seismology and geophysics to non-destructive testing and medical imaging. Designing accurate and efficient numerical solvers for such wave propagation problems remains a computational challenge, particularly in the presence of complex, non-convex interfaces of material discontinuity. Boundary Integral Equation (BIE) methods are naturally suited for these problems, as they reduce the dimensionality of the domain and rigorously enforce radiation conditions at infinity. 

Standard BIE formulations for elastodynamics rely heavily on Navier boundary integral operators (BIOs). However, owing to the more complicated tensorial nature and stronger singularities of the Navier Green's tensor, the numerical implementation of high-order discretizations for these operators is significantly more challenging and computationally expensive than for their scalar acoustic counterparts~\cite{chapko2000numerical,DomCat:2022,DoSaSa:2015}; see also~\cite{LELOUER20141,chaillat2017fast} for 3D discretization algorithms. To bypass the use of Navier BIOs altogether, an attractive alternative is to seek the elastic fields in their Helmholtz decomposition form; see~\cite{lai2019framework,dominguez2022robust} for single layer and combined layer representations, and~\cite{dong2022spectral} for their 3D counterparts. By doing so, the displacement fields are represented via scalar potentials corresponding to the pressure (longitudinal) and shear (transversal) wavenumbers, allowing the scattering problem to be solved using exclusively scalar Helmholtz layer potentials~\cite{dominguez2022nystrom}.

While the continuity of Dirichlet traces translates seamlessly into the Helmholtz decomposition framework, the enforcement of Neumann traces (the traction operator) is notoriously difficult. The traction operator applied to the Helmholtz decompositions naturally yields second-order normal derivatives of the potentials (Hessian traces), which are not available from the Cauchy data. In this paper, we introduce a simple approach to overcome this bottleneck: we recast the normal and tangential Hessian traces in the local orthogonal $(\nnn,\ttt)$ frame of the boundary. By exploiting the Helmholtz equation, specifically the formulation of the Laplacian in the local frame, we replace the problematic second-order normal derivatives with local tangential derivatives and curvature terms acting directly on the Cauchy data. This advancement allows for a straightforward and robust integration of the four standard Helmholtz BIOs of the Calder\'on calculus into the Navier transmission problem.
Building upon these local trace operators, we propose a one-parameter family of indirect combined field integral equation (CFIE) formulations for two-dimensional elastic transmission problems, depending on a coupling constant $\alpha\neq0$, which we refer to as the $(\nnn,\ttt)$-CFIE$_\alpha$. They are similar in nature to the classical CFIE for Helmholtz problems \cite{costabel-stephan,kress1978transmission} and to its extension to Helmholtz transmission problems in \cite{dominguez2016well}. We prove that the $(\nnn,\ttt)$-CFIE$_\alpha$ is uniquely solvable for every $\alpha\neq0$, every density ratio and every admissible pair of Lam\'e parameters. However, a pseudodifferential analysis of the resulting block operator reveals a serious drawback: the formulation is defective. The principal part of the operator has an infinite-dimensional kernel and cokernel, so that the operator is not Fredholm, let alone a compact perturbation of the identity. This is not unexpected: the same phenomenon was observed in \cite{lai2019framework,dominguez2022robust} for Helmholtz-decomposition formulations of Dirichlet and Neumann problems for the Navier equation. Upon discretization, the defectiveness translates into severely ill-conditioned linear systems and into Krylov iteration counts (e.g.\ for GMRES) that grow with the number of unknowns, more so as the frequency increases.

The main theoretical contribution of this paper is the construction of explicit left and right analytical preconditioners for the $(\nnn,\ttt)$-CFIE$_\alpha$ that cure this defectiveness: we prove that the preconditioned operator is a compact perturbation of the identity. Operator and analytical preconditioning techniques have received considerable attention as a means to accelerate the iterative solution of wave propagation problems, from the scalar Helmholtz equation~\cite{antoine2005alternative,antoineB,bruno2012regularized,darbas2013analytical,turc2} to Maxwell~\cite{hiptmair2006operator} and elastodynamic~\cite{chaillat2020analytical} problems. In our framework preconditioning is not merely an acceleration device but a necessity, since the unpreconditioned formulation is not Fredholm. Furthermore, a thorough understanding of these preconditioned operators in the two-dimensional setting is a necessary stepping stone towards the three-dimensional case, where Krylov subspace solvers are virtually mandatory.

Beyond the specific formulations analyzed in this work, the decoupling strategy and local boundary differential calculus developed here are intended to lay the mathematical and algorithmic groundwork for broader integral equation paradigms on composite structures.
We are confident that the numerical algorithms can eventually be rigorously analyzed for the present configurations using techniques similar to those employed in \cite{dominguez2022nystrom} (see also \cite{dong2021highly}).

In the present work, we restrict our analysis to smooth domains to ensure the validity of the pseudo-differential expansions and the smoothness of geometric terms, such as the signed curvature. Nevertheless, we expect that corners can be successfully handled in practice by employing parameterizations graded towards the corners to regularize the boundary singularities \cite{Kress,KressCorner,dominguez2016well,helsing2008corner,helsing2011fast}. Additionally, recasting elastic transmission boundary conditions solely through Helmholtz-type boundary integral operators provides a viable stepping stone toward fully systematic global and local Multi-Trace Formulations (MTF) \cite{jerez-hanckes1,jerez-hanckes2,HJP13,jerez2017multitrace,CHAILLAT2026293}. By avoiding the direct inversion and singular integration of the full elastodynamic Green's tensors at complex geometric interfaces and multiple junctions, this approach offers an appealing modular framework that can be seamlessly incorporated into multi-trace Calder\'on projectors as well as generalized operator-preconditioned boundary element solvers.

Furthermore, while the scope of this paper is restricted to two-dimensional settings, a primary motivation behind this detailed study is to establish a rigorous mathematical and computational foundation for three-dimensional wave propagation problems. In 3D, the elastodynamic Helmholtz decomposition becomes inherently more intricate, involving a scalar potential for pressure waves and a divergence-free vector potential for shear waves. The mathematical structure of this vector potential shares deep similarities with the electromagnetic fields governed by Maxwell's equations, requiring complex vector trace algebras and surface tangential calculus. By thoroughly understanding the defectiveness of the CFIE formulation and successfully designing exact analytical preconditioners in 2D, we gain crucial insights that pave the way for extending these localized trace techniques and regularization strategies not only to the significantly more challenging 3D elastic transmission problems, but also to analogous BIE formulations in 3D electromagnetics.

The paper is organized as follows. In Section~\ref{setting}, we introduce the Navier equations in two dimensions, define the transmission problem, and present our novel Helmholtz decomposition approach featuring the localized Neumann trace formulas. In Section~\ref{GradHessCalculations}, we introduce the Helmholtz boundary layer operators and derive the indirect combined field formulation for the transmission problem. Through a rigorous principal symbol analysis, we expose the defective nature of this formulation and subsequently, in Section~\ref{sec:preconditioning}, we construct the analytical preconditioners to restore well-posedness. Finally, in Section~\ref{NR}, we provide extensive numerical evidence using high-order Nystr\"om discretizations~\cite{KressH,Kress} based on global trigonometric interpolation and splittings of the logarithmic, strongly singular, and hypersingular kernels.  Our numerical experiments, conducted on both convex and complex non-convex geometries in the high-frequency regime, demonstrate that our analytical preconditioning strategy clusters the spectrum of the discrete systems and makes the number of GMRES iterations essentially independent of the discretization.

\paragraph{Notation.} Throughout the paper the alphabet indicates the block structure of the objects involved. Roman letters denote operators acting on scalar functions on the interface $\Gamma$ (the Helmholtz boundary integral operators, the tangential derivative, the Hilbert transform, the mean value operator) and $H^s(\Gamma)$ denotes the usual Sobolev spaces. Calligraphic letters denote objects with two components: the product spaces $\mathcal H^s(\Gamma):=H^s(\Gamma)\times H^s(\Gamma)$ and the $2\times2$ block operators acting on them; for a scalar operator $\mathrm A$, the same letter in calligraphic font, $\mathcal A$, denotes its componentwise action. Fraktur letters denote $2\times2$ block operators whose blocks are themselves calligraphic operators, such as the transmission operator $\mathfrak A_\alpha$ and the preconditioners $\mathfrak P^{\rm left}$, $\mathfrak P^{\rm right}$ of Sections~\ref{GradHessCalculations} and~\ref{sec:preconditioning}. Vectors and vector fields are set in bold.

\section{Navier equations and Helmholtz decompositions}\label{setting}

\begin{figure}

\caption{Sketch of the domain of our problem and the $\{\nnn,\ttt\}$ frame.}
\begin{center}
\begin{tikzpicture}
    \coordinate (O) at (0,0);

    \begin{scope}[rotate=30]
        \path[draw,fill=gray!30] (O) ellipse (2cm and 1.6666cm);
    \coordinate (P) at ($(O)+(195:2cm and 1.666cm)$);
    \end{scope}


    \draw[->,black,thick] (P) -- ++(140:-1cm) node[below] {$\ttt$};;
    \draw[->,black,thick] (P) -- ++(230:1cm) node[right] {$\nnn$};
\node at (2.1,2.1) {$\Omega_1$};
    \node at (O) {$\Omega_2$};
    \node at (2.,-1.4) {$\Gamma$};
    \filldraw [black] (P) circle (2pt);
\end{tikzpicture}
\end{center}
\end{figure}
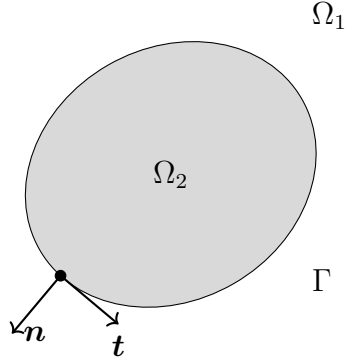

\subsection{Transmission problem}
For any vector field ${\bf u}=(u_1,u_2)^\top:\mathbb{R}^2\to\mathbb{R}^2$
(vectors in this paper will always be regarded as {\em column} vectors), the
strain tensor in a linear, isotropic, homogeneous elastic medium with Lam\'e
constants $\lambda$ and $\mu$ is defined by
\[
 \bm{\epsilon}({\bf u})
 :=
 \frac{1}{2}\left(\nabla {\bf u}+(\nabla {\bf u})^\top\right)
 =
 \begin{bmatrix}
  \partial_{x_1} u_1
  &
  \tfrac{1}{2}\left(\partial_{x_1} u_2+\partial_{x_2} u_1\right)
  \\
  \tfrac{1}{2}\left(\partial_{x_1} u_2+\partial_{x_2} u_1\right)
  &
  \partial_{x_2} u_2
 \end{bmatrix}.
\]
The stress tensor is then given by 
\[
 \bm{\sigma}_{\mu,\lambda}({\bf u})
 :=
 2\mu \bm{\epsilon}({\bf u})
 +
 \lambda (\divv{\bf u}) I_2,
\]
where $I_2$ is the identity matrix of order $2$. The Lam\'e coefficients are assumed to satisfy
\[
 \mu_j>0,
 \qquad
 \lambda_j+\mu_j>0,\qquad j=1,2,
\]
which is equivalent to the positive definiteness of the two-dimensional
isotropic elasticity tensor and covers all physically admissible materials.
The time-harmonic elastic wave equation, or Navier equation, is
\[
 \divv\bm{\sigma}_{\mu,\lambda}({\bf u})
 +
 \omega^2{\bf u}
 =
 \mu\Delta {\bf u}
 +
 (\lambda+\mu)\nabla(\divv {\bf u})
 +
 \omega^2 {\bf u}
 =
 0,
\]
where $\omega\in\mathbb{R}^+$ is the frequency and the divergence operator is
applied row-wise to matrix fields.

Let $\Omega_2\subset\mathbb{R}^2$ be a bounded domain with boundary $\Gamma$,
where $\Gamma$ is a closed curve, and let
\[
 \Omega_1:=\mathbb{R}^2\setminus\overline{\Omega_2}
\]
denote the corresponding unbounded exterior domain. We consider a
time-harmonic elastic transmission scattering problem in $\mathbb{R}^2$ for two
linear, isotropic, homogeneous media occupying the domains $\Omega_j$,
$j=1,2$.

The Lam\'e coefficients $(\mu_j,\lambda_j)$ used below are understood to be
density-normalized coefficients. More precisely, the exterior density is taken
as reference, while the density contrast in the interior medium $\Omega_2$ is represented
by a positive parameter $\rho$. With this convention, the Navier equation in
each medium is written as
\[
 \divv\bm{\sigma}_{\mu_j,\lambda_j}({\bf u}_j)
 +
 \omega^2{\bf u}_j
 =
 0
 \qquad \text{in } \Omega_j,
 \quad j=1,2.
\]
Thus, if the physical density in the exterior medium is normalized to one and
the physical density in the interior medium is $\rho$, then the physical Lam\'e
coefficients in the interior medium are $( 
 \rho\mu_2,
 \rho\lambda_2)$.
Equivalently, the density contrast has been absorbed into the interior Lam\'e
coefficients in the differential equation, and it will reappear only in the
transmission condition for the physical traction.
The Cauchy data on $\Gamma$ consist of Dirichlet data, namely the
trace of the displacement field, and Neumann data, expressed in terms
of the traction operator. For sufficiently smooth fields defined in
$\Omega_j$, the traction operator is given by
\begin{equation*}
  \begin{multlined}
  {\rm T}_j{\bf u}_j
 :=
 \bm{\sigma}_{\mu_j,\lambda_j}({\bf u}_j)\nnn
 =
 \lambda_j \divv{\bf u}_j\,\nnn
 +
 2\mu_j \partial_{\nnn}{\bf u}_j
 -
 \mu_j(\curl{\bf u}_j)\,\ttt
 \quad\text{on } \Gamma,
 \\
 j=1,2,
 \qquad
 \curl {\bf u}:=\partial_{x_1}u_2-\partial_{x_2}u_1.
  \end{multlined}
\end{equation*}
Here $\nnn$ denotes the unit normal vector field on $\Gamma$, pointing toward the unbounded exterior domain $\Omega_1$ and $\partial_{\nnn}{\bf u}_j$ is the normal derivative understood componentwise. Moreover, $\ttt$ denotes the positively oriented unit tangent field on $\Gamma$:
\[
 \ttt:=-\mathrm Q\nnn,
 \qquad
 \mathrm Q:=
 \begin{bmatrix}
 0 & 1\\
 -1 & 0
 \end{bmatrix}.
\]
With this convention, $\ttt$ is counterclockwise oriented. When needed, both
vector fields are smoothly extended to a neighborhood of $\Gamma$. 

We now state the transmission problem. The scattering process is generated by
an incident elastic field ${\bf u}^{\rm inc}$ impinging on the bounded inclusion
$\Omega_2$. The incident field is assumed to satisfy the Navier equation with
the exterior material parameters in a neighborhood of $\overline{\Omega_2}$.
This includes, in particular, the standard cases of incident plane waves and
point-source elastic waves whose source lies in $\Omega_1$, the complement of  $\overline{\Omega_2}$.

We seek a pair of fields $({\bf u}_1,{\bf u}_2)$ satisfying
\begin{subequations}\label{eq:NavD}
\begin{equation}\label{eq:NavDa}
  \divv\bm{\sigma}_{\mu_j,\lambda_j}({\bf u}_j)
  +\omega^2{\bf u}_j=0
  \quad\text{in } \Omega_j,\qquad j=1,2.
\end{equation}
The total displacement field is then defined by
\[
{\bf u}:=
\begin{cases}
  {\bf u}_1+{\bf u}^{\rm inc}, & \text{in } \Omega_1,\\
  {\bf u}_2, & \text{in } \Omega_2.
\end{cases}
\]
The scattered field ${\bf u}_1$ satisfies the Kupradze radiation condition at
infinity \cite{AmKaLe:2009,kupradze2012three}. On the interface $\Gamma$, the
total field ${\bf u}$ satisfies the transmission conditions
\begin{align}
  \gamma {\bf u}_2
  &=
  \gamma {\bf u}_1+\gamma {\bf u}^{\rm inc},
  \label{eq:NavDb}
  \\
  \rho\,{\rm T}_2{\bf u}_2
  &=
  {\rm T}_1{\bf u}_1+{\rm T}_1{\bf u}^{\rm inc}.
  \label{eq:NavDb2}
\end{align}
\end{subequations}
Here $\gamma$ denotes the trace operator, taken from the appropriate side of
$\Gamma$.

The transmission problem \eqref{eq:NavD} is well-posed: it admits a unique
solution for all frequencies $\omega>0$ and all admissible Lam\'e parameters
$\lambda_j,\mu_j$, $j=1,2$; see, for instance,
\cite{costabel-stephan-elasticity:1990}.

\subsection{Helmholtz decomposition approach}
Our approach seeks the fields ${\bf u}_j$ in the Helmholtz decomposition form, that is
\begin{equation}\label{eq:Hdecomp2}
{\bf u}_j=\nabla u_p^{(j)}+\vcurl  {u_s^{(j)}},\ j=1,2\qquad \text{  with  }\vcurl{\varphi}:=\mathrm{Q}\nabla{\varphi}
\end{equation}
where the scalar fields $u_p^{(j)}$ and $u_s^{(j)}$ are solutions of the Helmholtz equations in the domains $\Omega_j$, $j=1,2$, respectively, with wavenumbers $k_p^{(j)}$ and $k_s^{(j)}$  defined as
\begin{equation}\label{eq:ks:kp}
(k_p^{(j)})^2:=\frac{\omega^2}{\lambda_j+2 \mu_j}, \quad (k_s^{(j)})^2:=\frac{\omega^2}{\mu_j}.
\end{equation}
In addition,   the exterior scalar fields $u_p^{(1)}$ and $u_s^{(1)}$
satisfy the Sommerfeld radiation conditions at infinity, a consequence of the Kupradze radiation condition. The enforcement of the continuity of the Dirichlet traces in the Helmholtz decomposition approach is rather straightforward, as it amounts to taking the scalar product of the decomposition~\eqref{eq:Hdecomp2} with the orthogonal frame $(\ttt, \nnn)$ on $\Gamma$.  Hence, denoting
\[
 \partial_{\nnn} u := \gamma(\nnn\cdot \nabla u) =- \gamma(\ttt\cdot \vcurl u),\quad
 \partial_{\ttt} u := \gamma(\ttt\cdot \nabla u) = \gamma(\nnn \cdot \vcurl u),
\]
condition \eqref{eq:NavDb} leads to the following system of boundary conditions that couples the Cauchy data on $\Gamma$ of the fields $u_p^{(j)}$ and $u_s^{(j)}$
\begin{equation}\label{eq:DH}
-
\begin{bmatrix}
{\bf u}^{\rm inc}\cdot\nnn\\
{\bf u}^{\rm inc}\cdot\ttt
\end{bmatrix}
=
\begin{bmatrix}
\partial_{\nnn}u_p^{(1)}+\partial_{\ttt}u_s^{(1)}\\
\partial_{\ttt}u_p^{(1)}-\partial_{\nnn}u_s^{(1)}
\end{bmatrix}
-
\begin{bmatrix}
\partial_{\nnn}u_p^{(2)}+\partial_{\ttt}u_s^{(2)}\\
\partial_{\ttt}u_p^{(2)}-\partial_{\nnn}u_s^{(2)}
\end{bmatrix}
\qquad {\rm on}\ \Gamma .
\end{equation}

The same approach can be applied to the enforcement of the continuity of Neumann traces on the boundary  \eqref{eq:NavDb2}. First,  we notice that
\begin{eqnarray*}
  {\rm T}_j[\nabla u_p^{(j)}] &=& \left(2\mu_j \Hes u_p^{(j)}+\lambda_j \Delta u_p^{(j)}\mathrm{I}\right)\nnn=
  \left(2\mu_j \Hes u_p^{(j)}-\lambda_j (k_p^{(j)})^2u_p^{(j)}\mathrm{I}\right)\nnn\\
  {\rm T}_j[\vcurl u_s^{(j)}] &=& {\rm Q}\left(2\mu_j \Hes u_s^{(j)}-\mu_j  \Delta u_s^{(j)}\mathrm{I} \right)\nnn={\rm Q}\left(2\mu_j \Hes u_s^{(j)}+\mu_j (k_s^{(j)})^2u_s^{(j)}\mathrm{I}\right)\nnn,\ j=1,2
\end{eqnarray*}
where $\Hes:=\nabla\nabla^\top$ is the Hessian  matrix operator. These relations lead immediately to formulas for the normal and tangential components of the traction operator applied to the vector fields ${\bf u}_j$:
\begin{equation}\label{eq:NH}
\left|
\begin{array}{rclrcl}
  {\rm T}_j  {{\bf u}_j}\cdot\nnn&=&2\mu_j \nnn^\top \Hes u_p^{(j)}{\bm n}-\lambda_j (k_p^{(j)})^2 u_p^{(j)}&+&2\mu_j
 \ttt^\top \Hes u_s^{(j)}{\bm n},
 \\
  {\rm T}_j  {{\bf u}_j}\cdot\ttt &=&2\mu_j \ttt^\top \Hes u_p^{(j)}{\bm n} &-& 2\mu_j \nnn^\top \Hes u_s^{(j)}{\bm n}-\mu_j (k_s^{(j)})^2 u_s^{(j)}.
\end{array}
\right.
\end{equation}

The expressions in \eqref{eq:NH} involve second-order normal derivatives of the potentials, packaged in the Hessian matrices. Unlike the Dirichlet and Neumann traces, these quantities are not available on $\Gamma$ from the boundary data alone. The key observation is that, since $u_p^{(j)}$ and $u_s^{(j)}$ solve Helmholtz equations, all second-order normal derivatives can be eliminated in favour of tangential derivatives, curvature terms and the Cauchy data, so that only the classical BIOs are needed. We now carry out this reduction.
Indeed, we start from the formula of the Laplacian in the $(\ttt,\nnn)$ frame
\[
\Delta u=\partial_{\nnn}^2 u+\kappa\ \partial_{\nnn}u+\partial_{\ttt}^2 u
\]
where $\kappa$ denotes the {signed  curvature on $\Gamma$}. That is,
\[
 \kappa\circ{\bf x} = \frac{1}{|{\bf x}'|^3}  \left({\bf x}'\right)^\top\mathrm{Q} {\bf x}''
\]
where ${\bf x}$ is any smooth regular parameterization. This yields
\[
\nnn^\top \Hes u_p^{(j)}{\bm n} =\partial_{\nnn}^2 u_p^{(j)}=-\kappa \partial_{\nnn}u_p^{(j)}-\partial_{\ttt}^2 u_p^{(j)}-(k_p^{(j)})^2 u_p^{(j)}
\]
and
\[
\nnn^\top \Hes u_s^{(j)}{\bm n} =\partial_{\nnn}^2 u_s^{(j)}=-\kappa \partial_{\nnn}u_s^{(j)}-\partial_{\ttt}^2 u_s^{(j)}-(k_s^{(j)})^2 u_s^{(j)},\ j=1,2.
\]
Also, using the formula for the gradient in the $(\ttt,\nnn)$ frame, that is $\nabla u =\partial_{\nnn}u\ \nnn + \partial_{\ttt} u\ \ttt$,
together with the Frenet-Serret identity $\partial_{\ttt}\nnn= \kappa\ttt$, we obtain
\[
\ttt^\top \Hes u_p^{(j)}{\bm n} = \partial_{\ttt}\partial_{\nnn}u_p^{(j)}-\kappa \partial_{\ttt} u_p^{(j)}\qquad \ttt^\top \Hes u_s^{(j)}{\bm n} = \partial_{\ttt}\partial_{\nnn}u_s^{(j)}-\kappa \partial_{\ttt} u_s^{(j)},\ j=1,2.
\]
In conclusion,  we reformulate Neumann conditions \eqref{eq:NavDb2} as
\begin{equation}\label{eq:NHT}
\left|
\begin{array}{rclll}
  \rho{\rm T}_2{\bf u}_2\cdot \nnn&=&  {\rm T}_1{\bf u}_1\cdot\nnn&+&  {\rm T}_1{\bf u}^{\rm inc}\cdot \nnn,
\\
 \rho {\rm T}_2{\bf u}_2\cdot \ttt&=&  {\rm T}_1{\bf u}_1\cdot\ttt&+&  {\rm T}_1{\bf u}^{\rm inc}\cdot \ttt,
\end{array}
\right.
\end{equation}
where  the normal and tangential components of the traction tensor are expressed as 
\begin{equation}\label{eq:NH_local}
\left|
\begin{array}{rclrcl}
  {\rm T}_j {\bf u}_j\cdot\nnn &=& -\omega^2 u_p^{(j)} -2\mu_j(\kappa  \partial_{\nnn}u_p^{(j)}+ \partial_{\ttt}^2 u_p^{(j)}) & +& 2\mu_j(\partial_{\ttt}\partial_{\nnn}u_s^{(j)}-\kappa \partial_{\ttt} u_s^{(j)}),&
 \\
  {\rm T}_j {\bf u}_j\cdot\ttt &=&2\mu_j (\partial_{\ttt}\partial_{\nnn}u_p^{(j)}-\kappa \partial_{\ttt} u_p^{(j)})&+&\omega^2 u_s^{(j)} + 2\mu_j(\kappa\ \partial_{\nnn}u_s^{(j)}+\ \partial_{\ttt}^2 u_s^{(j)}).& 
\end{array}
\right.
\end{equation}
We notice that in deriving \eqref{eq:NH_local}, we have used \eqref{eq:ks:kp}, namely
\begin{equation}\label{eq:Hdecomp2b}
(\lambda_j+2\mu_j)(k_p^{(j)})^2=\omega^2,
\qquad
\mu_j(k_s^{(j)})^2=\omega^2
\end{equation}
to simplify the expressions.
Relations \eqref{eq:Hdecomp2}--\eqref{eq:Hdecomp2b} will be used repeatedly in the rest of the paper to simplify the expressions, especially in Section~\ref{sec:preconditioning}.

Summarizing, from \eqref{eq:NH_local} and \eqref{eq:NHT}, the Neumann transmission conditions for the Helmholtz decomposition components of the solution are written as
\begin{equation}\label{eq:TnU}
  \begin{aligned}
  -
\begin{bmatrix}
 {\rm T}_1{\bf u}^{\rm inc}\cdot \nnn \\
 {\rm T}_1{\bf u}^{\rm inc}\cdot \ttt 
\end{bmatrix}&
\ = \ \begin{bmatrix}
 -\omega^2 u_p^{(1)} -2\mu_1(\kappa \partial_{\nnn}u_p^{(1)}+ \partial_{\ttt}^2 u_p^{(1)}) & +& 2\mu_1(\partial_{\ttt}\partial_{\nnn}u_s^{(1)}-\kappa \partial_{\ttt} u_s^{(1)})
 \\
  2\mu_1 (\partial_{\ttt}\partial_{\nnn}u_p^{(1)}-\kappa \partial_{\ttt} u_p^{(1)})&+&\omega^2 u_s^{(1)} 
  + 2\mu_1(\kappa\ \partial_{\nnn}u_s^{(1)}+\ \partial_{\ttt}^2 u_s^{(1)})
\end{bmatrix}\\
-\rho& 
\begin{bmatrix}
 -\omega^2 u_p^{(2)} -2\mu_2(\kappa\ \partial_{\nnn}u_p^{(2)}+ \partial_{\ttt}^2 u_p^{(2)})& +& 2\mu_2(\partial_{\ttt}\partial_{\nnn}u_s^{(2)}-\kappa \partial_{\ttt} u_s^{(2)}) \
 \\
2\mu_2 (\partial_{\ttt}\partial_{\nnn}u_p^{(2)}-\kappa \partial_{\ttt} u_p^{(2)})&+&\omega^2 u_s^{(2)} + 2\mu_2(\kappa\ \partial_{\nnn}u_s^{(2)}+ \partial_{\ttt}^2 u_s^{(2)})
\end{bmatrix}.
\end{aligned}
\end{equation}


\section{BIE formulations of the Helmholtz decomposition approach via Helmholtz BIOs}\label{GradHessCalculations}

\subsection{Helmholtz boundary layer operators}
For a given wave-number $k$ and a density $\varphi$ on the boundary $\Gamma$ we define the Helmholtz single and double layer potentials in the form
\[
 \SL_k[\varphi](\x):=\int_\Gamma \phi_k(\x-\y)\varphi(\y) {{\rm d}\y},\quad
 \DL_k[\varphi](\x):=\int_\Gamma \frac{\partial \phi_k(\x-\y)}{\partial \nnn(\y)}\varphi(\y) {{\rm d}\y},\ \x\in\mathbb{R}^2\setminus\Gamma
\]
where $k>0$ and
\[
 \phi_k(\x) = \frac{i}4 H_0^{(1)}(k|\x|)
\] 
is the fundamental solution of the Helmholtz equation.

The four BIOs of the Calder\'on calculus associated with  the Helmholtz equation are defined by applying the exterior/interior Dirichlet and Neumann traces on $\Gamma$ (denoted in what follows by $\gamma^{(1)}/\gamma^{(2)}$ and $\partial_{\nnn}^{(1)}/\partial_{\nnn}^{(2)}$ respectively) to the Helmholtz single and double layer potentials. We have then the classical relations
\cite{hsiao2008boundary,mclean:2000,Saranen}
\begin{equation}\label{eq:traces}
\begin{aligned}
{\gamma^{(j)}}  \SL_k\varphi &=\mathrm{V}_{k}\varphi, \quad                               &{\partial_{\nnn}^{(j)}} \SL_k\varphi &=\frac{{(-1)^{j }}}{2}\varphi +\mathrm{K}_{k}^\top\varphi,  \\
{\gamma^{(j)}} \DL_k\varphi &= \frac{{(-1)^{j+1}}}{2}\varphi+\mathrm{K}_{k}\varphi, \quad &{\partial_{\nnn}^{(j)}}\DL_k\varphi &=\mathrm{W}_{k}\varphi
\end{aligned}
\end{equation}
where, for $\x\in\Gamma$,
\begin{eqnarray*}
( \mathrm{V}_{k}\varphi)(\x) &:=&\int_\Gamma \phi_k(\x-\y)\varphi(\y) {{\rm d}\y},\\
( \mathrm{K}_k\varphi)(\x) &:=&\int_\Gamma \frac{\partial \phi_k(\x-\y)}{\partial \nnn(\y)}\varphi(\y) {{\rm d}\y},\\
( \mathrm{K}_k^\top\varphi)(\x)  &:=& \int_\Gamma \frac{\partial \phi_k(\x-\y)}{\partial \nnn(\x)}\varphi(\y) {{\rm d}\y},\\
( \mathrm{W}_k\varphi)(\x) &:=& \mathrm{f.p.}\int_\Gamma \frac{\partial^2 \phi_k( \x-\y )}{\partial \nnn(\x)\,\partial \nnn(\y) }\varphi(\y) {{\rm d}\y}=\partial_{\bm t}\mathrm{V}_{k}[\partial_{\bm t}\varphi](\x)+k^2
{\bm n}(\x)\cdot \mathrm{V}_{k}[\varphi \nnn].
\end{eqnarray*}
These operators are known as the single layer, double layer, adjoint double layer, and hypersingular operators.
In the definition of the hypersingular operator,  ``f.p.'' denotes the Hadamard {\em finite part} integral, since its kernel is hypersingular. Additionally, the second expression for $\mathrm{W}_k$ is the Maue formula, which rewrites the hypersingular operator in terms of tangential derivatives and the single layer operator.

Functional properties in Sobolev spaces $H^s(\Gamma)$ are well known in the literature; we refer the reader to \cite{SaVa:2002,mclean:2000,hsiao2008boundary}, for example. For a smooth boundary $\Gamma$,  
\[
\begin{aligned}
\mathrm{V}_k &: H^s(\Gamma)\to H^{s+1}(\Gamma),\qquad
\mathrm{K}_k,\ \mathrm{K}^\top_k&: H^s(\Gamma)\to H^{s+3}(\Gamma),\qquad
\mathrm{W}_k&:H^s(\Gamma)\to H^{s-1}(\Gamma)
\end{aligned}
\]
We then write $\mathrm{V}_k\in\mathrm{OPS}(-1)$, $\mathrm{W}_k\in\mathrm{OPS}(1)$ and $\mathrm{K}_k,\ \mathrm{K}_k^\top \in\mathrm{OPS}(-3)$. For smoothing operators, that is, for operators that map $H^s(\Gamma)$ to $H^{s+r}(\Gamma)$ for any $r>0$, we simply write $\mathrm{OPS}(-\infty)$.

If ${\bf x}:[0,L]\to\Gamma$ (where $L$ is the length of $\Gamma$) is an arc-length parameterization  of $\Gamma$ and $\varphi:\Gamma\to\mathbb{C}$ a sufficiently smooth function, it holds that
\[
\partial_{\bm t}\varphi\circ {\bf x}(t)= (\varphi\circ{\bf x})'(t). 
\]
We then introduce the derivative operator $\D$ as well as the Hilbert transform operator  $\HH$ on $\Gamma$ as basic blocks for the asymptotic expansions of the Helmholtz BIOs we will use in the next section written in terms of $L-$periodic Fourier series expansions with respect to the parameterization ${\bf x}$:
\begin{align*}
(\mathrm{D}\varphi)\circ{\bf x}(t)&:= (\varphi\circ {\bf x})'(t)= \frac{2\pi i}{L} \sum_{m\ne 0} m \widehat{(\varphi\circ{\bf x})}(m) \exp( i (2\pi m/L)t ),\\
 (\mathrm{H}\varphi)\circ{\bf x}(t) &:=  - i  \sum_{m<0} \widehat{(\varphi\circ{\bf x})}(m) \exp( i (2\pi m/L)t )+
 i\sum_{m\ge 0} \widehat{(\varphi\circ{\bf x})}(m) \exp( i (2\pi m/L)t )&\\
 & \hfill {\widehat{f}(m) := \frac{1}{ L }\int_0^{  L} f(t)\exp( -i (2\pi m/L)t )\,{\rm d}t}, \quad m\in\mathbb{Z}.
\end{align*} 
Using the identification of the Sobolev spaces $H^s(\Gamma)$ with the $L-$periodic Sobolev spaces $H^s_{\rm per}(0,L)$ cf. \cite[Ch. 8]{Kress}, we have that $\D\in\mathrm{OPS}(1)$ and $\HH\in\mathrm{OPS}(0)$. It is also noteworthy that $\mathrm{H}^2 = -\mathrm{I}$.

Then  
\begin{equation}\label{eq:expansions}
\begin{aligned}
 &\mathrm{V}_k - \left[\frac{1}{2}\HH\D_{-1}  -\frac{k^2}4 \HH\D_{-3}\right] \in\mathrm{OPS}(-5),\quad
  \mathrm{W}_k - \left[\frac{1}{2}\HH\D   +\frac{k^2}4 \HH\D_{-1}\right] \in\mathrm{OPS}(-3),\\
 &\mathrm{K}_k^\top - \left[-\frac{k^2\kappa}{4}\HH\D_{-3} \right], \mathrm{K}_k - \left[-\frac{k^2\kappa}{4}\HH\D_{-3} \right]   \in\mathrm{OPS}(-4),\\
 \end{aligned}
\end{equation}
where\footnote{The constant in the expansion of $\mathrm{K}_k$, $\mathrm{K}_k^\top$ corrects a misprint ($k^2\kappa/2$) in \cite[Prop.~3.7 and Thm.~A.3]{dominguez2022robust}; see \cite{arXiv:Robust}. It plays no role in what follows, since these terms are absorbed in the remainders.}, for any positive integer $r$, $\D_r = \D^r$ whereas for any integer $r<0$,
\[
 \D_r = (\D_{-1})^{-r},\quad (\D_{-1}\varphi)\circ{\bf x} =  \frac{L}{2\pi i} \sum_{m\ne 0} \frac{1}m \widehat{(\varphi\circ{\bf x})}(m) \exp( i (2\pi m/L)t  ).
\]
The last operator we will use is the mean value or averaging operator
\[
  \mathrm{J} \varphi = \frac{1}{L}\int_\Gamma  \varphi =  \widehat{\varphi}(0)
  \]
which obviously is an $\mathrm{OPS}(-\infty)$ operator.
For further details, we refer the reader to \cite{hsiao2008boundary}, and for a more comprehensive exposition on the 2D specific case, to \cite{dominguez2022robust}.

\subsection{BIE systems for the Helmholtz decomposition approach for the transmission problem}

Seeking the Helmholtz decomposition fields $(u_p^{(j)}, u_s^{(j)})$, $j=1,2$, as single-layer potentials of the form
\begin{equation}\label{eq:u:as:SL}
u_p^{(j)}={\rm SL}_{p_j}[\varphi^{(p)}]\qquad u_s^{(j)}={\rm SL}_{s_j}[\varphi^{(s)}]
\end{equation}
we find that imposing \eqref{eq:DH} leads to

\begin{eqnarray}
\begin{bmatrix}\partial_{\nnn} u_p^{(j)}+\partial_{\ttt}u_s^{(j)} \\ \partial_{\ttt}u_p^{(j)}-\partial_{\nnn}u_s^{(j)}\end{bmatrix}&=&{\cal A}_{\rm SL}^{(j)} \begin{bmatrix}\varphi^{(p)}\\ \varphi^{(s)}\end{bmatrix}\nonumber\\
{\cal A}^{(j)}_{\rm SL} &=&\frac{1}{2}\begin{bmatrix}
(-1)^j \I  &   \\
&(-1)^{j+1}\I 
\end{bmatrix} +\widetilde{{\cal A}}^{(j)}_{\rm SL}\nonumber\\
\widetilde{{\cal A}}^{(j)}_{\rm SL}&=&
\begin{bmatrix}
 \mathrm{K}_{p_j}^\top &  \partial_{\ttt} \mathrm{V}_{s_j}  \\
\partial_{\ttt} \mathrm{V}_{p_j}  &- \mathrm{K}_{s_j}^\top
\end{bmatrix},\ j=1,2.\label{eq:ASL} 
\end{eqnarray}
On the other hand, seeking the fields in the form of double-layer potentials
\[
u_p^{(j)}={\rm DL}_{p_j}[g^{(p)}]\qquad u_s^{(j)}={\rm DL}_{s_j}[g^{(s)}]
\]
we obtain that the densities $(g^{(p)},g^{(s)})$ must satisfy the following equations
\begin{eqnarray}
\begin{bmatrix}\partial_{\nnn} u_p^{(j)}+\partial_{\ttt}u_s^{(j)} \\ \partial_{\ttt}u_p^{(j)}-\partial_{\nnn}u_s^{(j)}\end{bmatrix}&=&{\cal A}_{\rm DL}^{(j)} \begin{bmatrix}g^{(p)}\\ g^{(s)}\end{bmatrix}\nonumber\\
{\cal A}^{(j)}_{\rm DL} &=&\frac{1}{2}\begin{bmatrix}
&(-1)^{j+1} \partial_{\ttt}       \\
 {(-1)^{j+1} }\partial_{\ttt} &
\end{bmatrix} +\widetilde{{\cal A}}^{(j)}_{\rm DL}\nonumber\\
\widetilde{{\cal A}}^{(j)}_{\rm DL} &=&
\begin{bmatrix}
\mathrm{W}_{p_j}& \partial_{\ttt} \mathrm{K}_{s_j}      \\
\partial_{\ttt} \mathrm{K}_{p_j}&-\mathrm{W}_{s_j}
\end{bmatrix},\ j=1,2.\label{eq:ADL}
 \end{eqnarray}

Let us now focus on applying the Neumann boundary conditions within the framework of the Helmholtz decomposition. 

In our earlier work \cite{dominguez2022robust} we derived Maue-type integration-by-parts formulas to apply these boundary conditions more effectively. These same Neumann boundary conditions, when restated through equations~\eqref{eq:NH_local}, are more easily applied to Helmholtz single and double layer potentials since they are expressed in terms of local boundary operators on the Cauchy data.
Specifically, seeking single-layer potential representations within the context of the Helmholtz decomposition leads immediately to the following formulas. For
\[
u_p^{(j)}={\rm SL}_{p_j}[\varphi^{(p)}]\qquad u_s^{(j)}={\rm SL}_{s_j}[\varphi^{(s)}],
\]
we obtain
\begin{eqnarray*}
\begin{bmatrix}
 {\rm T}_{j}{\bf u}_j\cdot\nnn \\
 {\rm T}_{j}{\bf u}_j\cdot\ttt
\end{bmatrix}
&=&
\mathcal{C}_{\rm SL}^{(j)}
\begin{bmatrix}
\varphi^{(p)}\\
\varphi^{(s)}
\end{bmatrix}
=
(-1)^j\mu_j
\begin{bmatrix}
-\kappa \I & \partial_{\ttt}\\
\partial_{\ttt} & \kappa \I
\end{bmatrix}
\begin{bmatrix}
\varphi^{(p)}\\
\varphi^{(s)}
\end{bmatrix}
+
\widetilde{\mathcal{C}}_{\rm SL}^{(j)}
\begin{bmatrix}
\varphi^{(p)}\\
\varphi^{(s)}
\end{bmatrix}
\\
\widetilde{\mathcal{C}}_{\rm SL}^{(j)}
&:=&
\begin{bmatrix}
(\mathcal{C}_{\rm SL}^{(j)})_{11} & (\mathcal{C}_{\rm SL}^{(j)})_{12}\\
(\mathcal{C}_{\rm SL}^{(j)})_{21} & (\mathcal{C}_{\rm SL}^{(j)})_{22}
\end{bmatrix}
\\
\left(\mathcal{C}_{\rm SL}^{(j)}\right)_{11} &:=&-\omega^2 \mathrm{V}_{p_j} -2\mu_j\Big(\kappa \mathrm{K}_{p_j}^\top  + \partial_{\ttt}^2 \mathrm{V}_{p_j}\Big),\\
\left(\mathcal{C}_{\rm SL}^{(j)}\right)_{12} &:=&2\mu_j \Big(\partial_{\ttt} \mathrm{K}_{s_j}^\top-\kappa \partial_{\ttt} \mathrm{V}_{s_j} \Big),\\
\left(\mathcal{C}_{\rm SL}^{(j)}\right)_{21} &:=&2\mu_j \Big(\partial_{\ttt} \mathrm{K}_{p_j}^\top-\kappa \partial_{\ttt} \mathrm{V}_{p_j} \Big),\\
\left(\mathcal{C}_{\rm SL}^{(j)}\right)_{22} &:=&\omega^2 \mathrm{V}_{s_j} +2\mu_j \Big( \kappa \mathrm{K}_{s_j}^\top + \partial_{\ttt}^2 \mathrm{V}_{s_j}\Big).
\end{eqnarray*}
Alternatively, if we write the solution in the form of double layer potentials
\[
u_p^{(j)}={\rm DL}_{p_j}[g^{(p)}]\qquad u_s^{(j)}={\rm DL}_{s_j}[g^{(s)}]
\]
the Neumann boundary conditions in the $(\nnn,\ttt)$ frame recast via equations~\eqref{eq:NH_local}--\eqref{eq:TnU} lead to the following formulas for $j=1,2$
\[
\begin{aligned}
\begin{bmatrix}
 {\rm T}_{j}{\bf u}_j\cdot\nnn \\
 {\rm T}_{j}{\bf u}_j\cdot\ttt
\end{bmatrix}
&=
\mathcal{C}_{\rm DL}^{(j)}
\begin{bmatrix}
g^{(p)}\\
g^{(s)}
\end{bmatrix}
=
(-1)^j
\begin{bmatrix}
\frac{1}{2}\omega^2 \mathrm{I} + \mu_j\partial_{\ttt}^2 & \mu_j \kappa \partial_{\ttt}\\
\mu_j\kappa \partial_{\ttt} & -\frac{1}{2}\omega^2 \mathrm{I} - \mu_j\partial_{\ttt}^2
\end{bmatrix}
\begin{bmatrix}
g^{(p)}\\
g^{(s)}
\end{bmatrix}
+
\widetilde{\mathcal{C}}_{\rm DL}^{(j)}
\begin{bmatrix}
g^{(p)}\\
g^{(s)}
\end{bmatrix}
\\
&
\widetilde{\mathcal{C}}_{\rm DL}^{(j)}:=
\begin{bmatrix}
(\mathcal{C}_{\rm DL}^{(j)})_{11} & (\mathcal{C}_{\rm DL}^{(j)})_{12}\\
(\mathcal{C}_{\rm DL}^{(j)})_{21} & (\mathcal{C}_{\rm DL}^{(j)})_{22}
\end{bmatrix}
\end{aligned}
\]
with
\begin{eqnarray*} 
\left(\mathcal{C}_{\rm DL}^{(j)}\right)_{11}&:=&-\omega^2 \mathrm{K}_{p_j}-2\mu_j\Big(\kappa \mathrm{W}_{p_j} + \partial_{\ttt}^2 \mathrm{K}_{p_j} \Big),\\
\left(\mathcal{C}_{\rm DL}^{(j)}\right)_{12}&:=&2\mu_j \Big(\partial_{\ttt} \mathrm{W}_{s_j}-\kappa \partial_{\ttt} \mathrm{K}_{s_j}  \Big),\\
\left(\mathcal{C}_{\rm DL}^{(j)}\right)_{21}&:=&2\mu_j \Big(\partial_{\ttt} \mathrm{W}_{p_j}-\kappa \partial_{\ttt} \mathrm{K}_{p_j}  \Big),\\
\left(\mathcal{C}_{\rm DL}^{(j)}\right)_{22}&:=&\omega^2 \mathrm{K}_{s_j} +2\mu_j \Big(\kappa \mathrm{W}_{s_j}  + \partial_{\ttt}^2 \mathrm{K}_{s_j} \Big).
\end{eqnarray*}

For analyzing the functional properties of the four matrix operators $\mathcal{A}_{\rm SL}^{(j)},$ $\mathcal{A}_{\rm DL}^{(j)},$ $\mathcal{C}_{\rm SL}^{(j)}$ and $\mathcal{C}_{\rm DL}^{(j)}$ introduced above  in the framework of Sobolev spaces we extend the notation introduced above and write ${\cal A}\in\mathrm{OPS}(r)$  when ${\cal A}:{\cal H}^s(\Gamma)\to {\cal H}^{s-r}(\Gamma)$ is continuous for any $s$, where
\[
{\cal H}^{s}(\Gamma) := H^s(\Gamma)\times H^s(\Gamma).
\]
We will also write
\[
{\cal A} = {\cal B} +\mathrm{OPS}(r)
\]
to mean that ${\cal A} - {\cal B} \in \mathrm{OPS}(r)$, where ${\cal A}$ and ${\cal B}$ are operators of order higher than $r$. We also follow the convention
\[
{\cal C}{\rm A} = \begin{bmatrix}
C_{11}{\rm A} & C_{12}{\rm A}\\
C_{21}{\rm A} & C_{22}{\rm A}
\end{bmatrix}
\]
for any matrix ${\cal C}$ and scalar operator ${\rm A}$. Typical examples of ${\rm A}$ include multiplication by a smooth function ${a}$, the operators $\D_r$ and $\HH$, as well as products of such operators.

With these conventions it is straightforward to prove the following result:

\begin{proposition}\label{prop:anathema}
For $j=1$ (exterior) and $j=2$ (interior), the following properties hold:
 \[
  {\cal A}^{(j)}_{\rm SL} \in\mathrm{OPS}(0), \quad
  {\cal C}^{(j)}_{\rm SL},\ {\cal A}^{(j)}_{\rm DL} \in\mathrm{OPS}(1),\quad
  {\cal C}^{(j)}_{\rm DL} \in\mathrm{OPS}(2).
 \]
 Furthermore,
 \[
 \begin{aligned}
  {\cal A}^{(j)}_{\rm SL} =&\frac{(-1)^{j}}{2}{\cal H}^{(j)} +
                                 \frac{1}4
                                  \begin{bmatrix}
                                       &(k_s^{(j)})^2 \HH\\
                                    (k_p^{(j)})^2 \HH &
                                   \end{bmatrix} (\HH \D_{-1})^2
+\mathrm{OPS}(-3),\\
   {\cal A}^{(j)}_{\rm DL}  =&   \frac{1}{2} {\cal H}^{(j)}\HH\D
  +\frac{1}{4}
\begin{bmatrix}
 (k_p^{(j)})^2\I & \\
                        & -(k_s^{(j)})^2\I
\end{bmatrix}  \HH \D_{-1} +{\mathrm{OPS}(-2)},
\\
 {\cal C}^{(j)}_{\rm SL}
    = &   -\mu_j{\cal H}^{(j)} \HH\D
                  -\mu_{j}\kappa(-1)^{j} 
                  {\cal H}^{(j)}
+{\mathrm{OPS}(-1)},
                                     \\
{\cal C}^{(j)}_{\rm DL}=&
 (-1)^{j+1}\mu_j{\cal H}^{(j)} (\HH\D)^2
                                   -\mu_j\kappa {\cal H}^{(j)}\HH\D
+ \mathrm{OPS}(0),
\end{aligned}
 \]
 with
 \[
  {\cal H}^{(j)} := \begin{bmatrix}
                                     \I &(-1)^{j }\HH\\
                                    (-1)^{j }\HH & - \I
                                   \end{bmatrix}.
 \]
\end{proposition}
Let us point out that $\left({\cal H}^{(j)}\right)^2 = 0$, which ultimately implies that the underlying  operators ${\cal A}_{\mathrm{SL}}^{(j)}$, ${\cal A}_{\mathrm{DL}}^{(j)}$, ${\cal C}_{\mathrm{SL}}^{(j)}$ and ${\cal C}_{\mathrm{DL}}^{(j)}$   are defective: both their kernels and their cokernels  are infinite-dimensional. This indicates that the equations derived from \eqref{eq:DH}-\eqref{eq:NHT}, for any of the above representations of the solution or combinations thereof, are not well posed.


 \subsection{Combined field formulations}
 
We consider an {\em indirect} combined field approach to the elastic transmission problem within the Helmholtz decomposition framework. More precisely, in each domain $\Omega_j$, $j=1,2$, we seek the scalar fields in the Helmholtz decomposition in the form of an indirect combined field representation
\begin{subequations}\label{eq:representation}
\begin{align}
u^{(p)}_1 & = -\mathrm{SL}_{p_1}[\varphi^{(p)}] + \mathrm{DL}_{p_1}[g^{(p)}], &
u^{(s)}_1 & = -\mathrm{SL}_{s_1}[\varphi^{(s)}] + \mathrm{DL}_{s_1}[g^{(s)}] \quad\text{in } \Omega_1, \label{eq:representation:a} \\
u^{(p)}_2 & = \alpha\left( \mathrm{SL}_{p_2}[\varphi^{(p)}] - \mathrm{DL}_{p_2}[g^{(p)}]\right), &
u^{(s)}_2 & = \alpha\left( \mathrm{SL}_{s_2}[\varphi^{(s)}] - \mathrm{DL}_{s_2}[g^{(s)}]\right) \quad\text{in } \Omega_2. \label{eq:representation:b}
\end{align}
\end{subequations}
where $\alpha\ne 0$ is a coupling parameter at our disposal that can be chosen to optimize the robustness of the combined-field formulation. Using the results of Section~\ref{GradHessCalculations}, the enforcement of the transmission continuity conditions in the $(\nnn,\ttt)$ frame on $\Gamma$ leads to the following system of BIEs
\begin{equation}\label{eq:cfie_t}
\mathfrak{A}_{\alpha}
\begin{bmatrix}
{\bm \varphi}\\
{\bm g}
\end{bmatrix}
=
\begin{bmatrix}
{\bm f}\\
{\bm \lambda}
\end{bmatrix}
\end{equation}
where
\[
{\bm \varphi}:=\begin{bmatrix*}
\varphi^{(p)}\\ 
\varphi^{(s)}
\end{bmatrix*},\quad{\bm g}:=\begin{bmatrix*}
g^{(p)}\\
g^{(s)}
\end{bmatrix*},\
\quad
{\bm f} =
- \begin{bmatrix*}
{\bf u}^{\rm inc}\cdot\nnn\\
{\bf u}^{\rm inc}\cdot\ttt
  \end{bmatrix*},\quad
{\bm \lambda} =
- \begin{bmatrix*}
  {\rm T}_1 {\bf u}^{\rm inc}\cdot\nnn\\
  {\rm T}_1 {\bf u}^{\rm inc}\cdot\ttt
  \end{bmatrix*},\quad
\]
and
\begin{equation}\label{eq:cfie_tb}
\mathfrak{A}_{\alpha}
 =\mathfrak{A}^{(1)}-\alpha\mathfrak{A}^{(2)}  
:=
\begin{bmatrix} -{\cal A}^{(1)}_{\rm SL}  & {\cal A}^{(1)}_{\rm DL}  \\
 -{\cal C}^{(1)}_{\rm SL} &  {\cal C}^{(1)}_{\rm DL} 
\end{bmatrix} -\alpha 
\begin{bmatrix}
{\cal A}^{(2)}_{\rm SL}  &  - {\cal A}^{(2)}_{\rm DL}\\
\rho{\cal C}^{(2)}_{\rm SL}  &  - \rho{\cal C}^{(2)}_{\rm DL}
\end{bmatrix}.
\end{equation}

We will first prove the uniqueness of the solution of the BIE formulation \eqref{eq:cfie_t}. To achieve this, we will use the following result on the invertibility of a class of operators arising from BIE formulation.

\begin{theorem}\label{theo:transmission:helmholtz} 
  Let $k_1, k_2>0$, $\{\V_{k_j},\K_{k_j},\K_{k_j}^\top,  \W_{k_j}\}$, $j=1,2$ be the Helmholtz boundary integral operators associated with wavenumbers $k_1$ and $k_2$. Then, for any $\alpha\ne 0$ and $\rho>0$, the operator ${\cal L}_{\alpha,\rho}:H^s (\Gamma)\times H^{s-1}(\Gamma) \to H^s(\Gamma)\times H^{s-1}(\Gamma)$ given by
\begin{equation}\label{eq:Lformulation}
{\cal L}_{\alpha,\rho}\begin{bmatrix}
g\\
\varphi\\
\end{bmatrix}:= 
\frac{1}2 \begin{bmatrix}
(1-\alpha)g\\
(1-\alpha\rho )\varphi\\
\end{bmatrix} +\begin{bmatrix}
{\rm K}_{k_1}+\alpha {\rm K}_{k_2} & -{\rm V}_{k_1}-\alpha {\rm V}_{k_2}\\
{\rm W}_{k_1}+\alpha\rho {\rm W}_{k_2} & -{\rm K}^\top_{k_1}-\alpha  \rho{\rm K}^\top_{k_2}\\
\end{bmatrix} \begin{bmatrix}
g\\
\varphi\\
\end{bmatrix} 
\end{equation}
is invertible.
\end{theorem}\begin{proof} Notice first that, as a consequence of \eqref{eq:expansions}, the principal part of ${\cal L}_{\alpha,\rho}$ is given by
  \begin{equation}
 {\cal L}^{\rm p}_{\alpha,\rho}:=
\frac{1}2\begin{bmatrix*}[r]
(1-\alpha)\mathrm{I} &   -(1+\alpha) \mathrm{H}(\D_{-1}+\J)\\
(1+\rho \alpha)\mathrm{H}(\D +\J) & (1-\alpha\rho)\I
\end{bmatrix*}. 
\end{equation}
That is, ${\cal L}_{\alpha,\rho}-{\cal L}^{\rm p}_{\alpha,\rho}: H^s(\Gamma)\times H^{s-1}(\Gamma) \to H^s(\Gamma)\times H^{s-1}(\Gamma)$ is compact. Moreover, a direct calculation shows that
\[
   \left({\cal L}^{\rm p}_{\alpha,\rho}\right)^{-1}=
-\frac{1}{\alpha(\rho+1)}  \begin{bmatrix*}[r]
 (1-\alpha\rho)\I &  (1+\alpha ) \mathrm{H}(\D_{-1}+\J)\\
 -(1+\rho \alpha)\mathrm{H}(\D +\J) &  (1-\alpha)\mathrm{I}
\end{bmatrix*}.
\]
Therefore, ${\cal L}_{\alpha,\rho}$ is Fredholm of index zero.  

We will prove the injectivity of ${\cal L}_{\alpha,\rho}$ by applying a well-known argument that uses the jump relations for single and double layer potentials and the dual transmission problem. References for such arguments for $\alpha=\pm 1$ can be found in \cite{costabel-stephan,kress1978transmission}. We extend such results to any $\alpha\ne 0$ (and $\rho>0$) for the sake of completeness. Let $(g,\varphi)\in \ker({\cal L}_{\alpha,\rho})$, and define in $\mathbb{R}^2\setminus\Gamma$ the fields
\[
u_1:= -{\rm SL}_{k_1}[\varphi] +{\rm DL}_{k_1}[g],\quad  u_2:= \alpha \left( {\rm SL}_{k_2}[\varphi] - {\rm DL}_{k_2}[g]\right). 
\]Since $(g,\varphi)\in \ker({\cal L}_{\alpha,\rho})$, the pair $v_1 := u_1|_{\Omega_1}$, $v_2 := u_2|_{\Omega_2}$ solves the homogeneous transmission problem
\begin{equation}\label{eq:trhelmholt:01}
 \left\{
 \begin{array}{ll}
 \Delta v_1 + k_1^2 v_1 = 0, & \text{in } \Omega_1,\\
 \Delta v_2 + k_2^2 v_2 = 0, & \text{in } \Omega_2,\\
 \gamma^{(1)} v_1 - \gamma^{(2)} v_2 =  0, & \\
 \partial_{\bm n}^{(1)} v_1 -  \rho\, \partial_{\bm n}^{(2)} v_2 = 0, & \\
 \partial_{\bm r} v_1 - i k_1 v_1 = o(r^{-1/2}), &
 \end{array}
 \right.
\end{equation}
which is well posed. Hence $u_1=0$ in $\Omega_1$ and $u_2=0$ in $\Omega_2$, and the jump relations \eqref{eq:traces} yield
\[
\begin{aligned}
 \gamma^{(2)} u_1 &= -\big(\gamma^{(1)} u_1-\gamma^{(2)} u_1\big) = -g, &
 \partial_{\bm n}^{(2)} u_1 &= -\big(\partial_{\bm n}^{(1)} u_1-\partial_{\bm n}^{(2)} u_1\big) = -\varphi,\\
 \gamma^{(1)} u_2 &= \gamma^{(1)} u_2-\gamma^{(2)} u_2 = -\alpha g, &
 \partial_{\bm n}^{(1)} u_2 &= \partial_{\bm n}^{(1)} u_2-\partial_{\bm n}^{(2)} u_2 = -\alpha\varphi .
 \end{aligned}
\]
That is, the Cauchy data of $u_2|_{\Omega_1}$ are $\alpha$ times those of $u_1|_{\Omega_2}$, so that, by Green's formula in $\Omega_2$,
\[
 \int_{\Gamma} \partial_{\bm n}^{(1)} u_2\, \overline{u_2}
 = |\alpha|^2 \int_{\Gamma} \partial_{\bm n}^{(2)} u_1\, \overline{u_1}
 = |\alpha|^2\Big(\int_{\Omega_2}|\nabla u_1|^2 - k_1^2 \int_{\Omega_2}| u_1|^2\Big)\in\mathbb{R}.
\]
Since $u_2|_{\Omega_1}$ is a radiating solution of the Helmholtz equation with wavenumber $k_2$, applying Green's formula again, now in the intersection of $\Omega_1$ with the ball $B({\bf 0},R)$ for $R$ large enough, we obtain
\[
0=\Im \int_{\Gamma} \partial_{\bm n}^{(1)} u_2\, \overline{u_2}=\lim_{R\to \infty}
\Im  \int_{|{\bf x}|=R} \partial_{\bm r} u_2\, {\overline{u_2}}=k_2\lim_{R\to \infty}
 \int_{|{\bf x}|=R} |u_2|^2.
\]
Rellich's lemma (cf. \cite[Lemma 3.3.2]{SaVa:2002} or \cite[Lemma 2.12]{ColtonKress}) implies $u_2=0$ in $\Omega_1$. Therefore $g=-\alpha^{-1}\gamma^{(1)}u_2=0$ and $\varphi=-\alpha^{-1}\partial_{\bm n}^{(1)}u_2=0$, which proves the injectivity of ${\cal L}_{\alpha,\rho}$ and therefore its invertibility.
\end{proof}

We are ready to prove the uniqueness of the solution for the BIE formulation \eqref{eq:cfie_t} of our problem.
\begin{theorem}\label{theo:unicity}
Equation \eqref{eq:cfie_t} admits at most  one solution for any $\alpha\ne 0$ and $\rho> 0$.
\end{theorem}
\begin{proof}  
Let $({\bm \varphi},{\bm g})$ be a solution of \eqref{eq:cfie_t} with homogeneous data, and let $u^{(p)}_j$ and $u^{(s)}_j$ be defined by \eqref{eq:representation}. By uniqueness of the transmission problem, we have
\[
\nabla u^{(p)}_j + \vcurl u^{(s)}_j = {\bf 0},
\qquad \text{in } \Omega_j, \quad j=1,2.
\]
Taking the divergence, we obtain
\[
\Delta u^{(p)}_j
=
\divv\big(\nabla u^{(p)}_j + \vcurl u^{(s)}_j\big)
=0,\quad \text{in } \Omega_j, \quad j=1,2.
\]
Since $u^{(p)}_j$ also satisfies the Helmholtz equation in $\Omega_j$, it follows that actually $u^{(p)}_j=0$ in $\Omega_j$. 

Hence the Cauchy data of $u^{(p)}_1$ and $u^{(p)}_2$ on $\Gamma$ vanish and, by \eqref{eq:representation:a}, the pair $(g^{(p)},\varphi^{(p)})$ solves the homogeneous system of BIEs
\begin{equation}\label{eq:Lformulation:principal:part}
\frac{1}2 \begin{bmatrix}
(1-\alpha)g^{(p)}\\
(1-\alpha\rho )\varphi^{(p)}\\
\end{bmatrix} +\begin{bmatrix}
{\rm K}_{p_1}+\alpha {\rm K}_{p_2} & -{\rm V}_{p_1}-\alpha {\rm V}_{p_2}\\
{\rm W}_{p_1}+\alpha\rho {\rm W}_{p_2} & -{\rm K}^\top_{p_1}-\alpha  \rho{\rm K}^\top_{p_2}\\
\end{bmatrix} \begin{bmatrix}
g^{(p)}\\
\varphi^{(p)}\\
\end{bmatrix}
=\begin{bmatrix}
  0\\
  0
\end{bmatrix},
\end{equation}
which is \eqref{eq:Lformulation} with $k_j=k_p^{(j)}$.
By Theorem \ref{theo:transmission:helmholtz}, this system is uniquely solvable, which implies that $g^{(p)}=\varphi^{(p)}=0$.

The argument for $u_j^{(s)}$ is analogous. Indeed, taking the curl in
\[
\nabla u^{(p)}_j + \vcurl u^{(s)}_j = {\bf 0},
\]
we deduce
\[
0=\curl\big(\nabla u^{(p)}_j + \vcurl u^{(s)}_j\big) = -\Delta u^{(s)}_j.
\]
Since $u^{(s)}_j$ also satisfies the Helmholtz equation in $\Omega_j$, we conclude that $u^{(s)}_j=0$ in $\Omega_j$. Repeating the same argument as above for the shear wave component, we obtain $\varphi^{(s)}=0$ and $g^{(s)}=0$.
\end{proof}

Existence of a solution will be derived as a consequence of Theorem \ref{theo:01}; see Corollary \ref{cor:01}.

In what follows, we refer to the CFIE formulation~\eqref{eq:cfie_t} as $(\nnn,\ttt)$-CFIE$_\alpha$.   It is easy to conclude from Proposition \ref{prop:anathema} that $\mathfrak{A}_{\alpha}:  \mathcal{H}^{s}(\Gamma)\times \mathcal{H}^{s+1}(\Gamma) \to  \mathcal{H}^{s} (\Gamma)\times \mathcal{H}^{s-1}(\Gamma)$.

We will now turn our attention to the principal part of $\mathfrak{A}_{\alpha}$. For this purpose, we introduce the following notation: For $\mathrm{A}$ a scalar operator, we define the matrix operator
\begin{equation}\label{eq:notation:cal}
{\cal A } := {\rm I}\otimes {\mathrm{A}} = \begin{bmatrix}
\mathrm{A} & \\
          &\mathrm{A}
\end{bmatrix}.
\end{equation}
We will apply this convention to $\HH$, $\D_r$ and $\J$ to define ${\cal H}$, ${\cal D}_r$, ${\cal I}$ and ${\cal J}$. We also introduce the operator,
\begin{equation}\label{eq:Hbeta}
\mathcal{H}_{\bm{\beta}}= 
\begin{bmatrix}
\beta_1\I & \beta_2\HH\\
\beta_2\HH & -\beta_1 \I 
\end{bmatrix},\qquad \bm{\beta} = (\beta_1,\beta_2).
\end{equation}
Observe that this notation extends that for ${\cal H}^{(j)}$ used in Proposition \ref{prop:anathema}, since ${\cal H}^{(1)} = \mathcal{H}_{(1,-1)}$ and ${\cal H}^{(2)} = \mathcal{H}_{(1,1)}$. Notice also that
\begin{equation}\label{eq:Hinvbeta}
  r \mathcal{H}_{\bm{\beta}}+s \mathcal{H}_{\bm{\gamma}} =  
\mathcal{H}_{r\bm{\beta}+s\bm{\gamma}},\quad
\mathcal{H}^{-1}_{\bm{\beta}} = \frac{1}{\beta_1^2-\beta_2^2}  
\mathcal{H}_{\bm{\beta}}.
\end{equation}
In particular, $\mathcal{H}_{\bm{\beta}}$  is invertible if and only if $\beta_1^2-\beta_2^2\ne 0$.

With this notation it is straightforward to check that the principal part of $\mathfrak{A}_{\alpha}$  in \eqref{eq:cfie_tb} is given by 
\[
\mathfrak{A}_{\rm p,\alpha}:=
\mathfrak{A}_{\rm p}^{(1)} - \alpha \mathfrak{A}_{\rm p}^{(2)} + \mathrm{OPS}(-1)
\]
with
\[
\begin{aligned}
\mathfrak{A}_{\rm p}^{(1)}& :=
\frac{1}2 \begin{bmatrix}
{\cal I} & \\
&  2\mu_1\mathcal{H}(\mathcal{D}+{\cal \J}) 
\end{bmatrix}
\begin{bmatrix*}[r]
\mathcal{H}_{(1,-1)} &    \mathcal{H}_{(1,-1)}\\
 \mathcal{H}_{(1,-1)} &  \mathcal{H}_{(1,-1)} 
\end{bmatrix*}
\begin{bmatrix}
{\cal I} & \\
&\mathcal{H}(\mathcal{D}+{\cal \J}) 
\end{bmatrix},\\
\mathfrak{A}_{\rm p}^{(2)} & :=
\frac{1}2 \begin{bmatrix}
{\cal I} & \\
& 2\rho\mu_2 \mathcal{H}(\mathcal{D}+{\cal \J}) 
\end{bmatrix}
\begin{bmatrix*}[r]
\mathcal{H}_{(1,1)} &  - \mathcal{H}_{(1,1)}\\
-\mathcal{H}_{(1,1)} &  \mathcal{H}_{(1,1)} 
\end{bmatrix*}
\begin{bmatrix}
{\cal I} & \\
& \mathcal{H}(\mathcal{D}+{\cal \J}) 
\end{bmatrix}.
\end{aligned}
\] 
\begin{proposition}
Let $\mathfrak{A}_{\mathrm{p},\alpha}$ be the principal part of
$\mathfrak{A}_{\alpha}:\mathcal{H}^{s}(\Gamma)\times \mathcal{H}^{s+1}(\Gamma) \to \mathcal{H}^{s}(\Gamma)\times \mathcal{H}^{s-1}(\Gamma)$.
Then, the kernel of $\mathfrak{A}_{\mathrm{p},\alpha}$ is given by
\[
\left\{ (-(\D + \J)f,0,0,f)^\top + (0,(\D + \J)g,g,0)^\top \ :\ f,g \in H^{s+1}(\Gamma) \right\}
\subset \mathcal{H}^{s}(\Gamma)\times \mathcal{H}^{s+1}(\Gamma),
\]
and is therefore independent of the choice of $\alpha$ and $\rho$. In particular, $\ker(\mathfrak{A}_{\mathrm{p},\alpha})$ is infinite-dimensional.

Furthermore, the cokernel of $\mathfrak{A}_{\mathrm{p},\alpha}$ is also infinite-dimensional, although it depends on the choice of $\alpha$ and $\rho$.

Consequently, $\mathfrak{A}_{\alpha}:\mathcal{H}^{s}(\Gamma)\times \mathcal{H}^{s+1}(\Gamma) \to \mathcal{H}^{s}(\Gamma)\times \mathcal{H}^{s-1}(\Gamma)$ is not a Fredholm operator.
\end{proposition}

\begin{proof}
It is easy to check that for any $f,g \in H^{s+1}(\Gamma)$, we have
\[
(-(\D + \J)f,0,0,f)^\top, \quad
(0,(\D + \J)g,g,0)^\top
\in\ker(\mathfrak{A}_{\rm p}^{(j)}), \quad j=1,2.
\]
Now, let $(\bm{\varphi},\bm{g}) \in \mathcal{H}^{s}(\Gamma)\times \mathcal{H}^{s+1}(\Gamma)$ be an arbitrary element, which can be written as
\[
(\bm{\varphi},\bm{g}) = (-(\D + \J)f_1, ( \D + \J)f_2, f_3, f_4)^\top, \quad f_i \in H^{s+1}(\Gamma).
\]
Notice that we can decompose it as
\[
(\bm{\varphi},\bm{g}) = (-(\D + \J)f_1,0,0,f_1)^\top + (0,(\D + \J)f_2,f_2,0)^\top + (0,0,f_3-f_2,f_4-f_1)^\top.
\]
Since the first two terms belong to the kernel of $\mathfrak{A}_{\rm p}^{(j)}$, we have that $(\bm{\varphi},\bm{g}) \in \ker(\mathfrak{A}_{\rm p,\alpha})$ if and only if (cf. \eqref{eq:Hinvbeta})
\[
(f_3-f_2, f_4-f_1)^\top \in \ker(\mathcal{H}_{(1,-1)} + \alpha \mathcal{H}_{(1,1)})\cap\ker(\mathcal{H}_{(1,-1)} - r \mathcal{H}_{(1,1)}), \quad r = \alpha\rho \frac{\mu_2}{\mu_1} \ne 0,
\]
the two conditions coming from the first and second block rows of $\mathfrak{A}_{\rm p,\alpha}$, respectively. Since $\mathcal{H}_{(1,-1)} + \alpha \mathcal{H}_{(1,1)}=\mathcal{H}_{(1+\alpha,-1+\alpha)}$ is invertible by \eqref{eq:Hinvbeta} ($(1+\alpha)^2-(\alpha-1)^2=4\alpha\ne0$), we conclude that $f_3=f_2$ and $f_4=f_1$, which proves the first part.

For the cokernel, one checks in the same way that the range of $\mathfrak{A}_{\rm p,\alpha}$ is the graph of a bounded operator over two of the four components, so that it has infinite codimension; its description depends on $\alpha$, $\rho$, $\mu_1$, $\mu_2$ and we omit it.
\end{proof}

Thus, the $(\nnn,\ttt)$-CFIE$_\alpha$ is not a Fredholm equation: \eqref{eq:cfie_tb} is not boundedly invertible and, upon discretization, is expected to yield ill-conditioned linear systems. Moreover, iterative solvers are expected to require many iterations, with performance deteriorating further as the discretization is refined. Therefore, effective preconditioners are essential in practice for the efficient solution of the resulting linear systems.

 \section{Preconditioning}
\label{sec:preconditioning}
 The aim of this section is to propose a pair of preconditioners that give rise to a well-posed formulation. In other words, we seek to define easy-to-construct and discretizable preconditioners $\mathfrak{P}^{\rm left}$ and $\mathfrak{P}^{\rm right}$ such that
\[
  \mathfrak{P}^{\rm left}  \mathfrak{A}_\alpha  \mathfrak{P}^{\rm right} :{\cal H}^{s}(\Gamma)\times {\cal H}^{s}(\Gamma)\to {\cal H}^{s}(\Gamma)\times {\cal H}^{s}(\Gamma)
\]
is not only continuous and invertible but also a compact perturbation of the identity, endowing the operator with a highly clustered spectrum---which translates into fast convergence of Krylov solvers after discretization.

The construction of such a pair of preconditioners relies on pseudodifferential expansions of $\mathfrak{A}_\alpha$ to correct the defective principal part. By composing first- and second-order pseudodifferential operators---essentially combinations of tangential derivatives, Hilbert transforms, and multiplication operators---we reconstruct an invertible principal part. Fredholm theory is then invoked to establish the invertibility of the resulting preconditioned operator; by Theorem \ref{theo:unicity} and the invertibility of the preconditioners, this reduces simply to proving injectivity. Although the conceptual design is straightforward and essentially follows an LU-like factorization of the matrix representation of the principal symbol of $\mathfrak{A}_\alpha$, the explicit calculations are rather involved due to the large number of parameters ($\mu_j$, $\lambda_j$, $\omega$, $\rho$ in the original problem and the coupling parameter $\alpha$) in the BIE formulation. We emphasize that such preconditioners are not unique, and the construction presented here is one of several possible choices.

Let us briefly sketch how such a construction can be obtained. We start by noticing that with
\[
\mathfrak{A}_1^{\rm right}:=\begin{bmatrix}
{\cal I}\\
& ({\cal H D}+{\cal J})^{-1}
\end{bmatrix},  \quad  \mathfrak{A}^{\rm left}_{1}:=\mathfrak{A}_1^{\rm right},
\]
we have
\[
\begin{aligned}
&\mathfrak{A}^{\rm left}_{1}\mathfrak{A}_\alpha\mathfrak{A}^{\rm right}_{1}
=\
\begin{bmatrix}
\tfrac12\mathcal H_{(1,-1)}-\tfrac{\alpha}{2}\mathcal H_{(1,1)} &
\tfrac12\mathcal H_{(1,-1)}+\tfrac{\alpha}{2}\mathcal H_{(1,1)}\\[1ex]
\mu_1\mathcal H_{(1,-1)}+\alpha\rho\mu_2\mathcal H_{(1,1)} &
\mu_1\mathcal H_{(1,-1)}-\alpha\rho\mu_2\mathcal H_{(1,1)}
\end{bmatrix}
\\
&
+\begin{bmatrix}
0 & 0\\[1ex]
(\mathcal H\mathcal D+\mathcal J)^{-1}\kappa\,
\bigl(-\mu_1\mathcal H_{(1,-1)}+\alpha\rho\mu_2\mathcal H_{(1,1)}\bigr) &
-(\mathcal H\mathcal D+\mathcal J)^{-1}\kappa\,
\bigl(\mu_1\mathcal H_{(1,-1)}+\alpha\rho\mu_2\mathcal H_{(1,1)}\bigr)
\end{bmatrix}
\\
& +\frac14
\begin{bmatrix}
\begin{bmatrix*}[r]
      0         & -\nu_s^{+}\,\HH \\[0.5ex]
 -\nu_p^{+}\,\HH &      0
\end{bmatrix*}
&
\begin{bmatrix*}[r]
 -\nu_p^{+}\,\I &      0        \\[0.5ex]
      0        &  \nu_s^{+}\,\I
\end{bmatrix*}
\\[4ex]
\begin{bmatrix*}[r]
 2\bigl(\xi_p-(1+\alpha\rho)\omega^2\bigr)\I & 0 \\[0.5ex]
      0                                  & 0
\end{bmatrix*}
&
\begin{bmatrix*}[r]
 -2(1-\alpha\rho)\,\omega^2\,\I & 2(1+\alpha\rho)\,\omega^2\,\HH \\[0.5ex]
  2\,\xi_p\,\HH                 & 2(1-\alpha\rho)\,\omega^2\,\I
\end{bmatrix*}
\end{bmatrix}
(\mathcal H\mathcal D+\mathcal J)^{-2}\\
&
+{\rm OPS}(-3),\end{aligned}
\]
where $\xi_p$ is the first of several coefficients that we will need to introduce, namely
\[
\xi_p := \mu_1 (k_{p}^{(1)})^2 + \alpha\rho\mu_2(k_{p}^{(2)})^2.
\]
Next define
\[
  \mathfrak{A}^{\rm left}_{2} := \begin{bmatrix}
    {\frac{1}2 \mathcal{H}_{(1,1)}-\frac{1}{2\alpha} \mathcal{H}_{(1, -1)} }& 0\\
   -\frac12\,\mathcal H_{(1,1)}
+
\frac{1}{2\alpha}\,\mathcal H_{(1,-1)}& \frac{1}{4\alpha\rho\mu_1\mu_2}
\left(
\mu_1 \mathcal H_{(1,-1)}
+
\alpha\rho\mu_2 \mathcal H_{(1,1)}
\right)
    \end{bmatrix}.
\]
Then, simple calculations, which now take into account the next term of the expansions \eqref{eq:expansions}, show that
\[
\begin{aligned}
\mathfrak{A}^{\rm left}_{2}\mathfrak{A}^{\rm left}_{1}\mathfrak{A}_\alpha\mathfrak{A}^{\rm right}_{1}
 =\ &
\begin{bmatrix} \mathcal I & -\mathcal H^*\\ 0 & 0\end{bmatrix}
+\begin{bmatrix} 0 & 0\\
(\mathcal H\mathcal D+\mathcal J)^{-1}\kappa\,\mathcal H^* &
-(\mathcal H\mathcal D+\mathcal J)^{-1}\kappa
\end{bmatrix} \\
&
+(\mathcal H\mathcal D+\mathcal J)^{-2}\begin{bmatrix} \mathcal R_{11} & \mathcal R_{12}\\ \mathcal R_{21} & \mathcal R_{22}\end{bmatrix}
+{\rm OPS}(-3),
\end{aligned}
\]
where
\[
\begin{aligned}
  \mathcal H^*\ :=\ &\begin{bmatrix} & \HH\\ -\HH & \end{bmatrix},\\
\mathcal R_{11}  \ :=\ &\frac{1}{8\alpha}\begin{bmatrix}
\displaystyle(\alpha+1)\nu_p^{+} \I  & \displaystyle -(\alpha-1)\nu_s^{+} \HH\\[2mm]
\displaystyle(\alpha-1)\nu_p^{+} \HH &\displaystyle (\alpha+1)\nu_s^{+} \I
\end{bmatrix},\qquad &
\mathcal R_{12} \ :=\ &\frac{1}{8\alpha}\begin{bmatrix}
\displaystyle-(\alpha-1)\nu_p^{+} \I & \displaystyle(\alpha+1)\nu_s^{+} \HH\\[2mm]
\displaystyle-(\alpha+1)\nu_p^{+} \HH &\displaystyle -(\alpha-1)\nu_s^{+} \I
\end{bmatrix},\\
 {\cal R}_{21} \ := \ &
\frac{1}{8\alpha}
\begin{bmatrix}
\displaystyle
-\left( \eta \chi_p^- + \zeta_{+,+} \right)\I
&
\displaystyle
(\alpha-1)\nu^+_s\,\HH
\\[2mm]
\displaystyle
-\left( \eta \chi_p^+ - \zeta_{-,+} \right)\HH
&
\displaystyle
-(\alpha+1)\nu^+_s\,\I
\end{bmatrix}, &
{\cal R}_{22}
\  := \ &
\frac{1}{8\alpha}
\begin{bmatrix}
\displaystyle
\left( \eta \chi_p^+ - \zeta_{+,-} \right)\I
&
\displaystyle
(\alpha-1)\nu^-_s\,\HH
\\[2mm]
\displaystyle
\left( \eta\chi_p^- + \zeta_{-,-} \right)\HH
&
\displaystyle
-(\alpha+1)\nu^-_s\,\I
\end{bmatrix},
\end{aligned}
\]
with the following coefficients
\[
\begin{aligned}
  \nu_p^{+}& \ :=(k_{p}^{(1)})^2+\alpha (k_{p}^{(2)})^2,\\
\nu_s^\pm &\ := (k_s^{(1)})^2 \pm \alpha (k_s^{(2)})^2, &
\chi_p^\pm &\ := \pm\mu_1(k_p^{(1)})^2 + \alpha^2\mu_2\rho(k_p^{(2)})^2, \\[2mm]
\eta &\ := \frac{\mu_1-\mu_2\rho}{\mu_1\mu_2\rho}&
\zeta_{\pm,\pm} &\ := \frac{(\mu_1 \pm \alpha\mu_2\rho)(1 \pm \alpha\rho)}{\mu_1\mu_2\rho}\omega^2 .
\end{aligned}
\]
To obtain the term of order $-1$ in the form displayed above we have used that $\HH a-a\HH\in{\rm OPS}(-\infty)$ for any smooth function $a$.

We can continue with
\[
\mathfrak{A}^{\rm left}_{3}:=\begin{bmatrix}
{\cal I}& 0\\
-({\cal H}{\cal D}+{\cal J})\kappa {\cal H}^*     & ({\cal H}{\cal D}+{\cal J})^{2}
\end{bmatrix}
\]
that cancels the leading term in the second row of the matrix operator above and elevates the next term to become the new principal part, resulting in
\[
\begin{aligned}
  \mathfrak{A}^{\rm left}_{3}
\mathfrak{A}^{\rm left}_{2}\mathfrak{A}^{\rm left}_{1}
\mathfrak{A}_{\alpha}\mathfrak{A}^{\rm right}_1
& = \begin{bmatrix}
  {\cal I } & -{\cal H}^*\\
{\cal R}_{21} &{\cal R}_{22}
  \end{bmatrix} +{\rm OPS}(-1).
\end{aligned}
\]
The first row is not affected by $\mathfrak A^{\rm left}_3$; the terms $(\mathcal H\mathcal D+\mathcal J)^{-2}\mathcal R_{11}$, $(\mathcal H\mathcal D+\mathcal J)^{-2}\mathcal R_{12}$ are of order $-2$, and their contribution to the second row through the $(2,1)$ entry of $\mathfrak A^{\rm left}_3$ is of order $-1$ and therefore can be absorbed into the remainder term ${\rm OPS}(-1)$.

This suggests taking
\[
\mathfrak{A}_{4}^{\rm left} :=
\begin{bmatrix}
{\cal I} & 0\\
-8\alpha \rho \mu_1 \mu_2 {\cal R}_{22} {\cal H}^* &
-8\alpha \rho \mu_1 \mu_2 {\cal I}
\end{bmatrix}
\]
which cancels out the bottom-right entry of the matrix operator above:
\[
\begin{aligned}
  \mathfrak{A}^{\rm left}_{4}
  \mathfrak{A}^{\rm left}_{3}
\mathfrak{A}^{\rm left}_{2}\mathfrak{A}^{\rm left}_{1}
\mathfrak{A}_{\alpha}\mathfrak{A}^{\rm right}_1
& =
\begin{bmatrix}
  {\cal I } & -{\cal H}^*\\
{\cal S}_{21} & 0
  \end{bmatrix}+{\rm OPS}(-1),
\end{aligned}
\]
where
\begin{equation}\label{eq:S21}
  {\cal S}_{21}:=
\begin{bmatrix}
(\sigma+\zeta)\I & (\sigma-\zeta)\HH\\
-(\sigma-\zeta)\HH & (\sigma+\zeta) \I
\end{bmatrix}
\end{equation}
with
\[
\begin{aligned}
\sigma \ :=\ &
(\mu_1+\mu_2\rho)\omega^2
-
\mu_1(\mu_1-\mu_2\rho)\left(k_{p}^{(1)}\right)^2,\\
\zeta \ :=\ & \alpha^2\rho
\left[
(\mu_1+\mu_2\rho)\omega^2
+
\mu_2(\mu_1-\mu_2\rho)\left(k_{p}^{(2)}\right)^2
\right].
\end{aligned}
\]
Notice that
\[
{\cal S}_{21}^{-1} = \frac{1}{4\sigma\zeta}
\begin{bmatrix}
(\sigma+\zeta)\I & -(\sigma-\zeta)\HH\\
(\sigma-\zeta)\HH & (\sigma+\zeta) \I
\end{bmatrix}.
\]
Then, with
\[
\mathfrak{A}_{5}^{\rm left} :=
\begin{bmatrix}
{\cal I} & 0\\
0  &
{\cal S}_{21}^{-1}
\end{bmatrix}
\]
we get
\[
\begin{aligned}
  \mathfrak{A}^{\rm left}_{5}
  \mathfrak{A}^{\rm left}_{4}
  \mathfrak{A}^{\rm left}_{3}
\mathfrak{A}^{\rm left}_{2}\mathfrak{A}^{\rm left}_{1}
\mathfrak{A}_{\alpha}\mathfrak{A}^{\rm right}_1
& =
\begin{bmatrix}
  {\cal I } & -{\cal H}^*\\
{\cal I } & 0
  \end{bmatrix}+{\rm OPS}(-1).
\end{aligned}
\]
Now, it suffices to take
\[
  \mathfrak{A}^{\rm left}_{6} := \begin{bmatrix*}[r]
   0      & {\cal I} \\
  {\cal I}& -{\cal I}
  \end{bmatrix*},\quad
  \mathfrak{A}^{\rm right}_{2} := \begin{bmatrix}
  {\cal I}&  \\
         & -{\cal H}^*
  \end{bmatrix}
\]
to obtain, after preconditioning, a compact perturbation of the identity, i.e.,
\[
\mathfrak A^{\rm left}_6\cdots\mathfrak A^{\rm left}_1\,\mathfrak A_\alpha\,\mathfrak A_1^{\rm right}\mathfrak A_2^{\rm right}=\mathfrak I+{\rm OPS}(-1), \quad \text{with}\quad
\mathfrak I :=  \begin{bmatrix*}[r]
  {\cal I}& \\
         &  {\cal I}
  \end{bmatrix*}.
\]

We summarize, slightly rewrite and complete the above results in the following theorem.
\begin{theorem}\label{theo:01}
Define
\[
\begin{aligned}
\mathfrak{P}^{\rm left} &\ : = \mathfrak{A}^{\rm left}_{6} \mathfrak{A}^{\rm left}_{5} \mathfrak{A}^{\rm left}_{4}
  \mathfrak{A}^{\rm left}_{3}
\mathfrak{A}^{\rm left}_{2}\mathfrak{A}^{\rm left}_{1},\\
\mathfrak{P}^{\rm right}&\ :=
 \begin{bmatrix*}
{\cal I} & 0\\
       0  & \begin{bmatrix*}[r]
         0&-(\D +\J)^{-1}\\
         (\D+\J)^{-1}& 0\\
         \end{bmatrix*} 
        \end{bmatrix*}=
        \mathfrak{A}_1^{\rm right} \mathfrak{A}_2^{\rm right} +{\rm OPS}(-\infty).
\end{aligned}
\]
Then  $ \mathfrak{P}^{\rm left}  \mathfrak{A}_\alpha  \mathfrak{P}^{\rm right} : {\cal H}^{s}(\Gamma)\times {\cal H}^{s}(\Gamma)\to {\cal H}^{s}(\Gamma)\times {\cal H}^{s}(\Gamma)$ is continuous and invertible. Moreover,   $\mathfrak{P}^{\rm left}  \mathfrak{A}_\alpha  \mathfrak{P}^{\rm right}$ is a compact perturbation of the identity, i.e.,
\begin{equation}\label{eq:01:theo:01}
  \mathfrak{P}^{\rm left}  \mathfrak{A}_\alpha  \mathfrak{P}^{\rm right} =
\mathfrak{I}+{\rm OPS}(-1).
\end{equation}
\end{theorem}
\begin{proof}
  We have already shown that $\mathfrak P^{\rm left}\mathfrak A_\alpha\mathfrak A_1^{\rm right}\mathfrak A_2^{\rm right}$ is a compact perturbation of the identity, and $\mathfrak P^{\rm right}$ differs from $\mathfrak A_1^{\rm right}\mathfrak A_2^{\rm right}$ by a smoothing operator. Hence, \eqref{eq:01:theo:01} is proven.
  
  Furthermore, $\mathfrak{P}^{\rm right}$  is invertible, in particular injective.  On the other hand,   $\mathfrak{P}^{\rm left}$ is injective if so are
 $\mathfrak{A}^{\rm left}_{2}:{\cal H}^s(\Gamma)\times {\cal H}^s(\Gamma)\to {\cal H}^s(\Gamma)\times {\cal H}^s(\Gamma)$ and ${\cal S}_{21}:{\cal H}^s(\Gamma)\to {\cal H}^s(\Gamma)$ in \eqref{eq:S21}.  For the former, the diagonal blocks are
\begin{align*}
    {\frac{1}2 \mathcal{H}_{(1,1)}-\frac{1}{2\alpha} \mathcal{H}_{(1, -1)} } &=
    \frac1{2\alpha}\mathcal{H}_{(\alpha-1, \alpha+1)},\\
\frac{1}{4\alpha\rho\mu_1\mu_2}
\left(
\mu_1 \mathcal H_{(1,-1)}
+
\alpha\rho\mu_2 \mathcal H_{(1,1)}\right) &=
\frac{1}{4\alpha\rho\mu_1\mu_2}\mathcal H_{(\alpha\rho\mu_2+\mu_1,\alpha\rho\mu_2-\mu_1)},
\end{align*}
which are invertible by \eqref{eq:Hinvbeta}, since 
\[
(\alpha-1)^2-(\alpha+1)^2=-4\alpha\ne0, \quad (\alpha\rho\mu_2+\mu_1)^2-(\alpha\rho\mu_2-\mu_1)^2=4\alpha\rho\mu_1\mu_2\ne0.
\]

On the other hand, the eigenvalues of ${\cal S}_{21}$ are $2\sigma$ and $2\zeta$, so that ${\cal S}_{21}$ is injective if
 \[
\sigma,\zeta\ne 0.
\]
But, using \eqref{eq:ks:kp}, we have
\[
\sigma = \frac{\omega^2}{\lambda_1+2\mu_1}\Big[(\lambda_1+\mu_1)(\mu_1+\mu_2\rho)+2\mu_1\mu_2\rho\Big],\qquad
\zeta  = \frac{\alpha^2\rho\,\omega^2}{\lambda_2+2\mu_2}\Big[(\lambda_2+\mu_2)(\mu_1+\mu_2\rho)+2\mu_1\mu_2\Big],
\]
which are strictly positive since $\mu_j,\rho>0$, $\lambda_j+\mu_j>0$ and $\alpha\ne0$. Hence $\mathcal S_{21}$ is invertible.  

Since $\mathfrak P^{\rm left}$ and $\mathfrak P^{\rm right}$ are injective, Theorem \ref{theo:unicity} implies that $\mathfrak P^{\rm left}\mathfrak A_\alpha\mathfrak P^{\rm right}$ is injective, and by the Fredholm alternative it is invertible, which finishes the proof.
\end{proof}

This theorem suggests a well-posed formulation for our problem. Namely, from \eqref{eq:cfie_t}, set first
\begin{equation}\label{eq:step:A1}
\begin{bmatrix}
                                        \check{\bm{f}} \\
                                        \check{\bm{\lambda}}\\
                                       \end{bmatrix}:=
 \mathfrak{P}^{\rm left}\begin{bmatrix}
                                        \bm{f} \\
                                        \bm{\lambda} \\
                                       \end{bmatrix},
 \end{equation}
and then solve
\begin{equation}\label{eq:step:A2}
 (\mathfrak{P}^{\rm left}
 \mathfrak{A}_\alpha
 \mathfrak{P}^{\rm right})
 \begin{bmatrix}
 \check{\bm{\varphi}}\\
\check{ \bm{g}}
 \end{bmatrix}
 =
\begin{bmatrix}
                                        \check{\bm{f}} \\
                                        \check{\bm{\lambda}}\\
                                       \end{bmatrix}
%
\end{equation}
Finally, we define
\begin{equation}\label{eq:step:A3}
\begin{bmatrix}
                                        \bm{\varphi}\\
                                        \bm{g}\\
                                       \end{bmatrix}
 = \mathfrak{P}^{\rm right} \begin{bmatrix}
                                        \check{\bm{\varphi}}\\
                                       \check{ \bm{g}}\\
                                       \end{bmatrix}
 \end{equation}
and use these densities $\bm{\varphi}=(\varphi^{(p)},\varphi^{(s)})$ and ${\bm g}=(g^{(p)},g^{(s)})$ in the representation formula \eqref{eq:representation}.

Moreover, for any $s$ there exists $C>0$ such that
\[
\|\check{\bm{\varphi}}\|_{{\cal H}^s(\Gamma)}
+\|\check{\bm{g}}\|_{{\cal H}^s(\Gamma)}
\le
C\Big(
\|\check{\bm{f}}\|_{{\cal H}^s(\Gamma)}
+\|\check{\bm{\lambda}}\|_{{\cal H}^s(\Gamma)}
\Big).
\]

Returning to the original equation \eqref{eq:cfie_t}, existence of a solution follows directly from  Theorem \ref{theo:01} (uniqueness was proven in Theorem \ref{theo:unicity}) together with a  stability estimate stated in the following corollary.

\begin{corollary}\label{cor:01}
  The equation in the original (unpreconditioned) $(\nnn,\ttt)$-CFIE$_\alpha$ formulation 
 \begin{equation}\label{eq:cfie_t:2}
\mathfrak{A}_{\alpha}
\begin{bmatrix}
{\bm \varphi}\\
{\bm g}
\end{bmatrix}
=
\begin{bmatrix}
{\bm f}\\
{\bm \lambda}
\end{bmatrix}
\end{equation}
is uniquely solvable. Furthermore, if $(\bm{f},\bm{\lambda})\in {\cal H}^{s+1}(\Gamma)\times {\cal H}^{s}(\Gamma)$, then the solution $(\bm{\varphi},\bm{g})$ belongs to ${\cal H}^{s-1}(\Gamma)\times {\cal H}^{s}(\Gamma)$ and satisfies the estimate, with $C$ independent of the data. 
\[
\|\bm{\varphi}\|_{{\cal H}^{s-1}(\Gamma)}+
\|\bm{g}\|_{{\cal H}^{s}(\Gamma)}
\le
C\Big(
\|\bm{f}\|_{{\cal H}^{s+1}(\Gamma)}+\|\bm{\lambda}\|_{{\cal H}^{s}(\Gamma)}
\Big).
\]
\end{corollary}
\begin{proof} Existence follows from Theorem \ref{theo:01} and \eqref{eq:step:A3}; uniqueness was proven in Theorem~\ref{theo:unicity}.  

By construction,
\begin{align}
  \mathfrak{P}^{\mathrm{left}} &: \
  \mathcal{H}^{s+1}(\Gamma)\times \mathcal{H}^{s}(\Gamma)\to \mathcal{H}^{s-1}(\Gamma)\times \mathcal{H}^{s-1}(\Gamma), 
  \label{eq:prop:P:01}\\
     \left(\mathfrak{P}^{\mathrm{left}} \mathfrak{A}_\alpha   \mathfrak{P}^{\mathrm{right}}\right)^{-1}&: 
   \mathcal{H}^{s-1}(\Gamma)\times \mathcal{H}^{s-1}(\Gamma)\to \mathcal{H}^{s-1}(\Gamma)\times \mathcal{H}^{s-1}(\Gamma),\label{eq:prop:P:02}\\
  \mathfrak{P}^{\mathrm{right}}&:\mathcal{H}^{s-1}(\Gamma)\times \mathcal{H}^{s-1}(\Gamma)\to \mathcal{H}^{s-1}(\Gamma)\times \mathcal{H}^{s}(\Gamma)\label{eq:prop:P:03}
\end{align}
 are continuous for all $s$ with continuity constants independent of the data and denoted by $C_{\mathfrak{P}^{\rm left}}$, $C$ and $C_{\mathfrak{P}^{\rm right}}$ respectively. We stress that $\mathfrak P^{\rm left}$ acts on the data in the order $(\bm f,\bm\lambda)$ of \eqref{eq:step:A1} and $\mathfrak P^{\rm right}$ on the unknowns in the order $(\check{\bm\varphi},\check{\bm g})$ of \eqref{eq:step:A3}; in the estimates below the order of the norms is immaterial, but the Sobolev index attached to each function is not.

Now  \begin{align*}
\|{\bm{\varphi}}\|_{\mathcal{H}^{s-1}(\Gamma)}+
  \|{\bm{g}}\|_{\mathcal{H}^{s}(\Gamma)}
 &\ \le C_{ \mathfrak{P}^{\rm right}}\left[ 
\|\check{\bm{\varphi}}\|_{\mathcal{H}^{s-1}(\Gamma)}+\| \check{\bm{g}}\|_{\mathcal{H}^{s-1}(\Gamma)}
 \right]\tag*{(by  \eqref{eq:step:A3} and  \eqref{eq:prop:P:03})}\\
&\ \le  C_{ \mathfrak{P}^{\rm right}} \,C\, \left[
\|\check{\bm{f}}\|_{\mathcal{H}^{s-1}(\Gamma)}+\|\check{\bm{\lambda}}\|_{\mathcal{H}^{s-1}(\Gamma)}
\right]\tag*{(by  \eqref{eq:step:A2} and  \eqref{eq:prop:P:02})}\\ 
&\ \le
C_{ \mathfrak{P}^{\rm right}} \,C  \, C_{ \mathfrak{P}^{\rm left}}\left[
\|\bm{f}\|_{\mathcal{H}^{s+1}(\Gamma)}+\| {\bm{\lambda}}\|_{\mathcal{H}^{s}(\Gamma)} \right]. \tag*{(by   
\eqref{eq:step:A1} and \eqref{eq:prop:P:01})}
\end{align*}
which proves the result.
\end{proof}

Ultimately, the practical efficacy of these preconditioners is best demonstrated in the solution of the linear systems arising from the discretization of \eqref{eq:step:A2}. In the next section, we will present a high-order Nystr\"om discretization for the boundary operators that accommodates the proposed preconditioners. We show that the resulting linear systems can be solved efficiently via Krylov subspace methods---specifically, GMRES---in a small number of iterations, essentially independent of the mesh size.

\section{Nystr\"om discretizations and numerical results}\label{NR}
In this section we perform several numerical experiments to illustrate the performance of the proposed formulation and preconditioners. To this end, we have implemented and adapted the Nystr\"om discretization based on global trigonometric interpolation and Kussmaul--Martensen product quadrature.
Let $Q_N$ denote the trigonometric interpolation operator at the $N$ equispaced nodes $t_j=2\pi j/N$. For a kernel of the form $A(s,t)\chi(s-t)$, with $A$ smooth and $\chi$ a known $2\pi$-periodic singular function, the Nystr\"om rule reads
\[
 \int_0^{2\pi} \chi(s-t)\,A(s,t)\,\varphi(t)\,{\rm d}t\approx
 \int_0^{2\pi} \chi(s-t)\, Q_N[A(s,\cdot)\varphi](t)\,{\rm d}t,
\]
where the integral on the right-hand side is computed exactly for the singular functions $\chi$ arising from the splitting of the kernels of the four classical Helmholtz BIOs into singular and regular parts; the regular parts correspond to $\chi\equiv1$, for which the rule reduces to the trapezoidal rule. The rule is spectrally accurate for smooth $\Gamma$ and smooth densities.

This method was originally presented and analyzed by R.~Kress and collaborators \cite{Kress, KressH} for Helmholtz boundary integral equations. It has also been successfully applied, using graded meshes, to domains with corners \cite{KressCorner,dominguez2016well}, to the Navier equation \cite{chapko2000numerical, DomCat:2022}, to preconditioned Helmholtz BIOs \cite{turc2016well, Dominguez2016, boubendir2015regularized}, and within the Helmholtz decomposition framework for the Navier equation considered here (cf.\ \cite{DomCat:2022, lai2019framework}), including a rigorous stability and error analysis for the Dirichlet problem in \cite{DomCat:IMAJNumer:2025}. A major advantage of this discretization is that Fourier multiplier (i.e., periodic convolution) operators can be evaluated efficiently via standard FFT routines. This applies directly to $\mathcal{D}$, $\mathcal{H}$ and $\mathcal J$, which are the building blocks of the preconditioners of Section~\ref{sec:preconditioning}. On the other hand, the framework relies on explicit, stable decompositions of the integral kernels, which shifts the main analytical and implementation effort to the kernel splitting. This further justifies decoupling via the Helmholtz decomposition: the scalar Helmholtz kernels are considerably simpler to split than those of the full Navier system.

\subsection{Discretization and implementation}\label{sec:implementation}

Throughout this section $\Gamma$ is given by a smooth $2\pi$-periodic parameterization $\bm x:[0,2\pi)\to\Gamma$, not necessarily by arc-length, and we write $\eta(t):=|\bm x'(t)|$. Each density is approximated by its values at the $N$ equispaced nodes $t_j=2\pi j/N$, so that the discrete counterpart of a $2\times 2$ block operator such as $\mathcal C_{\rm SL}$ is a $2N\times 2N$ matrix and that of $\mathfrak A_\alpha$ is a dense complex $4N\times 4N$ matrix. This convention ($N$ points per density) is used in all tables and figures below. The Helmholtz BIOs $\V_k$, $\K_k$, $\K_k^\top$ and $\W_k$ are discretized with the Kress--Nystr\"om method described above. The remaining ingredients are treated as follows. The tangential derivative $\partial_{\bm t}=\eta^{-1}\,{\rm d}/{\rm d}t$ is implemented as the Fourier multiplier ${\rm d}/{\rm d}t$ (exact on trigonometric polynomials) combined with the pointwise multiplication by $\eta^{-1}$, and $\partial_{\bm t}^2$ as $\partial_{\bm t}\partial_{\bm t}$. The Hilbert transform $\HH$ and the mean $\J$ are the Fourier multipliers ${\rm i}\,{\rm sign}(m)$ and $\delta_{m,0}$ in the parameter $t$, which differ from their arc-length counterparts by smoothing operators only. The curvature $\kappa$ is a pointwise multiplication, i.e., a diagonal matrix, and $(\mathcal H\mathcal D+\mathcal J)^{-1}$ is computed as the inverse of the corresponding matrix. All the factors $\mathfrak A_j^{\rm left}$, $\mathfrak A_j^{\rm right}$ of Section~\ref{sec:preconditioning} are assembled as matrices and multiplied, so that $\mathfrak P^{\rm left}\mathfrak A_\alpha\mathfrak P^{\rm right}$ is available explicitly; we recall that in $\mathfrak A_3^{\rm left}$ the curvature must stay to the right of $(\mathcal H\mathcal D+\mathcal J)$, since the two do not commute. The explicit assembly is adequate for the sizes considered here ($4N\le 4096$); for larger problems the preconditioners should be applied matrix-free via FFT, which is straightforward since all their blocks are Fourier multipliers, diagonal matrices or constant combinations of $\I$ and $\HH$. No attempt has been made to optimize the implementation: our aim is to validate the formulation and the preconditioners, not to assess computational cost.

The linear systems, with and without preconditioning, are solved both by Gaussian elimination and by GMRES. GMRES is run without restart on both systems with the same stopping criterion, a relative residual below $10^{-10}$, which guarantees that the iterative error is below the discretization error in all the experiments reported: the  errors of the direct and iterative solutions agree to all displayed digits, and so do those of the preconditioned and unpreconditioned solutions. For this reason only the errors of the solutions obtained by Gaussian elimination are reported. Condition numbers are computed in the $2$-norm from the singular value decomposition of the matrices. Errors are measured at observation points away from the scatterer: we evaluate $\nabla u_p$ and $\vcurl u_s$ at equispaced points on the circle of radius $4$ and report the maximum of the Euclidean norm of the pointwise error,  $\|\nabla e_p\|$ and $\|\vcurl e_s\|$, with $e_\delta=u_\delta-u_\delta^{N}$, $\delta\in\{p,s\}$.

\paragraph{What to expect from the discrete preconditioned systems.}
Theorem~\ref{theo:01} states that $\mathfrak P^{\rm left}\mathfrak A_\alpha\mathfrak P^{\rm right}$ is the identity plus a compact operator, so that its spectrum accumulates only at $1$. For the circle this property is inherited exactly by the discrete matrices: all the operators involved are diagonal in the Fourier basis, the curvature is constant and the arc-length parameterization is the natural one, so that the cancellations of Section~\ref{sec:preconditioning} take place mode by mode also at the discrete level. For non-circular boundaries the situation is different. The multiplication operators by $\kappa$ and, for non-arc-length parameterizations, by $\eta$ couple the Fourier modes, and the highest modes representable on the grid---those near the Nyquist frequency, where the Nystr\"om discretization of $\V_k$, $\K_k$, $\W_k$ is no longer accurate---are contaminated by aliasing. The left preconditioner contains the second-order factor $(\mathcal H\mathcal D+\mathcal J)^2$, which amplifies this contamination by $N^2$. We therefore expect, for the ellipse and the cavity, that the bulk of the spectrum of the preconditioned matrices clusters at $1$, as the theory predicts, together with a small number of outlying singular values produced by the highest discrete modes and growing with $N$. Since GMRES convergence is governed by the clustering and not by the extreme singular values, this should be visible in the condition numbers but not in the iteration counts. The circle thus plays the role of a control experiment, and the comparison between the ellipse and the cavity measures the influence of the geometry.

\subsection{Geometries}\label{sec:geometries}

Three smooth scatterers are considered, all of perimeter $2\pi$ (Figure~\ref{fig:scatterers}): the unit circle, an ellipse with semi-axes in a 2:1 ratio, and a non-convex cavity-shaped domain which, despite being smooth, exhibits a more complex geometry. The circle and the ellipse are parameterized by arc-length. The cavity deserves a comment. It is a $C^\infty$ curve, so the analysis of Sections~\ref{GradHessCalculations}--\ref{sec:preconditioning} applies to it without restriction and the Nystr\"om method converges spectrally. Nevertheless, it combines several features that make it considerably harder than the ellipse for the same number of degrees of freedom: the curvature varies by two orders of magnitude and changes sign along the neck, so that the terms in $\kappa$ and $\kappa'$ that the preconditioner has to cancel are large; the neck brings two portions of $\Gamma$ close to each other, which deteriorates the accuracy of the Nystr\"om quadrature until the mesh resolves the gap; and, being non-convex, it may support nearly trapped modes for which the solution of the transmission problem is large at some frequencies. None of these features is a spurious-resonance phenomenon of the formulation---Theorem~\ref{theo:unicity} excludes those for all $\alpha\neq0$---but all of them are reflected in larger constants, larger condition numbers and more GMRES iterations than for the ellipse, both before and after preconditioning. For this geometry we do not use the arc-length parameterization but the natural one,
\[
\bm x(t)=\frac{2\pi}{L_0}\Big(\frac45\cos 2t+\frac25\cos t,\;
\tfrac{17}{48}\sin 2t+\frac38\sin 3t-\frac1{24}\sin 4t+\frac7{12}\sin t\Big),
\qquad t\in[0,2\pi),
\]
where $L_0\approx 9.2411$ is the length of the curve in brackets. For this parameterization $\eta(t)$ ranges between $0.23$ and $1.53$ and attains its minimum precisely at the two points of maximal curvature ($\kappa\approx 57$). The equispaced grid in $t$ is thus implicitly graded towards the regions of large curvature, with roughly four times the point density of the arc-length grid there, which substantially improves the resolution of the Nystr\"om method for a given $N$. The price is that $\partial_{\bm t}$, $\partial_{\bm t}^2$ and $H$ must be implemented in the parameter $t$ with the geometric factors described in Section~\ref{sec:implementation}.
\begin{figure}[htbp]

\caption{\label{fig:scatterers} Smooth scatterers used in the numerical experiments, drawn to the same scale: unit circle (left), ellipse (centre) and cavity (right). All three have perimeter $2\pi$.}
\centering
\pgfplotsset{scatterer/.style={
    width=0.33\textwidth, height=0.33\textwidth,
    xmin=-1.5, xmax=1.5, ymin=-1.5, ymax=1.5,
    grid=both, axis equal image,
    xtick={-1,0,1}, ytick={-1,0,1},
    tick label style={font=\small}}}
\begin{tikzpicture}
  \begin{axis}[scatterer]
    \addplot[domain=0:360,samples=100,variable=\t,very thick, red!70!black]
      ({cos(t)},{sin(t)});
  \end{axis}
\end{tikzpicture}\hfill
\begin{tikzpicture}
  \begin{axis}[scatterer]
    \addplot[domain=0:360,samples=100,variable=\t,very thick, red!70!black]
      ({0.6485*cos(t)},{1.2970*sin(t)});
  \end{axis}
\end{tikzpicture}\hfill
\begin{tikzpicture}
  \begin{axis}[scatterer]
    \addplot[domain=0:360,samples=200,variable=\t,very thick, red!70!black]
      ({0.6799*((4*cos(2*t))/5+(2*cos(t))/5)},
       {0.6799*((17*sin(2*t))/48+(3*sin(3*t))/8-sin(4*t)/24+(7*sin(t))/12)});
  \end{axis}
\end{tikzpicture}
\end{figure}
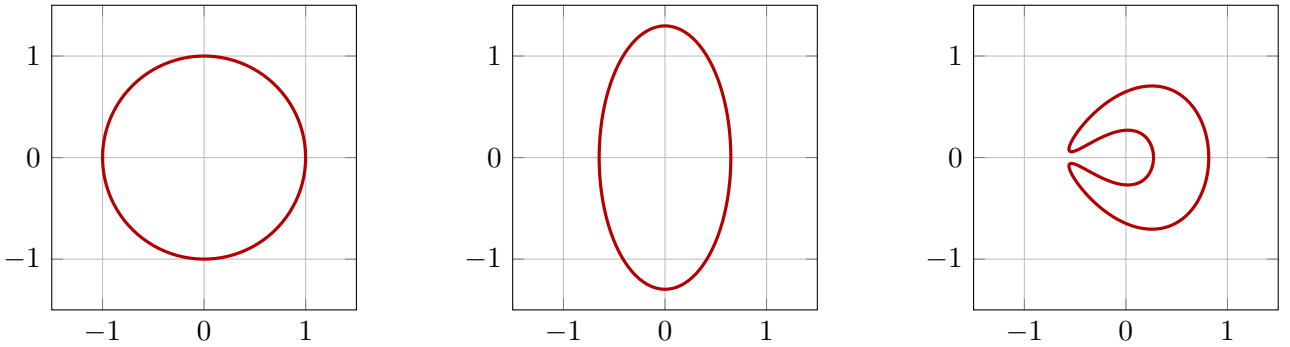


\subsection{Validation: exterior Neumann problem}\label{sec:neumann}

We begin with an illustration of the accuracy that the Nystr\"om discretization achieves when the Helmholtz decomposition approach~\eqref{eq:NH_local}, via the BIOs $\mathcal{C}_{\rm SL}$ and $\mathcal{C}_{\rm DL}$, is used for the Navier equation with Neumann boundary conditions. That is, we consider
\[
 u_p = \mathrm{SL}_{p_1}\varphi^{(p)},\ \
 u_s = \mathrm{SL}_{s_1}\varphi^{(s)}, \quad {\cal C}_{\rm SL}^{(1)}\begin{bmatrix}
                                                                  \varphi^{(p)}\\
                                                                 \varphi^{(s)}
                                                                 \end{bmatrix}
=  -\begin{bmatrix*}
  {\rm T}_1 {\bf u}^{\rm inc}\cdot\nnn\\
  {\rm T}_1 {\bf u}^{\rm inc}\cdot\ttt
  \end{bmatrix*},
\]
in the first case, and
\[
 u_p = \mathrm{DL}_{p_1}g^{(p)},\  \
 u_s = \mathrm{DL}_{s_1}g^{(s)}, \quad {\cal C}_{\rm DL}^{(1)}\begin{bmatrix}
                                                                  g^{(p)}\\
                                                                  g^{(s)}
                                                                 \end{bmatrix}
= -\begin{bmatrix*}
  {\rm T}_1 {\bf u}^{\rm inc}\cdot\nnn\\
  {\rm T}_1 {\bf u}^{\rm inc}\cdot\ttt
  \end{bmatrix*},
\]
in the second case. We recall that these formulations are defective, which is reflected in ill-conditioned matrices. The accuracy is assessed by the method of manufactured solutions: the Neumann data are produced by a point source $\x_0\in\Omega_2$ placed inside the scatterer,
 \[
    {\bf u}^{\rm inc}(\x):={\bm{\Phi}}(\x-\x_0) \bm{\beta},\quad \x\in\Gamma,
 \]
where
\[
\boldsymbol{\Phi}(\bm{r})
:=
\frac{{\rm i}}{4\mu}
H_0^{(1)}(k_s |\bm{r}|)\,I_2
+
\frac{{\rm i}}{4\omega^2}
\nabla \nabla^{\top}
\left[
H_0^{(1)}(k_s |\bm{r}|)
-
H_0^{(1)}(k_p |\bm{r}|)
\right]
\]
is the fundamental solution of the Navier equation and $\bm{\beta}=(1,-1)^\top$, so that the exterior field is known in closed form. We take $\lambda =2$, $\mu=3$, three frequencies, and measure the errors at $250$ observation points as described above. The results for the ellipse and the cavity are collected in Table~\ref{testCSL_CDL}. The errors decay rapidly as $N$ increases, consistently with the spectral convergence of the Nystr\"om method for smooth boundaries, until they stagnate at the level of $10^{-11}$--$10^{-13}$ (ellipse) and $10^{-13}$--$10^{-15}$ (cavity, thanks to the implicitly graded grid). The condition numbers of the matrices are in the range $10^{4}$--$10^{5}$: the defectiveness of the formulation does not prevent an accurate solution by Gaussian elimination, but it explains the stagnation of the error decay for the largest values of $N$.

\begin{table}[htpb]

\caption{Errors for the p-wave and s-wave exterior components at equispaced points using the Helmholtz decomposition BIE approach with single $\mathcal{C}_{\rm SL}$ and double layer $\mathcal{C}_{\rm DL}$ BIOs for an ellipse (top) and a cavity domain, with $\lambda=2$ and $\mu=3$, at various levels of discretization. $N$ is the number of quadrature points per density, so that the system size is $2N \times 2N$. \label{testCSL_CDL}}
\centering

\begin{tabular}{cc cc cc}
\toprule
\multicolumn{6}{c}{Ellipse} \\
\cmidrule(lr){1-6}
\multirow{2}{*}{$\omega$} & \multirow{2}{*}{$N$}
& \multicolumn{2}{c}{CSL (single-layer)}
& \multicolumn{2}{c}{CDL (double-layer)} \\
\cmidrule(lr){3-4} \cmidrule(lr){5-6}
& & $\|\nabla e_p\|$ & $\|\vcurl e_s\|$
  & $\|\nabla e_p\|$ & $\|\vcurl e_s\|$ \\
\midrule
\multirow{4}{*}{$8$}
 & 16  & $3.33 \times 10^{-6}$  & $1.81 \times 10^{-5}$  & $1.02 \times 10^{-5}$  & $5.33 \times 10^{-5}$ \\
 & 32  & $2.38 \times 10^{-10}$ & $1.64 \times 10^{-9}$  & $7.94 \times 10^{-10}$ & $6.31 \times 10^{-9}$ \\
 & 64  & $2.69 \times 10^{-13}$ & $1.75 \times 10^{-12}$ & $1.08 \times 10^{-12}$ & $7.44 \times 10^{-12}$ \\
 & 128 & $2.22 \times 10^{-13}$ & $1.86 \times 10^{-12}$ & $6.84 \times 10^{-13}$ & $6.10 \times 10^{-12}$ \\
\midrule
\multirow{4}{*}{$32$}
 & 16  & $2.03 \times 10^{-3}$  & $2.53 \times 10^{-2}$  & $1.70 \times 10^{-3}$  & $7.21 \times 10^{-3}$ \\
 & 32  & $1.74 \times 10^{-7}$  & $4.29 \times 10^{-6}$  & $1.18 \times 10^{-7}$  & $2.66 \times 10^{-6}$ \\
 & 64  & $6.88 \times 10^{-12}$ & $5.95 \times 10^{-11}$ & $6.59 \times 10^{-12}$ & $5.00 \times 10^{-11}$ \\
 & 128 & $4.87 \times 10^{-12}$ & $4.47 \times 10^{-11}$ & $4.51 \times 10^{-12}$ & $5.30 \times 10^{-11}$ \\
\midrule
\multirow{4}{*}{$128$}
 & 32  & $5.23 \times 10^{-3}$  & $2.77 \times 10^{-2}$  & $1.15 \times 10^{-3}$  & $7.90 \times 10^{-3}$ \\
 & 64  & $2.15 \times 10^{-11}$ & $1.78 \times 10^{-8}$  & $3.62 \times 10^{-11}$ & $1.15 \times 10^{-7}$ \\
 & 128 & $1.97 \times 10^{-11}$ & $6.14 \times 10^{-11}$ & $2.25 \times 10^{-11}$ & $8.48 \times 10^{-11}$ \\
 & 256 & $3.30 \times 10^{-11}$ & $8.06 \times 10^{-11}$ & $1.78 \times 10^{-11}$ & $1.06 \times 10^{-10}$ \\
\bottomrule
\end{tabular}

\vspace{1em}

\begin{tabular}{cc cc cc}
\toprule
\multicolumn{6}{c}{Cavity domain} \\
\cmidrule(lr){1-6}
\multirow{2}{*}{$\omega$} & \multirow{2}{*}{$N$}
& \multicolumn{2}{c}{CSL (single-layer)}
& \multicolumn{2}{c}{CDL (double-layer)} \\
\cmidrule(lr){3-4} \cmidrule(lr){5-6}
& & $\|\nabla e_p\|$ & $\|\vcurl e_s\|$
  & $\|\nabla e_p\|$ & $\|\vcurl e_s\|$ \\
\midrule
\multirow{4}{*}{$8$}
 & 128 & $6.35 \times 10^{-7}$  & $1.33 \times 10^{-7}$  & $4.57 \times 10^{-6}$  & $9.08\times 10^{-6}$ \\
 & 256  & $4.01 \times 10^{-11}$ & $1.15 \times 10^{-11}$ & $3.76 \times 10^{-11}$ & $1.09\times 10^{-10}$ \\
 & 512  & $2.50 \times 10^{-15}$   & $4.81 \times 10^{-15}$ & $2.41 \times 10^{-14}$ & $7.62\times 10^{-14}$ \\
 & 1024 & $5.11 \times 10^{-15}$   & $1.28 \times 10^{-15}$ & $7.70 \times 10^{-14}$ & $2.72\times 10^{-13}$ \\
\midrule
\multirow{4}{*}{$32$}
 & 128  & $1.02 \times 10^{-7}$  & $1.93 \times 10^{-7}$  & $2.25 \times 10^{-7}$  & $3.04 \times 10^{-7}$ \\
 & 256  & $1.41 \times 10^{-11}$ & $4.13 \times 10^{-11}$ & $1.24 \times 10^{-11}$ & $3.79 \times 10^{-11}$ \\
 & 512  & $2.71 \times 10^{-15}$ & $7.60 \times 10^{-15}$ & $2.60 \times 10^{-15}$ & $9.21 \times 10^{-15}$ \\
 & 1024 & $3.30 \times 10^{-15}$ & $6.33 \times 10^{-15}$ & $3.95 \times 10^{-15}$ & $9.72 \times 10^{-15}$ \\
\midrule
\multirow{4}{*}{$128$}
 &  256   & $2.63 \times 10^{-5}$  & $3.57 \times 10^{-5}$  & $2.46 \times 10^{-5}$  & $1.69 \times 10^{-4}$ \\
 &  512   & $1.77 \times 10^{-13}$ & $6.10 \times 10^{-13}$ & $5.47 \times 10^{-13}$ & $2.00 \times 10^{-12}$ \\
 &  1024  & $1.38 \times 10^{-13}$ & $3.99 \times 10^{-13}$ & $4.00 \times 10^{-13}$ & $1.42 \times 10^{-12}$ \\
 &  2048  & $1.04 \times 10^{-13}$ & $2.97 \times 10^{-13}$ & $2.55 \times 10^{-13}$ & $8.10 \times 10^{-13}$ \\
\bottomrule
\end{tabular}

\end{table}

\subsection{Transmission problem}\label{sec:transmission}

We now turn to the core subject of this work. In all the experiments the Lam\'e parameters are $\lambda_1=2$, $\mu_1=3$ in the exterior and $\lambda_2=3$, $\mu_2=2$ in the interior, the density ratio is $\rho=2$, and we consider two frequencies, $\omega=4$ and $\omega=64$ (shear wavelengths $2\pi\sqrt{\mu_1}/\omega\approx 2.7$ and $0.17$, i.e.\ about $2.3$ and $37$ wavelengths along $\Gamma$), and the four coupling parameters $\alpha\in\{1,-1,1/2,-1/2\}$. The incident field is a superposition of a pressure and a shear plane wave travelling in the same direction $\bm d=(0,-1)$,
\[
\bm u^{\rm inc}(\bm x)=\frac{1}{\lambda_1+2\mu_1}\,e^{{\rm i}k_p^{(1)}\bm d\cdot\bm x}\,\bm d
+\frac{\bm d^\perp\!\cdot\bm p}{\mu_1}\,e^{{\rm i}k_s^{(1)}\bm d\cdot\bm x}\,\bm d^\perp,
\qquad \bm d^\perp=(d_2,-d_1),\quad \bm p=\tfrac{1}{\sqrt2}(1,1),
\]
i.e.\ a P-wave of amplitude $1/8$ polarized along $\bm d$ and an S-wave of amplitude $1/(3\sqrt2)$ polarized along $\bm d^\perp$, and the data $(\bm f,\bm\lambda)$ in~\eqref{eq:cfie_t} are its traces in the $(\bm n,\bm t)$ frame. Since no closed-form solution is available, errors are measured against a reference solution computed with the preconditioned formulation with $\alpha=-1/2$ and Gaussian elimination on a mesh with twice as many points as the finest mesh reported; the errors are evaluated at $1240$ observation points. The errors displayed for the two finest meshes are therefore of the size of the accuracy floor of the reference solution itself and should not be read as convergence figures.

The results are collected in Table~\ref{tab:circle}. For the ellipse and the cavity, Tables~\ref{tab:ellipse-omega-4}--\ref{tab:cavity-omega-64} report, for each $\alpha$, the errors of the solutions of the preconditioned and unpreconditioned systems obtained by Gaussian elimination, the  condition numbers of both matrices and the number of GMRES iterations; an entry ``$>k$'' in the latter indicates that GMRES stagnated after $k$ iterations without reaching the tolerance. For the circle only the best-performing value $\alpha=-1/2$ is shown (Table~\ref{tab:circle}).

\subsubsection{The circle as a control experiment}
For the unit circle the discrete preconditioned matrices behave exactly as Theorem~\ref{theo:01} predicts for the operators. The condition number of $\mathfrak P^{\rm left}\mathfrak A_\alpha\mathfrak P^{\rm right}$ is independent of the discretization: about $8.4$ for $\omega=4$ and about $2.43\cdot10^3$ for $\omega=64$, the ratio reflecting the $\omega^2$ scaling of $\mathcal S_{21}$, cf.~\eqref{eq:S21}, and of the compact remainder, whereas that of the original matrix grows like $N^4$, from $2\cdot10^5$ to $10^{10}$ at $\omega=4$ and up to $7\cdot10^8$ at $\omega=64$. The number of GMRES iterations for the preconditioned system is $15$ for all $N$ at $\omega=4$ and $105$--$107$ for all $N\ge256$ at $\omega=64$, against $99$--$1060$ and $253$--$2206$, growing linearly with $N$, for the unpreconditioned system. The dependence on $\alpha$ is already visible here: for $\omega=64$ the preconditioned solver needs $354$, $353$, $141$ and $107$ iterations for $\alpha=1,\tfrac12,-1,-\tfrac12$ respectively (not shown), while the unpreconditioned counts are insensitive to $\alpha$. The errors reach the accuracy floor of the reference solution ($4\cdot10^{-10}$ at $\omega=4$ and $1.5\cdot10^{-11}$ at $\omega=64$) for $N\ge 32$ and $N\ge 256$ respectively, i.e.\ with about seven points per shear wavelength at the higher frequency.
\begin{table}[htbp]
\centering
\caption{Errors, matrix condition numbers, and number of GMRES iterations for the circle experiment ($\lambda_1 = 2$, $\mu_1 = 3$, $\lambda_2 = 3$, $\mu_2 = 2$, $\rho = 2$) with $\omega=4$ (top) and $\omega=64$ (bottom). The coupling parameter is $\alpha=-1/2$. $N$ is the number of quadrature points per density; the size of the linear system is $4N \times 4N$. GMRES is run without restart with a relative residual tolerance of $10^{-10}$ for both systems. In each block, ``Prec.'' refers to the preconditioned $(\nnn,\ttt)$-CFIE$_\alpha$ and ``Orig.'' to the $(\nnn,\ttt)$-CFIE$_\alpha$ itself.}
\label{tab:circle}

\setlength{\tabcolsep}{3.5pt}
\renewcommand{\arraystretch}{1.15}
\begin{subtable}{\textwidth}
\centering
\caption{$\omega=4$}

\resizebox{\textwidth}{!}{%
\begin{tabular}{
  r
  @{\hspace{2\tabcolsep}}
  *{4}{S[table-format=1.2e-2]}
  *{2}{S[table-format=1.2e+1,retain-zero-exponent = true]}
  *{2}{c}
}
\toprule
& \multicolumn{2}{c}{Prec. $(\nnn,\ttt)$-CFIE$_\alpha$}
& \multicolumn{2}{c}{$(\nnn,\ttt)$-CFIE$_\alpha$}
& \multicolumn{2}{c}{Condition number}
& \multicolumn{2}{c}{GMRES iterations} \\
\cmidrule(lr){2-3}
\cmidrule(lr){4-5}
\cmidrule(lr){6-7}
\cmidrule(lr){8-9}
$N$
& \multicolumn{1}{c}{$\lVert\nabla e_p\rVert$}
& \multicolumn{1}{c}{$\lVert\operatorname{curl} e_s\rVert$}
& \multicolumn{1}{c}{$\lVert\nabla e_p\rVert$}
& \multicolumn{1}{c}{$\lVert\operatorname{curl} e_s\rVert$}
& \multicolumn{1}{c}{Prec.}
& \multicolumn{1}{c}{Orig.}
& \multicolumn{1}{c}{Prec.}
& \multicolumn{1}{c}{Orig.} \\
\midrule
32
& 4.03e-10  & 7.79e-10
& 4.03e-10  & 7.79e-10
& 8.44e0   & 1.95e5
& 15      & 99 \\

64
& 4.03e-10  & 7.79e-10
& 4.03e-10  & 7.79e-10
& 8.44e0   & 2.98e6
& 15      & 163 \\

128
& 4.04e-10  & 7.81e-10
& 4.03e-10  & 7.79e-10
& 8.44e0   & 4.63e7
& 15      & 292 \\

256
& 4.02e-10  & 7.79e-10
& 4.03e-10  & 7.79e-10
& 8.44e0   & 7.31e8
& 15      & 548 \\

512
& 4.06e-10  & 6.12e-10
& 4.03e-10  & 7.79e-10
& 8.44e0   & 1.16e10
& 15      & 1060 \\

\bottomrule
\end{tabular}%
}
\end{subtable}

\medskip
\begin{subtable}{\textwidth}
\centering
\caption{$\omega=64$}

\resizebox{\textwidth}{!}{%
\begin{tabular}{
  r
  @{\hspace{2\tabcolsep}}
  *{4}{S[table-format=1.2e-2]}
  *{2}{S[table-format=1.2e+1]}
  *{2}{c}
}
\toprule
& \multicolumn{2}{c}{Prec. $(\nnn,\ttt)$-CFIE$_\alpha$}
& \multicolumn{2}{c}{$(\nnn,\ttt)$-CFIE$_\alpha$}
& \multicolumn{2}{c}{Condition number}
& \multicolumn{2}{c}{GMRES iterations} \\
\cmidrule(lr){2-3}
\cmidrule(lr){4-5}
\cmidrule(lr){6-7}
\cmidrule(lr){8-9}
$N$
& \multicolumn{1}{c}{$\lVert\nabla e_p\rVert$}
& \multicolumn{1}{c}{$\lVert\operatorname{curl} e_s\rVert$}
& \multicolumn{1}{c}{$\lVert\nabla e_p\rVert$}
& \multicolumn{1}{c}{$\lVert\operatorname{curl} e_s\rVert$}
& \multicolumn{1}{c}{Prec.}
& \multicolumn{1}{c}{Orig.}
& \multicolumn{1}{c}{Prec.}
& \multicolumn{1}{c}{Orig.} \\
\midrule
64
& 6.78e-2  & 5.85e-1
& 6.78e-2  & 5.85e-1
& 2.45e3   & 3.60e4
& 147      & 253 \\

128
& 1.65e-6  & 5.06e-4
& 1.65e-6  & 5.06e-4
& 2.43e3   & 4.54e5
& 120      & 432 \\

256
& 1.55e-11  & 8.59e-12
& 1.55e-11  & 8.60e-12
& 2.43e3   & 2.61e6
& 107      & 670 \\

512
& 1.58e-11  & 8.61e-12
& 1.55e-11  & 8.62e-12
& 2.43e3   & 4.43e7
& 105      & 1182 \\

1024
& 1.78e-11  & 8.41e-12
& 1.55e-11  & 8.63e-12
& 2.43e3   & 7.18e8
& 107      & 2206 \\

\bottomrule
\end{tabular}%
}
\end{subtable}

\end{table}

\subsubsection{The ellipse and the cavity}

\paragraph{Accuracy.} For the ellipse and the cavity we observe the typical behaviour of a spectral method: fast convergence once the mesh resolves the solution, followed by stagnation at the accuracy floor. For the ellipse the floor is $10^{-11}$ ($\omega=4$) and $2\cdot10^{-10}$ ($\omega=64$); for the cavity it is $3\cdot10^{-9}$ and $7\cdot10^{-11}$, two orders of magnitude higher than for the ellipse at the lower frequency, which we attribute to roundoff amplified by the much larger condition numbers of the cavity matrices. In all cases the errors of the preconditioned and unpreconditioned solutions agree to all displayed digits (and so do those of the GMRES solutions). The ill-conditioning of the unpreconditioned matrices, with condition numbers up to $2.6\cdot10^{12}$, therefore has a negligible effect on the accuracy when the system is solved by Gaussian elimination. This can be explained by the nature of the data: the extreme singular values of the matrix correspond to highly oscillatory singular vectors, on which the smooth incident field has a negligible projection, so that the ill-conditioning is never excited (see also Section~\ref{sec:spectrum}).

\paragraph{GMRES iterations.} Without preconditioning the number of GMRES iterations grows linearly with $N$ for all geometries and frequencies---it is essentially equal to the number of degrees of freedom divided by a constant---reaching $1220$ for the ellipse ($N=512$, $\omega=64$) and more than $2700$ for the cavity ($N=1024$, $\omega=64$); for the cavity at $\omega=4$ GMRES even stagnates before reaching the tolerance on the finest meshes. This is the expected behaviour for an operator that is not a compact perturbation of an invertible one. With preconditioning the number of iterations becomes independent of $N$ as soon as the mesh resolves the solution (the only exception being $\alpha=1$ on the finest mesh of the cavity at $\omega=4$, where the preconditioned GMRES also stagnates, at a residual close to the tolerance): with $\alpha=-1/2$, $23$ iterations for the ellipse at $\omega=4$ and $94$ at $\omega=64$, and about $100$ and $200$ for the cavity at $\omega=4$ and $\omega=64$ (with a very slow growth, from $73$ to $108$, at the lower frequency). The gain is substantial for all geometries and frequencies, but the sign of $\alpha$ matters: at $\omega=64$, negative values of $\alpha$ require three to four times fewer iterations than positive ones for the circle and the ellipse ($354$ vs.\ $107$ and $411$ vs.\ $94$ for $\alpha=1$ vs.\ $\alpha=-1/2$) and about twice fewer for the cavity ($459$ vs.\ $200$), while at $\omega=4$ the difference is smaller ($34$ vs.\ $23$ for the ellipse) or negligible (cavity). In all cases $\alpha=-1/2$ performs best or on a par with $\alpha=-1$. A theoretical explanation of this dependence, and the determination of an optimal $\alpha$ as a function of $(\lambda_j,\mu_j)$, $\rho$ and $\omega$, is left for future work.

\paragraph{Condition numbers.} Here the three geometries behave differently, and the differences are instructive. For the circle, as discussed above, the condition number of the preconditioned matrix is bounded independently of $N$. For the ellipse at $\omega=64$ it is also essentially bounded, about $2.2\cdot10^3$ for $N\ge128$ with $\alpha=-1/2$, i.e.\ of the same size as for the circle, while that of the original matrix grows from $10^5$ to $4.4\cdot10^7$; at $\omega=4$ it grows moderately, from $2\cdot10^2$ to $1.6\cdot10^4$, against $2\cdot10^5$ to $7\cdot10^8$ for the original matrix. For the cavity the condition numbers of the preconditioned matrices are large and grow with $N$: from $10^7$ to $1.3\cdot10^{10}$ at $\omega=4$ and between $5\cdot10^5$ and $9\cdot10^7$ at $\omega=64$ (for $\alpha=-1/2$); on the coarsest mesh at $\omega=4$ they even exceed those of the original matrices, which themselves reach $2.6\cdot10^{12}$. Yet the GMRES iteration counts of the preconditioned systems are essentially constant in $N$ for the cavity as well. This is precisely the situation anticipated in Section~\ref{sec:implementation}: the preconditioner clusters the bulk of the spectrum at $1$, which is what GMRES sees, while a small number of outlying singular values, produced by the discrete highest modes through the interaction of the second-order factor $(\mathcal H\mathcal D+\mathcal J)^2$ with the strongly varying coefficients $\kappa$ and $\eta$ of the cavity, grow with $N$ and dominate the condition number. The singular value and eigenvalue distributions of the next section confirm this picture, see Figures~\ref{fig:eigenvalues:cavity} and \ref{fig:svd:cavity}. It also indicates that the condition number alone is a poor indicator of the quality of an analytical preconditioner of this type on complex geometries, and that the spectral distribution should be examined instead.

\begin{table}[htbp]
\centering

\caption{Errors,  matrix condition numbers, and number of GMRES iterations for the ellipse experiment ($\lambda_1 = 2$, $\mu_1 = 3$, $\lambda_2 = 3$, $\mu_2 = 2$, $\rho = 2$) with $\omega=4$. The coupling parameter $\alpha$ in the formulation takes the values $1, -1, 1/2, -1/2$. $N$ is the number of quadrature points per density; the size of the linear system is $4N \times 4N$. GMRES is run without restart with a relative residual tolerance of $10^{-10}$ for both systems. In each block, ``Prec.'' refers to the preconditioned $(\nnn,\ttt)$-CFIE$_\alpha$ and ``Orig.'' to the $(\nnn,\ttt)$-CFIE$_\alpha$ itself.}
\label{tab:ellipse-omega-4}

\setlength{\tabcolsep}{3.5pt}
\renewcommand{\arraystretch}{1.15}

\begin{subtable}{\textwidth}
\centering
\caption{$\alpha=1$}

\resizebox{\textwidth}{!}{%
\begin{tabular}{
  r
  @{\hspace{2\tabcolsep}}
  *{4}{S[table-format=1.2e-2]}
  *{2}{S[table-format=1.2e+1]}
  *{2}{c}
}
\toprule
& \multicolumn{2}{c}{Prec. $(\nnn,\ttt)$-CFIE$_\alpha$}
& \multicolumn{2}{c}{$(\nnn,\ttt)$-CFIE$_\alpha$}
& \multicolumn{2}{c}{Condition number}
& \multicolumn{2}{c}{GMRES iterations} \\
\cmidrule(lr){2-3}
\cmidrule(lr){4-5}
\cmidrule(lr){6-7}
\cmidrule(lr){8-9}
$N$
& \multicolumn{1}{c}{$\lVert\nabla e_p\rVert$}
& \multicolumn{1}{c}{$\lVert\operatorname{curl} e_s\rVert$}
& \multicolumn{1}{c}{$\lVert\nabla e_p\rVert$}
& \multicolumn{1}{c}{$\lVert\operatorname{curl} e_s\rVert$}
& \multicolumn{1}{c}{Prec.}
& \multicolumn{1}{c}{Orig.}
& \multicolumn{1}{c}{Prec.}
& \multicolumn{1}{c}{Orig.} \\
\midrule
32
& 1.81e-5  & 4.38e-5
& 1.81e-5  & 4.38e-5
& 2.70e1   & 1.85e5
& 35      & 128 \\

64
& 2.11e-9  & 5.18e-9
& 2.11e-9  & 5.18e-9
& 2.70e1   & 3.20e6
& 34      & 222 \\

128
& 1.21e-11  & 3.12e-11
& 1.20e-11  & 3.13e-11
& 7.60e1   & 5.29e7
& 34      & 362 \\

256
& 1.22e-11  & 2.98e-11
& 1.24e-11  & 2.97e-11
& 3.02e2   & 8.60e8
& 34      & 616 \\

\bottomrule
\end{tabular}%
}
\end{subtable}

\medskip

\begin{subtable}{\textwidth}
\centering
\caption{$\alpha=-1$}

\resizebox{\textwidth}{!}{%
\begin{tabular}{
  r
  @{\hspace{2\tabcolsep}}
  *{4}{S[table-format=1.2e-2]}
  *{2}{S[table-format=1.2e+1]}
  *{2}{c}
}
\toprule
& \multicolumn{2}{c}{Prec. $(\nnn,\ttt)$-CFIE$_\alpha$}
& \multicolumn{2}{c}{$(\nnn,\ttt)$-CFIE$_\alpha$}
& \multicolumn{2}{c}{Condition number}
& \multicolumn{2}{c}{GMRES iterations} \\
\cmidrule(lr){2-3}
\cmidrule(lr){4-5}
\cmidrule(lr){6-7}
\cmidrule(lr){8-9}
$N$
& \multicolumn{1}{c}{$\lVert\nabla e_p\rVert$}
& \multicolumn{1}{c}{$\lVert\operatorname{curl} e_s\rVert$}
& \multicolumn{1}{c}{$\lVert\nabla e_p\rVert$}
& \multicolumn{1}{c}{$\lVert\operatorname{curl} e_s\rVert$}
& \multicolumn{1}{c}{Prec.}
& \multicolumn{1}{c}{Orig.}
& \multicolumn{1}{c}{Prec.}
& \multicolumn{1}{c}{Orig.} \\
\midrule
32
& 1.58e-5  & 4.69e-5
& 1.58e-5  & 4.69e-5
& 2.14e2   & 1.84e5
& 23      & 127 \\

64
& 1.83e-9  & 5.38e-9
& 1.83e-9  & 5.38e-9
& 8.84e2   & 3.19e6
& 23      & 222 \\

128
& 1.22e-11  & 3.13e-11
& 1.20e-11  & 3.16e-11
& 3.62e3   & 5.29e7
& 23      & 362 \\

256
& 1.31e-11  & 3.17e-11
& 1.22e-11  & 3.00e-11
& 1.47e4   & 8.60e8
& 23      & 616 \\

\bottomrule
\end{tabular}%
}
\end{subtable}

\medskip

\begin{subtable}{\textwidth}
\centering
\caption{$\alpha=\frac{1}{2}$}

\resizebox{\textwidth}{!}{%
\begin{tabular}{
  r
  @{\hspace{2\tabcolsep}}
  *{4}{S[table-format=1.2e-2]}
  *{2}{S[table-format=1.2e+1]}
  *{2}{c}
}
\toprule
& \multicolumn{2}{c}{Prec. $(\nnn,\ttt)$-CFIE$_\alpha$}
& \multicolumn{2}{c}{$(\nnn,\ttt)$-CFIE$_\alpha$}
& \multicolumn{2}{c}{Condition number}
& \multicolumn{2}{c}{GMRES iterations} \\
\cmidrule(lr){2-3}
\cmidrule(lr){4-5}
\cmidrule(lr){6-7}
\cmidrule(lr){8-9}
$N$
& \multicolumn{1}{c}{$\lVert\nabla e_p\rVert$}
& \multicolumn{1}{c}{$\lVert\operatorname{curl} e_s\rVert$}
& \multicolumn{1}{c}{$\lVert\nabla e_p\rVert$}
& \multicolumn{1}{c}{$\lVert\operatorname{curl} e_s\rVert$}
& \multicolumn{1}{c}{Prec.}
& \multicolumn{1}{c}{Orig.}
& \multicolumn{1}{c}{Prec.}
& \multicolumn{1}{c}{Orig.} \\
\midrule
32
& 1.87e-5  & 4.52e-5
& 1.87e-5  & 4.52e-5
& 2.97e1   & 1.98e5
& 35      & 128 \\

64
& 2.18e-9  & 5.33e-9
& 2.18e-9  & 5.33e-9
& 4.07e1   & 3.00e6
& 34      & 222 \\

128
& 1.20e-11  & 3.17e-11
& 1.20e-11  & 3.15e-11
& 1.61e2   & 4.65e7
& 34      & 362 \\

256
& 1.31e-11  & 3.37e-11
& 1.24e-11  & 2.95e-11
& 6.45e2   & 7.32e8
& 34      & 616 \\

\bottomrule
\end{tabular}%
}
\end{subtable}

\medskip

\begin{subtable}{\textwidth}
\centering
\caption{$\alpha=-\frac{1}{2}$}

\resizebox{\textwidth}{!}{%
\begin{tabular}{
  r
  @{\hspace{2\tabcolsep}}
  *{4}{S[table-format=1.2e-2]}
  *{2}{S[table-format=1.2e+1]}
  *{2}{c}
}
\toprule
& \multicolumn{2}{c}{Prec. $(\nnn,\ttt)$-CFIE$_\alpha$}
& \multicolumn{2}{c}{$(\nnn,\ttt)$-CFIE$_\alpha$}
& \multicolumn{2}{c}{Condition number}
& \multicolumn{2}{c}{GMRES iterations} \\
\cmidrule(lr){2-3}
\cmidrule(lr){4-5}
\cmidrule(lr){6-7}
\cmidrule(lr){8-9}
$N$
& \multicolumn{1}{c}{$\lVert\nabla e_p\rVert$}
& \multicolumn{1}{c}{$\lVert\operatorname{curl} e_s\rVert$}
& \multicolumn{1}{c}{$\lVert\nabla e_p\rVert$}
& \multicolumn{1}{c}{$\lVert\operatorname{curl} e_s\rVert$}
& \multicolumn{1}{c}{Prec.}
& \multicolumn{1}{c}{Orig.}
& \multicolumn{1}{c}{Prec.}
& \multicolumn{1}{c}{Orig.} \\
\midrule
32
& 1.52e-5  & 4.68e-5
& 1.52e-5  & 4.68e-5
& 2.35e2   & 1.97e5
& 23      & 126 \\

64
& 1.77e-9  & 5.30e-9
& 1.77e-9  & 5.30e-9
& 9.69e2   & 2.99e6
& 23      & 222 \\

128
& 1.22e-11  & 3.11e-11
& 1.21e-11  & 3.13e-11
& 3.97e3   & 4.65e7
& 23      & 362 \\

256
& 1.20e-11  & 3.28e-11
& 1.23e-11  & 3.03e-11
& 1.61e4   & 7.32e8
& 23      & 616 \\

\bottomrule
\end{tabular}%
}
\end{subtable}

\end{table}

\begin{table}[htbp]
\centering

\caption{Errors, matrix condition numbers, and number of GMRES iterations for the ellipse experiment ($\lambda_1 = 2$, $\mu_1 = 3$, $\lambda_2 = 3$, $\mu_2 = 2$, $\rho = 2$) with $\omega=64$. The coupling parameter $\alpha$ in the formulation takes the values $1, -1, 1/2, -1/2$. $N$ is the number of quadrature points per density; the size of the linear system is $4N \times 4N$. GMRES is run without restart with a relative residual tolerance of $10^{-10}$ for both systems. In each block, ``Prec.'' refers to the preconditioned $(\nnn,\ttt)$-CFIE$_\alpha$ and ``Orig.'' to the $(\nnn,\ttt)$-CFIE$_\alpha$ itself.}
\label{tab:ellipse-omega-64}

\setlength{\tabcolsep}{3.5pt}
\renewcommand{\arraystretch}{1.15}

\begin{subtable}{\textwidth}
\centering
\caption{$\alpha=1$}

\resizebox{\textwidth}{!}{%
\begin{tabular}{
  r
  @{\hspace{2\tabcolsep}}
  *{4}{S[table-format=1.2e-2]}
  *{2}{S[table-format=1.2e+1]}
  *{2}{c}
}
\toprule
& \multicolumn{2}{c}{Prec. $(\nnn,\ttt)$-CFIE$_\alpha$}
& \multicolumn{2}{c}{$(\nnn,\ttt)$-CFIE$_\alpha$}
& \multicolumn{2}{c}{Condition number}
& \multicolumn{2}{c}{GMRES iterations} \\
\cmidrule(lr){2-3}
\cmidrule(lr){4-5}
\cmidrule(lr){6-7}
\cmidrule(lr){8-9}
$N$
& \multicolumn{1}{c}{$\lVert\nabla e_p\rVert$}
& \multicolumn{1}{c}{$\lVert\operatorname{curl} e_s\rVert$}
& \multicolumn{1}{c}{$\lVert\nabla e_p\rVert$}
& \multicolumn{1}{c}{$\lVert\operatorname{curl} e_s\rVert$}
& \multicolumn{1}{c}{Prec.}
& \multicolumn{1}{c}{Orig.}
& \multicolumn{1}{c}{Prec.}
& \multicolumn{1}{c}{Orig.} \\
\midrule
64
& 2.49e-2  & 4.92e-1
& 2.49e-2  & 4.92e-1
& 3.11e3   & 2.53e4
& 252      & 250 \\

128
& 1.58e-5  & 1.89e-4
& 1.58e-5  & 1.89e-4
& 3.61e3   & 1.36e5
& 413      & 452 \\

256
& 2.35e-10  & 3.72e-10
& 2.35e-10  & 3.72e-10
& 3.61e3   & 3.05e6
& 411      & 708 \\

512
& 2.01e-10  & 4.03e-10
& 2.01e-10  & 4.03e-10
& 3.61e3   & 5.30e7
& 411      & 1220 \\

\bottomrule
\end{tabular}%
}
\end{subtable}

\medskip

\begin{subtable}{\textwidth}
\centering
\caption{$\alpha=-1$}

\resizebox{\textwidth}{!}{%
\begin{tabular}{
  r
  @{\hspace{2\tabcolsep}}
  *{4}{S[table-format=1.2e-2]}
  *{2}{S[table-format=1.2e+1]}
  *{2}{c}
}
\toprule
& \multicolumn{2}{c}{Prec. $(\nnn,\ttt)$-CFIE$_\alpha$}
& \multicolumn{2}{c}{$(\nnn,\ttt)$-CFIE$_\alpha$}
& \multicolumn{2}{c}{Condition number}
& \multicolumn{2}{c}{GMRES iterations} \\
\cmidrule(lr){2-3}
\cmidrule(lr){4-5}
\cmidrule(lr){6-7}
\cmidrule(lr){8-9}
$N$
& \multicolumn{1}{c}{$\lVert\nabla e_p\rVert$}
& \multicolumn{1}{c}{$\lVert\operatorname{curl} e_s\rVert$}
& \multicolumn{1}{c}{$\lVert\nabla e_p\rVert$}
& \multicolumn{1}{c}{$\lVert\operatorname{curl} e_s\rVert$}
& \multicolumn{1}{c}{Prec.}
& \multicolumn{1}{c}{Orig.}
& \multicolumn{1}{c}{Prec.}
& \multicolumn{1}{c}{Orig.} \\
\midrule
64
& 1.01e-1  & 6.59e-1
& 1.01e-1  & 6.59e-1
& 2.14e3   & 8.60e4
& 146      & 250 \\

128
& 2.47e-5  & 4.51e-4
& 2.47e-5  & 4.51e-4
& 1.90e3   & 1.37e5
& 129      & 433 \\

256
& 1.78e-10  & 5.12e-10
& 1.78e-10  & 5.12e-10
& 1.90e3   & 3.03e6
& 122      & 682 \\

512
& 1.70e-10  & 3.26e-10
& 1.70e-10  & 3.26e-10
& 1.90e3   & 5.28e7
& 122      & 1194 \\

\bottomrule
\end{tabular}%
}
\end{subtable}

\medskip

\begin{subtable}{\textwidth}
\centering
\caption{$\alpha=\frac{1}{2}$}

\resizebox{\textwidth}{!}{%
\begin{tabular}{
  r
  @{\hspace{2\tabcolsep}}
  *{4}{S[table-format=1.2e-2]}
  *{2}{S[table-format=1.2e+1]}
  *{2}{c}
}
\toprule
& \multicolumn{2}{c}{Prec. $(\nnn,\ttt)$-CFIE$_\alpha$}
& \multicolumn{2}{c}{$(\nnn,\ttt)$-CFIE$_\alpha$}
& \multicolumn{2}{c}{Condition number}
& \multicolumn{2}{c}{GMRES iterations} \\
\cmidrule(lr){2-3}
\cmidrule(lr){4-5}
\cmidrule(lr){6-7}
\cmidrule(lr){8-9}
$N$
& \multicolumn{1}{c}{$\lVert\nabla e_p\rVert$}
& \multicolumn{1}{c}{$\lVert\operatorname{curl} e_s\rVert$}
& \multicolumn{1}{c}{$\lVert\nabla e_p\rVert$}
& \multicolumn{1}{c}{$\lVert\operatorname{curl} e_s\rVert$}
& \multicolumn{1}{c}{Prec.}
& \multicolumn{1}{c}{Orig.}
& \multicolumn{1}{c}{Prec.}
& \multicolumn{1}{c}{Orig.} \\
\midrule
64
& 1.81e-2  & 5.00e-1
& 1.81e-2  & 5.00e-1
& 3.51e3   & 2.72e4
& 251      & 255 \\

128
& 1.44e-5  & 1.96e-4
& 1.44e-5  & 1.96e-4
& 3.85e3   & 1.48e5
& 409      & 445 \\

256
& 2.23e-10  & 3.91e-10
& 2.23e-10  & 3.91e-10
& 3.85e3   & 2.70e6
& 405      & 701 \\

512
& 2.20e-10  & 4.72e-10
& 2.20e-10  & 4.72e-10
& 3.85e3   & 4.47e7
& 405      & 1213 \\

\bottomrule
\end{tabular}%
}
\end{subtable}

\medskip

\begin{subtable}{\textwidth}
\centering
\caption{$\alpha=-\frac{1}{2}$}

\resizebox{\textwidth}{!}{%
\begin{tabular}{
  r
  @{\hspace{2\tabcolsep}}
  *{4}{S[table-format=1.2e-2]}
  *{2}{S[table-format=1.2e+1]}
  *{2}{c}
}
\toprule
& \multicolumn{2}{c}{Prec. $(\nnn,\ttt)$-CFIE$_\alpha$}
& \multicolumn{2}{c}{$(\nnn,\ttt)$-CFIE$_\alpha$}
& \multicolumn{2}{c}{Condition number}
& \multicolumn{2}{c}{GMRES iterations} \\
\cmidrule(lr){2-3}
\cmidrule(lr){4-5}
\cmidrule(lr){6-7}
\cmidrule(lr){8-9}
$N$
& \multicolumn{1}{c}{$\lVert\nabla e_p\rVert$}
& \multicolumn{1}{c}{$\lVert\operatorname{curl} e_s\rVert$}
& \multicolumn{1}{c}{$\lVert\nabla e_p\rVert$}
& \multicolumn{1}{c}{$\lVert\operatorname{curl} e_s\rVert$}
& \multicolumn{1}{c}{Prec.}
& \multicolumn{1}{c}{Orig.}
& \multicolumn{1}{c}{Prec.}
& \multicolumn{1}{c}{Orig.} \\
\midrule
64
& 9.22e-2  & 4.66e-1
& 9.22e-2  & 4.66e-1
& 4.17e3   & 9.95e4
& 122      & 254 \\

128
& 1.85e-5  & 3.44e-4
& 1.85e-5  & 3.44e-4
& 2.19e3   & 1.52e5
& 97      & 435 \\

256
& 1.87e-10  & 4.83e-10
& 1.87e-10  & 4.83e-10
& 2.19e3   & 2.61e6
& 94      & 684 \\

512
& 1.52e-10  & 3.27e-10
& 1.52e-10  & 3.27e-10
& 2.19e3   & 4.44e7
& 94      & 1196 \\

\bottomrule
\end{tabular}%
}
\end{subtable}

\end{table}

\begin{table}[htbp]
\centering

\caption{Errors,   matrix condition numbers, and number of GMRES iterations for the cavity experiment ($\lambda_1 = 2$, $\mu_1 = 3$, $\lambda_2 = 3$, $\mu_2 = 2$, $\rho = 2$) with $\omega=4$. The coupling parameter $\alpha$ in the formulation takes the values $1, -1, 1/2, -1/2$. $N$ is the number of quadrature points per density; the size of the linear system is $4N \times 4N$. GMRES is run without restart with a relative residual tolerance of $10^{-10}$ for both systems. In each block, ``Prec.'' refers to the preconditioned $(\nnn,\ttt)$-CFIE$_\alpha$ and ``Orig.'' to the $(\nnn,\ttt)$-CFIE$_\alpha$ itself.}
\label{tab:cavity-omega-4}

\setlength{\tabcolsep}{3.5pt}
\renewcommand{\arraystretch}{1.15}


\begin{subtable}{\textwidth}
\centering
\caption{$\alpha=1$}

\resizebox{\textwidth}{!}{%
\begin{tabular}{
  r
  @{\hspace{2\tabcolsep}}
  *{4}{S[table-format=1.2e-2]}
  *{2}{S[table-format=1.2e+1]}
  *{2}{c}
}
\toprule
& \multicolumn{2}{c}{Prec. $(\nnn,\ttt)$-CFIE$_\alpha$}
& \multicolumn{2}{c}{$(\nnn,\ttt)$-CFIE$_\alpha$}
& \multicolumn{2}{c}{Condition number}
& \multicolumn{2}{c}{GMRES iterations} \\
\cmidrule(lr){2-3}
\cmidrule(lr){4-5}
\cmidrule(lr){6-7}
\cmidrule(lr){8-9}
$N$
& \multicolumn{1}{c}{$\lVert\nabla e_p\rVert$}
& \multicolumn{1}{c}{$\lVert\operatorname{curl} e_s\rVert$}
& \multicolumn{1}{c}{$\lVert\nabla e_p\rVert$}
& \multicolumn{1}{c}{$\lVert\operatorname{curl} e_s\rVert$}
& \multicolumn{1}{c}{Prec.}
& \multicolumn{1}{c}{Orig.}
& \multicolumn{1}{c}{Prec.}
& \multicolumn{1}{c}{Orig.} \\
\midrule
32
& 7.37e-4  & 2.77e-3
& 7.37e-4  & 2.77e-3
& 1.39e5   & 3.06e6
& 87      & 128 \\

64
& 3.36e-4  & 1.01e-3
& 3.36e-4  & 1.01e-3
& 7.60e5   & 2.02e7
& 89      & 256 \\

128
& 5.10e-5  & 1.56e-4
& 5.10e-5  & 1.56e-4
& 1.48e7   & 7.83e8
& 102      & 490 \\

256
& 3.57e-9  & 1.21e-8
& 3.58e-9  & 1.22e-8
& 1.55e8   & 4.08e10
& 108      & 950 \\

512
& 3.35e-9  & 8.32e-9
& 3.30e-9  & 8.22e-9
& 5.53e9   & 2.01e12
& {$>$119}      & {$>$1754} \\

\bottomrule
\end{tabular}%
}
\end{subtable}

\medskip

\begin{subtable}{\textwidth}
\centering
\caption{$\alpha=-1$}

\resizebox{\textwidth}{!}{%
\begin{tabular}{
  r
  @{\hspace{2\tabcolsep}}
  *{4}{S[table-format=1.2e-2]}
  *{2}{S[table-format=1.2e+1]}
  *{2}{c}
}
\toprule
& \multicolumn{2}{c}{Prec. $(\nnn,\ttt)$-CFIE$_\alpha$}
& \multicolumn{2}{c}{$(\nnn,\ttt)$-CFIE$_\alpha$}
& \multicolumn{2}{c}{Condition number}
& \multicolumn{2}{c}{GMRES iterations} \\
\cmidrule(lr){2-3}
\cmidrule(lr){4-5}
\cmidrule(lr){6-7}
\cmidrule(lr){8-9}
$N$
& \multicolumn{1}{c}{$\lVert\nabla e_p\rVert$}
& \multicolumn{1}{c}{$\lVert\operatorname{curl} e_s\rVert$}
& \multicolumn{1}{c}{$\lVert\nabla e_p\rVert$}
& \multicolumn{1}{c}{$\lVert\operatorname{curl} e_s\rVert$}
& \multicolumn{1}{c}{Prec.}
& \multicolumn{1}{c}{Orig.}
& \multicolumn{1}{c}{Prec.}
& \multicolumn{1}{c}{Orig.} \\
\midrule
32
& 8.66e-4  & 2.22e-3
& 8.66e-4  & 2.22e-3
& 6.33e6   & 5.37e6
& 67      & 128 \\

64
& 3.34e-4  & 1.00e-3
& 3.34e-4  & 1.00e-3
& 1.55e7   & 3.33e7
& 76      & 256 \\

128
& 5.10e-5  & 1.56e-4
& 5.10e-5  & 1.56e-4
& 1.49e8   & 1.21e9
& 90      & 485 \\

256
& 3.53e-9  & 1.22e-8
& 3.58e-9  & 1.22e-8
& 9.14e8   & 5.46e10
& 97      & 948 \\

512
& 3.57e-9  & 8.72e-9
& 3.30e-9  & 8.22e-9
& 1.11e10   & 2.34e12
& 106      & {$>$1682} \\

\bottomrule
\end{tabular}%
}
\end{subtable}

\medskip

\begin{subtable}{\textwidth}
\centering
\caption{$\alpha=\frac{1}{2}$}

\resizebox{\textwidth}{!}{%
\begin{tabular}{
  r
  @{\hspace{2\tabcolsep}}
  *{4}{S[table-format=1.2e-2]}
  *{2}{S[table-format=1.2e+1]}
  *{2}{c}
}
\toprule
& \multicolumn{2}{c}{Prec. $(\nnn,\ttt)$-CFIE$_\alpha$}
& \multicolumn{2}{c}{$(\nnn,\ttt)$-CFIE$_\alpha$}
& \multicolumn{2}{c}{Condition number}
& \multicolumn{2}{c}{GMRES iterations} \\
\cmidrule(lr){2-3}
\cmidrule(lr){4-5}
\cmidrule(lr){6-7}
\cmidrule(lr){8-9}
$N$
& \multicolumn{1}{c}{$\lVert\nabla e_p\rVert$}
& \multicolumn{1}{c}{$\lVert\operatorname{curl} e_s\rVert$}
& \multicolumn{1}{c}{$\lVert\nabla e_p\rVert$}
& \multicolumn{1}{c}{$\lVert\operatorname{curl} e_s\rVert$}
& \multicolumn{1}{c}{Prec.}
& \multicolumn{1}{c}{Orig.}
& \multicolumn{1}{c}{Prec.}
& \multicolumn{1}{c}{Orig.} \\
\midrule
32
& 7.62e-4  & 2.94e-3
& 7.62e-4  & 2.94e-3
& 1.93e6   & 5.59e6
& 92      & 128 \\

64
& 3.35e-4  & 1.01e-3
& 3.35e-4  & 1.01e-3
& 3.25e6   & 3.10e7
& 96      & 256 \\

128
& 5.10e-5  & 1.56e-4
& 5.10e-5  & 1.56e-4
& 2.61e7   & 1.08e9
& 111      & 492 \\

256
& 3.58e-9  & 1.17e-8
& 3.58e-9  & 1.22e-8
& 2.04e8   & 5.31e10
& 117      & {$>$960} \\

512
& 3.33e-9  & 9.77e-9
& 3.30e-9  & 8.22e-9
& 6.91e9   & 2.38e12
& 129      & {$>$1786} \\

\bottomrule
\end{tabular}%
}
\end{subtable}

\medskip

\begin{subtable}{\textwidth}
\centering
\caption{$\alpha=-\frac{1}{2}$}

\resizebox{\textwidth}{!}{%
\begin{tabular}{
  r
  @{\hspace{2\tabcolsep}}
  *{4}{S[table-format=1.2e-2]}
  *{2}{S[table-format=1.2e+1]}
  *{2}{c}
}
\toprule
& \multicolumn{2}{c}{Prec. $(\nnn,\ttt)$-CFIE$_\alpha$}
& \multicolumn{2}{c}{$(\nnn,\ttt)$-CFIE$_\alpha$}
& \multicolumn{2}{c}{Condition number}
& \multicolumn{2}{c}{GMRES iterations} \\
\cmidrule(lr){2-3}
\cmidrule(lr){4-5}
\cmidrule(lr){6-7}
\cmidrule(lr){8-9}
$N$
& \multicolumn{1}{c}{$\lVert\nabla e_p\rVert$}
& \multicolumn{1}{c}{$\lVert\operatorname{curl} e_s\rVert$}
& \multicolumn{1}{c}{$\lVert\nabla e_p\rVert$}
& \multicolumn{1}{c}{$\lVert\operatorname{curl} e_s\rVert$}
& \multicolumn{1}{c}{Prec.}
& \multicolumn{1}{c}{Orig.}
& \multicolumn{1}{c}{Prec.}
& \multicolumn{1}{c}{Orig.} \\
\midrule
32
& 8.17e-4  & 2.06e-3
& 8.17e-4  & 2.06e-3
& 9.83e6   & 7.37e6
& 73      & 128 \\

64
& 3.35e-4  & 1.00e-3
& 3.35e-4  & 1.00e-3
& 1.63e7   & 3.93e7
& 79      & 256 \\

128
& 5.10e-5  & 1.56e-4
& 5.10e-5  & 1.56e-4
& 1.45e8   & 1.32e9
& 93      & 484 \\

256
& 3.59e-9  & 1.22e-8
& 3.58e-9  & 1.22e-8
& 8.41e8   & 6.11e10
& 99      & {$>$946} \\

512
& 2.62e-9  & 6.57e-9
& 3.30e-9  & 8.22e-9
& 1.31e10   & 2.59e12
& 108      & {$>$1684} \\

\bottomrule
\end{tabular}%
}
\end{subtable}
\end{table}

\begin{table}[htbp]
\centering

\caption{Errors, matrix condition numbers, and number of GMRES iterations for the cavity experiment ($\lambda_1 = 2$, $\mu_1 = 3$, $\lambda_2 = 3$, $\mu_2 = 2$, $\rho = 2$) with $\omega=64$. The coupling parameter $\alpha$ in the formulation takes the values $1, -1, 1/2, -1/2$. $N$ is the number of quadrature points per density; the size of the linear system is $4N \times 4N$. GMRES is run without restart with a relative residual tolerance of $10^{-10}$ for both systems. In each block, ``Prec.'' refers to the preconditioned $(\nnn,\ttt)$-CFIE$_\alpha$ and ``Orig.'' to the $(\nnn,\ttt)$-CFIE$_\alpha$ itself.}
\label{tab:cavity-omega-64}

\setlength{\tabcolsep}{3.5pt}
\renewcommand{\arraystretch}{1.15}

\begin{subtable}{\textwidth}
\centering
\caption{$\alpha=1$}

\resizebox{\textwidth}{!}{%
\begin{tabular}{
  r
  @{\hspace{2\tabcolsep}}
  *{4}{S[table-format=1.2e-2,retain-zero-exponent=true]}
  *{2}{S[table-format=1.2e+1]}
  *{2}{c}
}
\toprule
& \multicolumn{2}{c}{Prec. $(\nnn,\ttt)$-CFIE$_\alpha$}
& \multicolumn{2}{c}{$(\nnn,\ttt)$-CFIE$_\alpha$}
& \multicolumn{2}{c}{Condition number}
& \multicolumn{2}{c}{GMRES iterations} \\
\cmidrule(lr){2-3}
\cmidrule(lr){4-5}
\cmidrule(lr){6-7}
\cmidrule(lr){8-9}
$N$
& \multicolumn{1}{c}{$\lVert\nabla e_p\rVert$}
& \multicolumn{1}{c}{$\lVert\operatorname{curl} e_s\rVert$}
& \multicolumn{1}{c}{$\lVert\nabla e_p\rVert$}
& \multicolumn{1}{c}{$\lVert\operatorname{curl} e_s\rVert$}
& \multicolumn{1}{c}{Prec.}
& \multicolumn{1}{c}{Orig.}
& \multicolumn{1}{c}{Prec.}
& \multicolumn{1}{c}{Orig.} \\
\midrule
64
& 1.50e-1  & 3.24e-1
& 1.50e-1  & 3.24e-1
& 4.52e5   & 6.51e6
& 256      & 256 \\

128
& 8.54e-3  & 5.72e-2
& 8.54e-3  & 5.72e-2
& 3.54e5   & 3.61e7
& 444      & 512 \\

256
& 8.98e-9  & 9.51e-8
& 8.98e-9  & 9.51e-8
& 6.59e5   & 1.65e9
& 462      & 1002 \\

512
& 6.67e-11  & 1.12e-10
& 6.65e-11  & 1.12e-10
& 2.50e6   & 3.32e10
& 459      & 1722 \\

1024
& 6.51e-11  & 1.14e-10
& 6.65e-11  & 1.12e-10
& 3.90e7   & 7.91e11
& 465      & {$>$2744} \\

\bottomrule
\end{tabular}%
}
\end{subtable}

\medskip

\begin{subtable}{\textwidth}
\centering
\caption{$\alpha=-1$}

\resizebox{\textwidth}{!}{%
\begin{tabular}{
  r
  @{\hspace{2\tabcolsep}}
  *{4}{S[table-format=1.2e-2,retain-zero-exponent=true]}
  *{2}{S[table-format=1.2e+1]}
  *{2}{c}
}
\toprule
& \multicolumn{2}{c}{Prec. $(\nnn,\ttt)$-CFIE$_\alpha$}
& \multicolumn{2}{c}{$(\nnn,\ttt)$-CFIE$_\alpha$}
& \multicolumn{2}{c}{Condition number}
& \multicolumn{2}{c}{GMRES iterations} \\
\cmidrule(lr){2-3}
\cmidrule(lr){4-5}
\cmidrule(lr){6-7}
\cmidrule(lr){8-9}
$N$
& \multicolumn{1}{c}{$\lVert\nabla e_p\rVert$}
& \multicolumn{1}{c}{$\lVert\operatorname{curl} e_s\rVert$}
& \multicolumn{1}{c}{$\lVert\nabla e_p\rVert$}
& \multicolumn{1}{c}{$\lVert\operatorname{curl} e_s\rVert$}
& \multicolumn{1}{c}{Prec.}
& \multicolumn{1}{c}{Orig.}
& \multicolumn{1}{c}{Prec.}
& \multicolumn{1}{c}{Orig.} \\
\midrule
64
& 3.85e-1  & 8.49e-1
& 3.85e-1  & 8.49e-1
& 4.57e6   & 3.14e7
& 208      & 256 \\

128
& 1.15e-2  & 9.04e-2
& 1.15e-2  & 9.04e-2
& 5.02e5   & 4.20e7
& 225      & 512 \\

256
& 1.34e-8  & 7.02e-8
& 1.34e-8  & 7.02e-8
& 2.10e6   & 2.35e9
& 240      & 1001 \\

512
& 6.51e-11  & 1.12e-10
& 6.64e-11  & 1.12e-10
& 7.01e6   & 1.78e10
& 234      & 1694 \\

1024
& 6.63e-11  & 1.24e-10
& 6.64e-11  & 1.12e-10
& 7.70e7   & 8.57e11
& 247      & {$>$2716} \\

\bottomrule
\end{tabular}%
}
\end{subtable}

\medskip

\begin{subtable}{\textwidth}
\centering
\caption{$\alpha=\frac{1}{2}$}

\resizebox{\textwidth}{!}{%
\begin{tabular}{
  r
  @{\hspace{2\tabcolsep}}
  *{4}{S[table-format=1.2e-2,retain-zero-exponent=true]}
  *{2}{S[table-format=1.2e+1]}
  *{2}{c}
}
\toprule
& \multicolumn{2}{c}{Prec. $(\nnn,\ttt)$-CFIE$_\alpha$}
& \multicolumn{2}{c}{$(\nnn,\ttt)$-CFIE$_\alpha$}
& \multicolumn{2}{c}{Condition number}
& \multicolumn{2}{c}{GMRES iterations} \\
\cmidrule(lr){2-3}
\cmidrule(lr){4-5}
\cmidrule(lr){6-7}
\cmidrule(lr){8-9}
$N$
& \multicolumn{1}{c}{$\lVert\nabla e_p\rVert$}
& \multicolumn{1}{c}{$\lVert\operatorname{curl} e_s\rVert$}
& \multicolumn{1}{c}{$\lVert\nabla e_p\rVert$}
& \multicolumn{1}{c}{$\lVert\operatorname{curl} e_s\rVert$}
& \multicolumn{1}{c}{Prec.}
& \multicolumn{1}{c}{Orig.}
& \multicolumn{1}{c}{Prec.}
& \multicolumn{1}{c}{Orig.} \\
\midrule
64
& 1.41e-1  & 2.98e-1
& 1.41e-1  & 2.98e-1
& 4.03e5   & 5.67e6
& 256      & 256 \\

128
& 3.40e-2  & 2.18e-1
& 3.40e-2  & 2.18e-1
& 6.43e5   & 1.10e8
& 446      & 512 \\

256
& 9.15e-9  & 6.17e-8
& 9.15e-9  & 6.17e-8
& 8.98e5   & 2.26e9
& 459      & 1008 \\

512
& 6.62e-11  & 1.12e-10
& 6.65e-11  & 1.12e-10
& 3.24e6   & 3.32e10
& 456      & 1724 \\

1024
& 6.60e-11  & 1.14e-10
& 6.65e-11  & 1.12e-10
& 4.68e7   & 8.96e11
& 462      & 2748 \\

\bottomrule
\end{tabular}%
}
\end{subtable}

\medskip

\begin{subtable}{\textwidth}
\centering
\caption{$\alpha=-\frac{1}{2}$}

\resizebox{\textwidth}{!}{%
\begin{tabular}{
  r
  @{\hspace{2\tabcolsep}}
  *{4}{S[table-format=1.2e-2,retain-zero-exponent=true]}
  *{2}{S[table-format=1.2e+1]}
  *{2}{c}
}
\toprule
& \multicolumn{2}{c}{Prec. $(\nnn,\ttt)$-CFIE$_\alpha$}
& \multicolumn{2}{c}{$(\nnn,\ttt)$-CFIE$_\alpha$}
& \multicolumn{2}{c}{Condition number}
& \multicolumn{2}{c}{GMRES iterations} \\
\cmidrule(lr){2-3}
\cmidrule(lr){4-5}
\cmidrule(lr){6-7}
\cmidrule(lr){8-9}
$N$
& \multicolumn{1}{c}{$\lVert\nabla e_p\rVert$}
& \multicolumn{1}{c}{$\lVert\operatorname{curl} e_s\rVert$}
& \multicolumn{1}{c}{$\lVert\nabla e_p\rVert$}
& \multicolumn{1}{c}{$\lVert\operatorname{curl} e_s\rVert$}
& \multicolumn{1}{c}{Prec.}
& \multicolumn{1}{c}{Orig.}
& \multicolumn{1}{c}{Prec.}
& \multicolumn{1}{c}{Orig.} \\
\midrule
64
& 5.86e-1  & 3.23e0
& 5.86e-1  & 3.23e0
& 6.23e6   & 6.65e7
& 199      & 256 \\

128
& 1.03e-2  & 6.97e-2
& 1.03e-2  & 6.97e-2
& 5.46e5   & 3.90e7
& 192      & 512 \\

256
& 1.12e-8  & 4.12e-8
& 1.12e-8  & 4.12e-8
& 2.41e6   & 1.32e9
& 205      & 1001 \\

512
& 6.39e-11  & 1.12e-10
& 6.64e-11  & 1.12e-10
& 8.24e6   & 2.23e10
& 200      & 1690 \\

1024
& 7.58e-11  & 1.19e-10
& 6.62e-11  & 1.12e-10
& 8.69e7   & 9.33e11
& 213      & {$>$2716} \\

\bottomrule
\end{tabular}%
}
\end{subtable}

\end{table}

\subsection{Spectral properties}\label{sec:spectrum}

Figures~\ref{fig:eigenvalues:ellipse}--\ref{fig:svd:cavity} display, for the ellipse and the cavity at $\omega=64$ with $\alpha=-1/2$, the eigenvalues and the singular values of the original and preconditioned matrices on fine meshes ($4N=4096$ and $4N=8192$, respectively). The eigenvalues of the original matrices spread over regions of size $10^{5}$--$10^{8}$ in the complex plane, whereas those of the preconditioned matrices cluster at $1$: for the ellipse essentially all of them lie within $0.03$ of $z=1$, and for the cavity a dense cluster within $0.1$ of $z=1$ coexists with a few hundred eigenvalues of moderate size. The singular values tell the same story. Those of the original matrices decay smoothly over nine (ellipse) and fourteen (cavity) orders of magnitude; those of the preconditioned matrices are equal to $1$ to within a few percent, except for two short tails: a few tens of values between $0.03$ and $30$ for the ellipse, and a few hundred values up to $10^{4}$--$10^{5}$ together with a few tens down to $10^{-2.5}$ for the cavity. These tails are the outliers responsible for the condition numbers of the previous section; the bulk of the spectrum behaves as the identity in both cases, so that the preconditioner confines the ill-conditioning to a low-dimensional subspace. Finally, Figure~\ref{fig:GMRES_convergence_history} shows the GMRES residual histories: the preconditioned curves are superimposed for all meshes that resolve the solution, while the unpreconditioned ones stretch linearly with $N$.
\begin{figure}[htbp]
\caption{Eigenvalue distribution in the complex plane for the elliptical domain with $\alpha = -1/2$, $\omega = 64$, $\lambda_1 = 2$, $\mu_1 = 3$, $\lambda_2 = 3$, $\mu_2 = 2$, $\rho = 2$, and $N = 1024$ ($4\,096 \times 4\,096$ matrix): (a) full spectrum; (b), (c) successive close-up views around $z = 1$. In the preconditioned formulation the eigenvalues cluster around $1$.}
    \label{fig:eigenvalues:ellipse}
    \centering
    \begin{subfigure}[b]{0.66\textwidth}
        \centering
        \includegraphics[width=\linewidth]{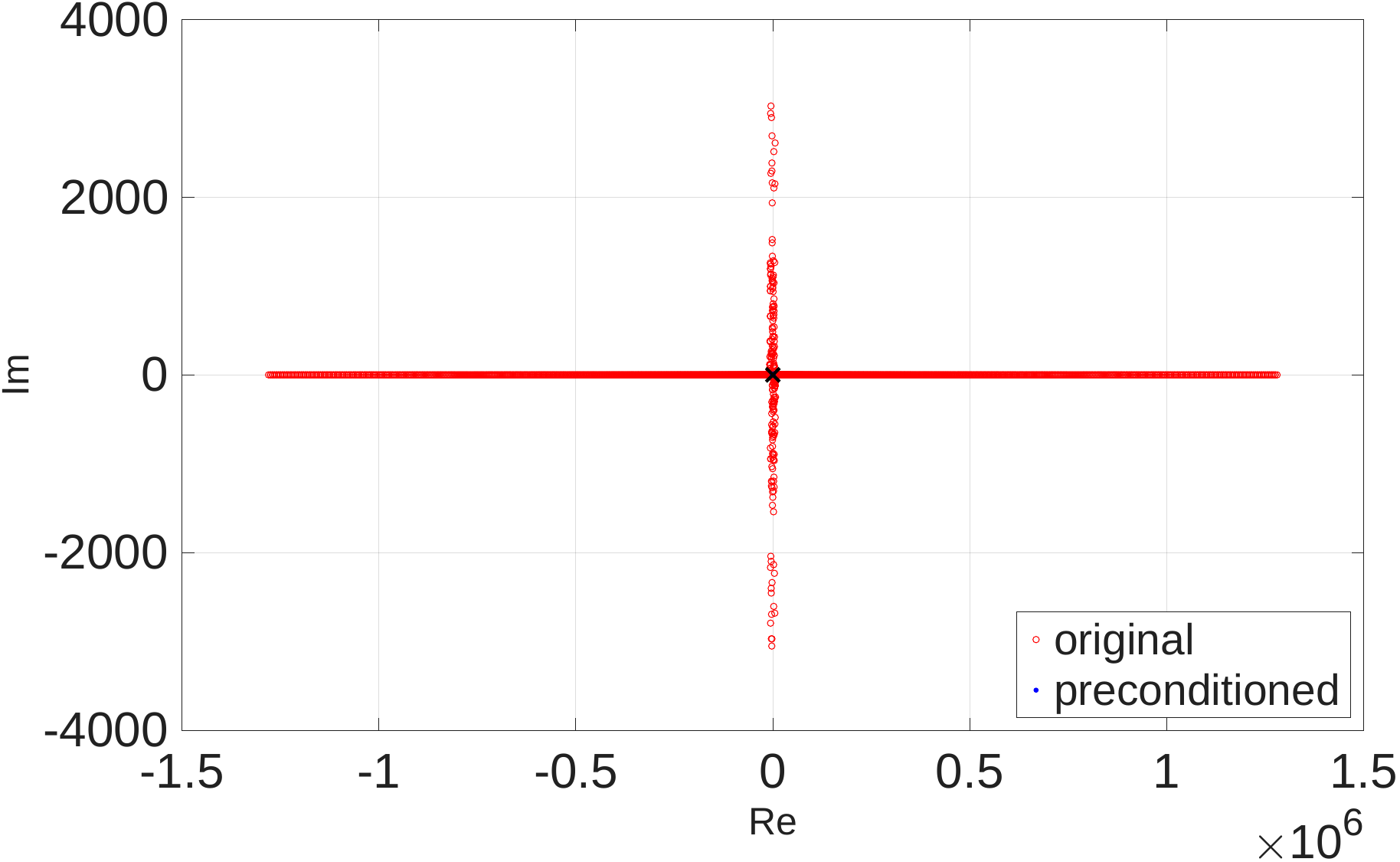}
        \caption{Full spectrum in the complex plane.}
        \label{subfig:spec_full}
    \end{subfigure}
    \vspace{0.8em}
    \begin{subfigure}[b]{0.495\textwidth}
        \centering
        \includegraphics[width=\linewidth]{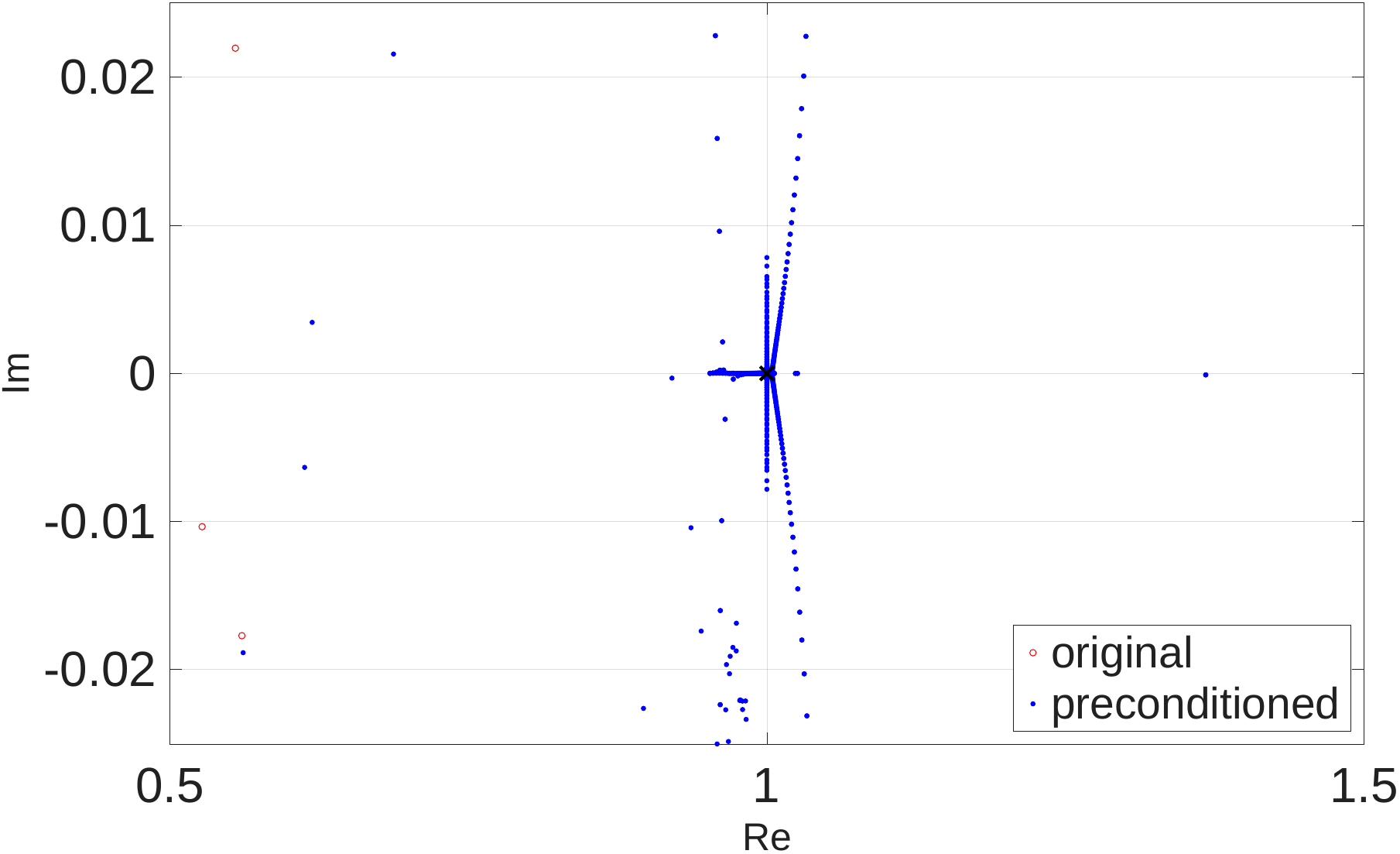}
        \caption{Zoom-in around $z = 1$.}
        \label{subfig:spec_zoom1}
    \end{subfigure}
    \hfill
    \begin{subfigure}[b]{0.47\textwidth}
        \centering
        \includegraphics[width=\linewidth]{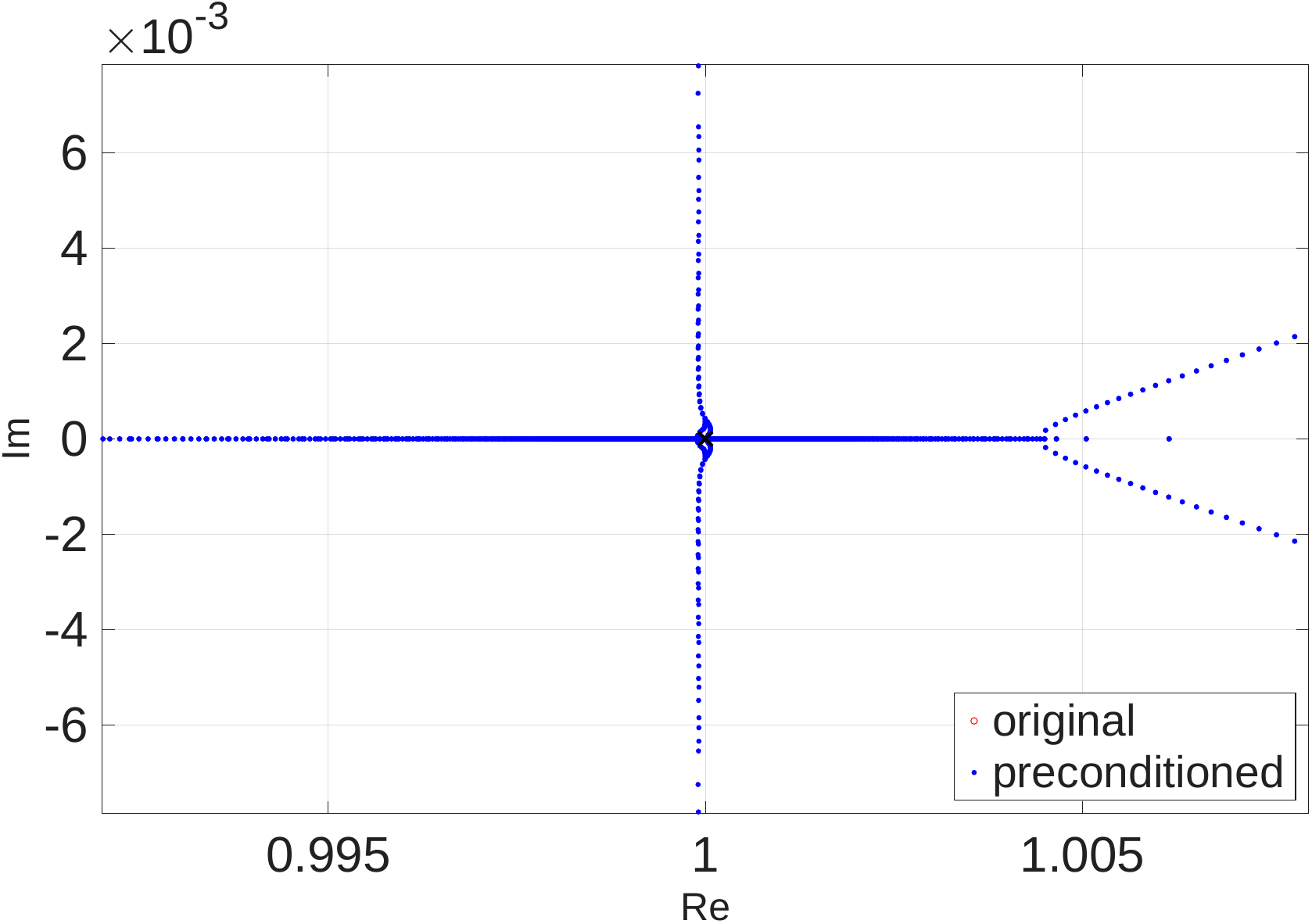}
        \caption{Further close-up around $z = 1$.}
        \label{subfig:spec_zoom2}
    \end{subfigure}
\end{figure}

\begin{figure}[htbp]
\caption{Eigenvalue distribution in the complex plane for the cavity domain with $\alpha = -1/2$, $\omega = 64$, $\lambda_1 = 2$, $\mu_1 = 3$, $\lambda_2 = 3$, $\mu_2 = 2$, $\rho = 2$, and $N = 2048$ ($8\,192 \times 8\,192$ matrix): (a) full spectrum; (b), (c) successive close-up views around $z = 1$. In the preconditioned formulation the eigenvalues cluster around $1$.}
    \label{fig:eigenvalues:cavity}
    \centering
    \begin{subfigure}[b]{0.66\textwidth}
        \centering
        \includegraphics[width=\linewidth]{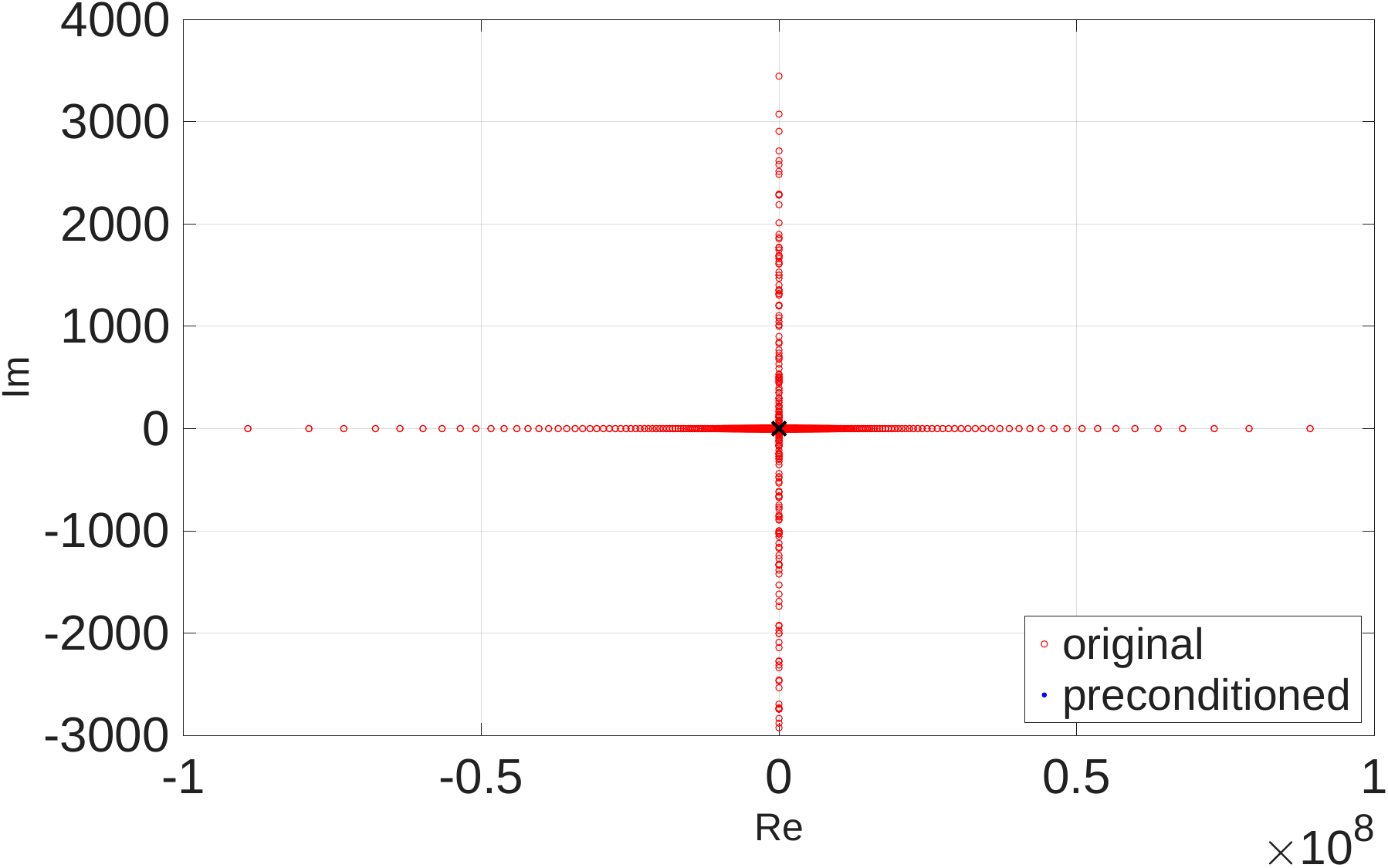}
        \caption{Full spectrum in the complex plane.}
        \label{subfig:spec_full:cavity}
    \end{subfigure}
    \vspace{0.8em}
    \begin{subfigure}[b]{0.48\textwidth}
        \centering
        \includegraphics[width=0.8\linewidth]{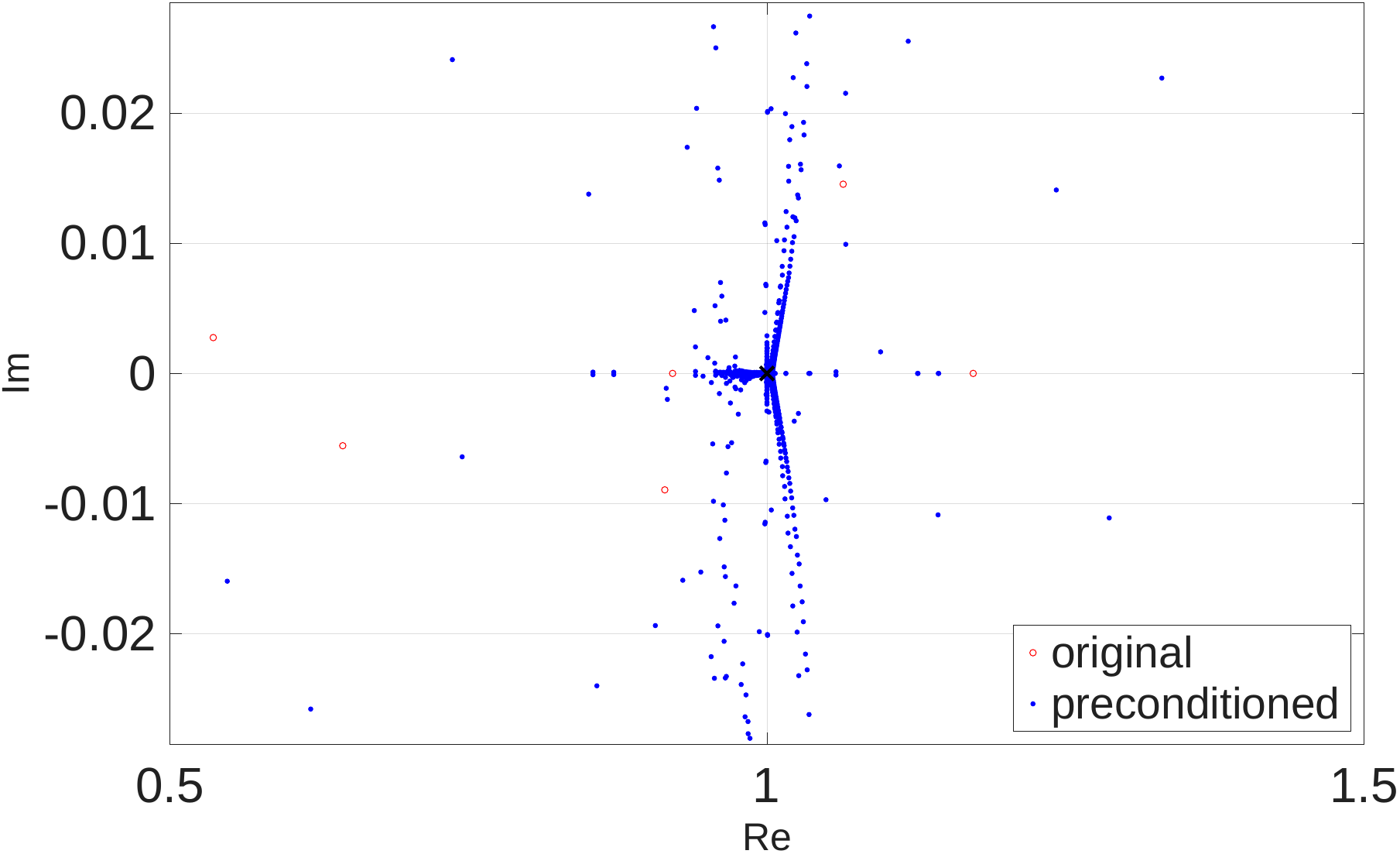}
        \caption{Zoom-in around $z = 1$.}
        \label{subfig:spec_zoom1:cavity}
    \end{subfigure}
    \hfill
    \begin{subfigure}[b]{0.48\textwidth}
        \centering
        \includegraphics[width=0.8\linewidth]{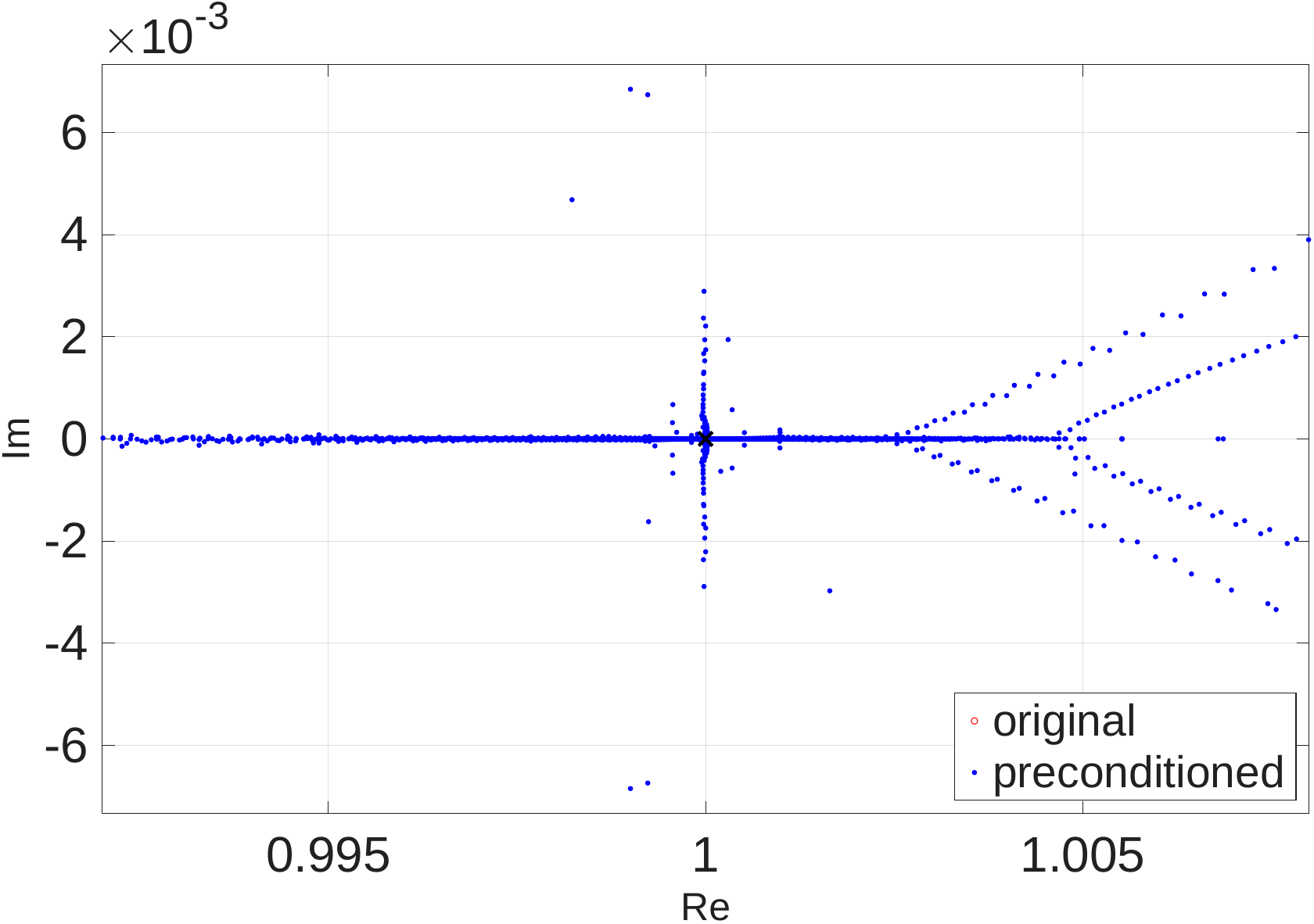}
        \caption{Further close-up around $z = 1$.}
        \label{subfig:spec_zoom2:cavity}
    \end{subfigure}
\end{figure}

\begin{figure}[htbp]
\caption{Singular values of the original (top) and preconditioned (bottom) matrices for the ellipse experiment ($\lambda_1 = 2$, $\mu_1 = 3$, $\lambda_2 = 3$, $\mu_2 = 2$, $\rho = 2$) with $\omega = 64$, $\alpha = -1/2$ and $N=1024$ ($4\,096 \times 4\,096$ matrix).}
\label{fig:svd:ellipse}
\begin{center}
        \includegraphics[width=0.8\linewidth]{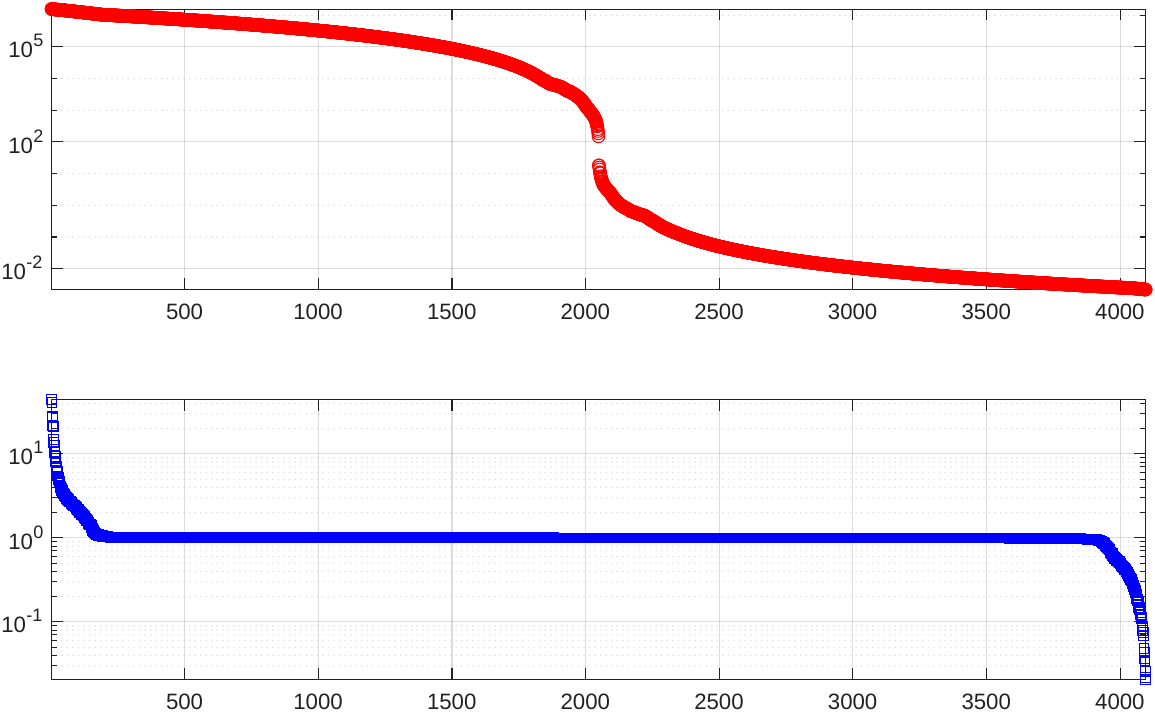}
\end{center}
\end{figure}

\begin{figure}[htbp]
\caption{Singular values of the original (top) and preconditioned (bottom) matrices for the cavity experiment ($\lambda_1 = 2$, $\mu_1 = 3$, $\lambda_2 = 3$, $\mu_2 = 2$, $\rho = 2$) with $\omega = 64$, $\alpha = -1/2$ and $N=2048$ ($8\,192 \times 8\,192$ matrix).}
\label{fig:svd:cavity}
\begin{center}
        \includegraphics[width=0.8\linewidth]{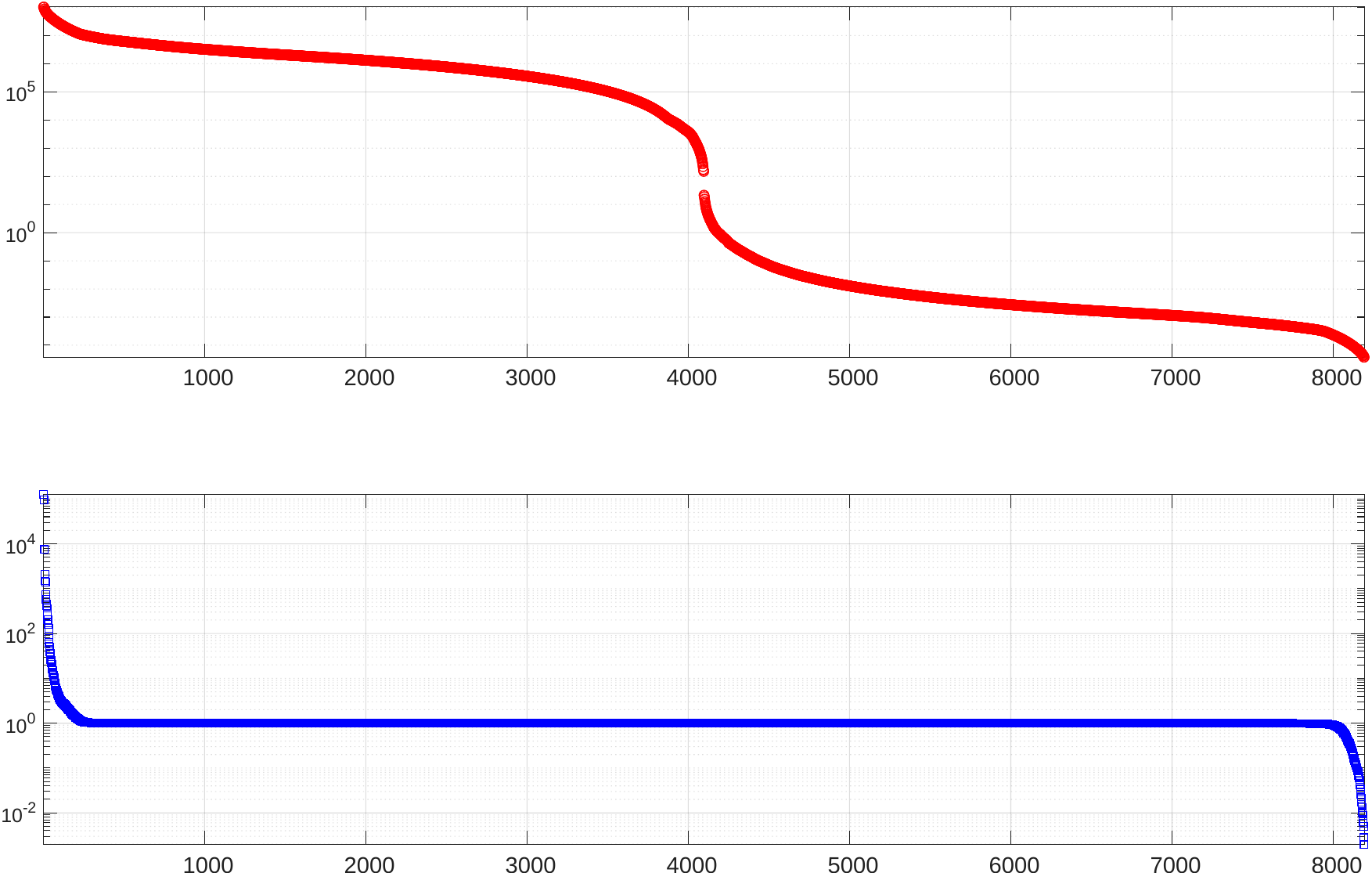}
\end{center}
\end{figure}

\begin{figure}[htbp]
   \caption{
GMRES convergence history (relative residual vs. iteration) for the original and preconditioned systems, ellipse (top) and cavity (bottom), $\omega=64$, $\alpha=-1/2$, for the sequence of meshes of Tables~\ref{tab:ellipse-omega-64} and~\ref{tab:cavity-omega-64}. The preconditioned system converges in a number of iterations independent of the mesh, about 100 for the ellipse and 200 for the cavity, whereas for the original system it grows linearly with $N$.
}  
    \label{fig:GMRES_convergence_history}
\begin{center}
 \includegraphics[width = 0.85\textwidth]{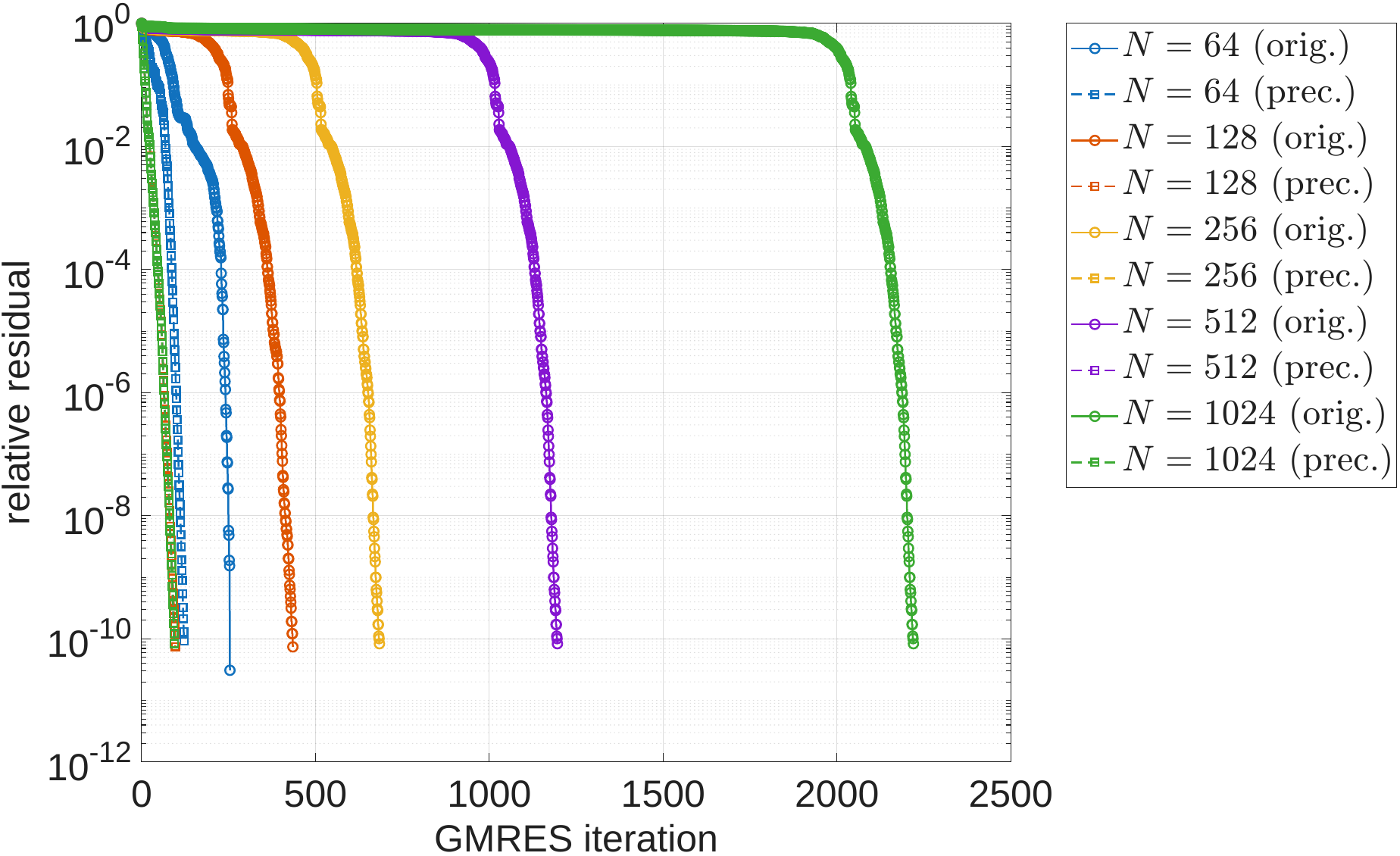}
 \includegraphics[width = 0.85\textwidth]{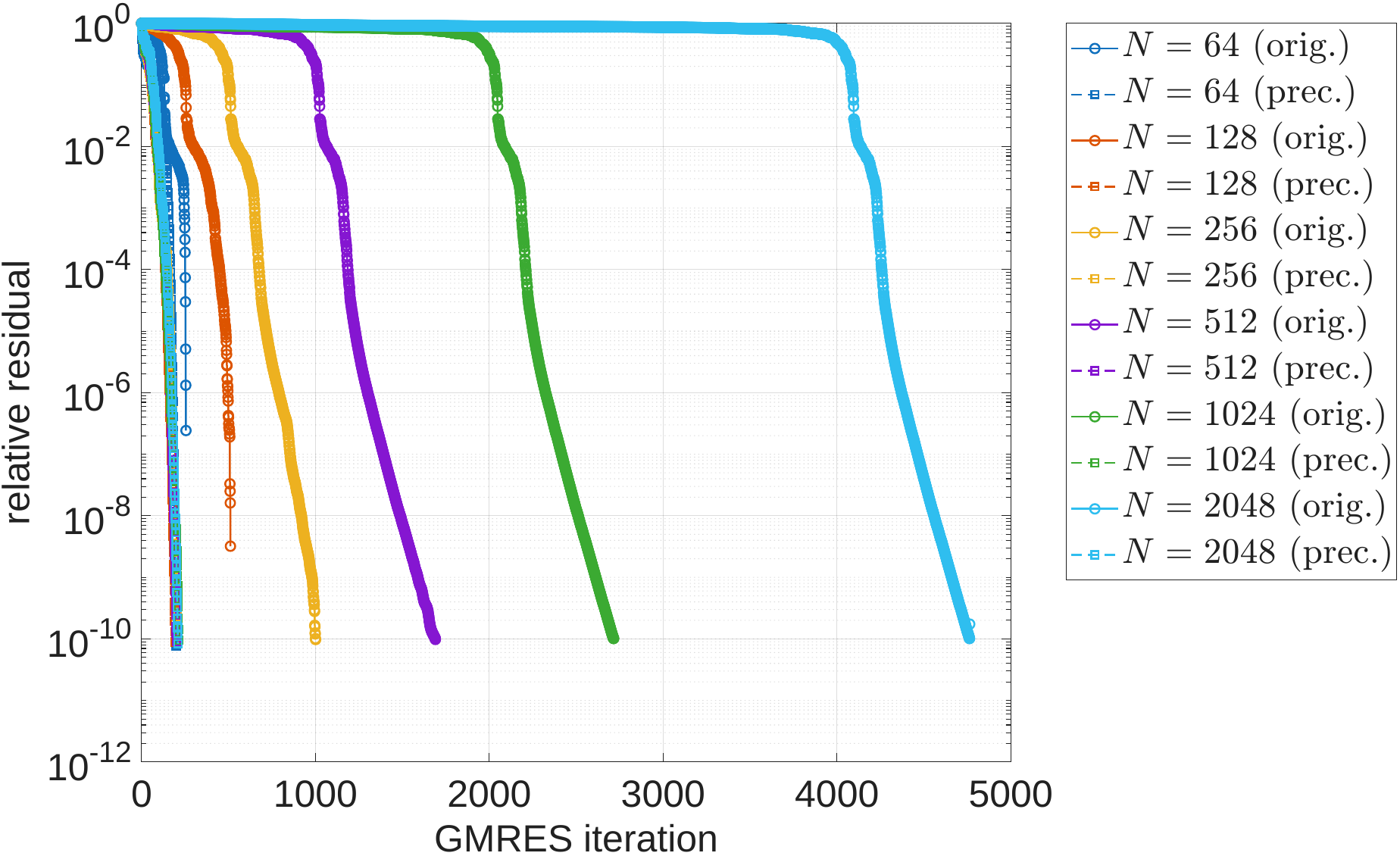}
\end{center}
\end{figure}

\clearpage

\section{Conclusions}\label{sec:conclusions}

We have presented a boundary integral framework for two-dimensional time-harmonic elastic transmission problems in which, through the Helmholtz decomposition, only the four classical Helmholtz BIOs are needed. The key ingredient is the rewriting of the Navier traction operator in the local normal--tangential frame of the interface, which eliminates second-order normal derivatives in favour of tangential derivatives and curvature terms. The resulting $(\nnn,\ttt)$-CFIE$_\alpha$ is uniquely solvable for every $\alpha\neq0$ and every admissible choice of material parameters, but it is defective: its principal part has infinite-dimensional kernel and cokernel, and the discretized systems are severely ill-conditioned. We have constructed explicit left and right analytical preconditioners, built from Hilbert transforms, tangential derivatives and the curvature, and proved that the preconditioned operator is a compact perturbation of the identity.

The numerical experiments confirm the analysis. For the circle the preconditioned matrices have mesh-independent condition numbers and GMRES converges in a fixed number of iterations. For the ellipse and for a smooth non-convex cavity the number of iterations becomes independent of the mesh and is reduced by one order of magnitude or more, also at high frequency, although for the cavity a few outlying singular values prevent the condition number from being uniformly bounded. Negative values of $\alpha$, in particular $\alpha=-1/2$, are consistently the most efficient.

Open questions include the dependence of the convergence on $\alpha$, the control of the outlying discrete singular values, and the extension to domains with corners, to multiple junctions via multi-trace formulations and to three dimensions, where the local-frame rewriting of the traction operator must be revisited.
\section*{Acknowledgments}
 Víctor Domínguez is partially supported by project ``Adquisición de conocimiento y minería de datos,
funciones especiales y métodos numéricos avanzados'' from Universidad Pública de Navarra and  ``Técnicas innovadoras para la resolución de problemas evolutivos'', ref.
PID2022-136441NB-I00, from the Ministerio de Ciencia e Innovación, Gobierno de España.

Catalin Turc gratefully acknowledges support from NSF through contract DMS-2408635.

\bibliography{transmission}

\end{document}